%% file: main.tex
\documentclass[a4paper]{amsart}
\usepackage[T1]{fontenc}
\usepackage{amstext}
\usepackage{amsthm}
\usepackage{amssymb}
\usepackage{stackrel}
\usepackage[pdfusetitle,
 bookmarks=true,bookmarksnumbered=false,bookmarksopen=false,
 breaklinks=false,pdfborder={0 0 0},pdfborderstyle={},backref=false,colorlinks=false]
 {hyperref}

\makeatletter

\numberwithin{equation}{section}
\numberwithin{figure}{section}

\usepackage{tikz-cd}
\usepackage{mathtools}
\usepackage{adjustbox}
\usepackage{enumitem}
\tikzcdset{scale cd/.style={every label/.append style={scale=#1},
    cells={nodes={scale=#1}}}}
\usepackage{quiver}
\usepackage{xcolor}
\usepackage{eucal}
\usepackage{tikz}
\usetikzlibrary{arrows.meta, positioning, shapes, fit}
\newtheorem{innercustomthm}{Theorem}
\newenvironment{customthm}[1]
  {\renewcommand\theinnercustomthm{#1}\innercustomthm}
  {\endinnercustomthm}
\DeclareMathSymbol{:}{\mathpunct}{operators}{"3A}
\usepackage{mleftright}
\mleftright
\usepackage[font=footnotesize]{caption}

\makeatother

\theoremstyle{plain}
\newtheorem{thm}{\protect\theoremname}[section]
\theoremstyle{definition}
\newtheorem{defn}[thm]{\protect\definitionname}
\theoremstyle{plain}
\newtheorem{prop}[thm]{\protect\propositionname}
\theoremstyle{remark}
\newtheorem{rem}[thm]{\protect\remarkname}
\theoremstyle{definition}
\newtheorem{example}[thm]{\protect\examplename}
\theoremstyle{plain}
\newtheorem{lem}[thm]{\protect\lemmaname}
\theoremstyle{definition}
\newtheorem{notation}[thm]{\protect\notationname}
\newtheorem{construction}[thm]{\protect\constructionname}
\theoremstyle{plain}
\newtheorem{cor}[thm]{\protect\corollaryname}
\theoremstyle{definition}
\newtheorem{recollection}[thm]{\protect\recollectionname}

\providecommand{\constructionname}{Construction}
\providecommand{\corollaryname}{Corollary}
\providecommand{\definitionname}{Definition}
\providecommand{\examplename}{Example}
\providecommand{\lemmaname}{Lemma}
\providecommand{\notationname}{Notation}
\providecommand{\propositionname}{Proposition}
\providecommand{\recollectionname}{Recollection}
\providecommand{\remarkname}{Remark}
\providecommand{\theoremname}{Theorem}

\begin{document}
\include{macros}

\vspace*{-1.5cm}
\title{On the equivalence of two approaches to multiplicative homotopy theories}
\email{arakawa.kensuke.22c@st.kyoto-u.ac.jp}
\address{Department of Mathematics, Kyoto University, Kyoto, 606-8502, Japan}
\author{Kensuke Arakawa}
\begin{abstract}
We study the relation of two frameworks for multiplicative homotopy
theories: Presentably symmetric monoidal $\infty$-categories and
combinatorial symmetric monoidal model categories. Our main theorem
establishes an equivalence of their homotopy theories.

As consequences, we solve Pavlov's conjecture and obtain a solution
to a special case of Hovey's 10th problem. We also prove several variations
of the main theorem, such as an analog for non-symmetric monoidal
semi-model categories. 
\end{abstract}

\maketitle
A recurring theme in homotopy theory is that the homotopy theory of
homotopy theories---possibly with additional structures---is equivalent
to the homotopy theory of 1-categories with a distinguished subcategory
of ``weak equivalences,'' again possibly with additional structures.
A primary example is a result of Barwick and Kan \cite{BK12b}, which
asserts that the homotopy theory of $\infty$-categories is equivalent
to that of relative categories. Additional results of this kind are
summarized in the table below. These results have profound implications,
providing not only philosophical clarity but also an essential link
between  the concrete and the abstract. The importance and the non-triviality
of this theme are also reflected in the existence of several unresolved
conjectures, such as \cite[Problems 8 and 9]{Hoveylist_model}, \cite[Conjecture 3.7]{KL18},
and \cite[Conjecture 1.10]{Pav25}.

{\footnotesize

\begin{table}[h]
\begin{tabular}{|p{4.6cm}|p{4.6cm}|p{1.6cm}|}
\hline

Abstract theory &
Concrete model &
Reference \\
\hline

$\infty$-Categories &
Relative categories &
\cite{BK12b} \\[1mm]

Finitely cocomplete $\infty$-categories &
Cofibration categories &  
\cite{Szu16}  \\[1mm]

$\infty$-Groupoids &
Categories &
\cite{Tho80} \\[1mm]

Presentable $\infty$-categories &
Combinatorial model categories &
\cite{Pav25}\\[1mm]

Monoidal relative categories &
Monoidal $\infty$-categories &
\cite{A25a}\\[1mm]

Relative opeards &
$\infty$-Operads &
\cite{ACP25}\\[1mm]

Presentably symmetric monoidal $\infty$-categories &
Combinatorial symmetric monoidal model categories &
\textbf{This work}\\

\hline
\end{tabular}
\end{table}

}

In this paper, we contribute a new result to this theme by establishing
an equivalence of the homotopy theories of two frameworks for higher
algebra. The first is the abstract theory of\textit{ presentably symmetric
monoidal $\infty$-categories}. These are \textit{presentable $\infty$-categories}
with a closed \textit{symmetric monoidal structure}. The second is
the concrete theory of \textit{combinatorial symmetric monoidal model
categories}. These are \textit{combinatorial model categories} with
a compatible \textit{symmetric monoidal structure}. In the literature,
the two frameworks stand out as two of the most popular settings for
higher algebraic structures, such as operads, operadic algebras, and
enrichment \cite{CH20,Haug19,Haug15,GH15}. Our main result asserts
the equivalence of the homotopy theories of these frameworks:

\begin{customthm}{A}[Theorems \ref{thm:main} and \ref{thm:var_SM}]\label{thm:main-inro}

The $\infty$-category of presentably symmetric monoidal $\infty$-categories
is the localization of the category of combinatorial symmetric monoidal
model categories and symmetric monoidal left Quillen functors. 

Moreover, the maximal sub $\infty$-groupoid of the former is the
localization of the subcategory of the latter spanned by symmetric
monoidal left Quillen equivalences.

\end{customthm}

In the absence of the symmetric monoidal structure, the first half
of Theorem \ref{thm:main-inro} was established recently by Pavlov
\cite{Pav25}, building on earlier work of Dugger and Lurie \cite{Dug01b, HTT}.
Pavlov also conjectured the first half of Theorem \ref{thm:main-inro};
given its naturality, the conjecture arises quite readily. The theorem
is significantly more subtle in the presence of the multiplicative
structure, as many techniques used to rigidify $\infty$-categories
break down multiplicatively. A concrete manifestation of this difficulty
is that even basic structural questions remain open. For instance,
it is not known whether every monoidal model category is monoidally
Quillen equivalent to a simplicial one, which is the content of Hovey\textquoteright s
10th problem \cite[Problem 10]{Hoveylist_model}. Partial progress
was made by Nikolaus and Sagave, who showed that every presentably
symmetric monoidal $\infty$-category arises from a combinatorial
symmetric monoidal model category \cite{NS17}. However, this does
not provide an equivalence of homotopy theories, nor is such an equivalence
clear from their argument. 

The second half of Theorem \ref{thm:main-inro} gives the \textit{existence
and essential uniqueness} of a combinatorial symmetric monoidal model
category modeling a given presentably symmetric monoidal $\infty$-category
(such as spectra with the smash product, and $\infty$-operads with
the Boardman--Vogt tensor product). This solves Hovey's 10th problem
in the special case of combinatorial symmetric monoidal model categories,
showing that every such category is symmetric-monoidally Quillen equivalent
to a simplicial one. As a consequence, we can extend results on combinatorial
symmetric monoidal model categories that have only been available
for simplicial ones, such as \cite{HB23}. (See also Remark \ref{rem:Hovey}
for the non-symmetric case.)

Beyond these applications, Theorem \ref{thm:main-inro} provides an
effective way to work with presentably symmetric monoidal $\infty$-categories.
In essence, it allows us to reduce general questions about presentably
symmetric monoidal $\infty$-categories to corresponding statements
about combinatorial symmetric monoidal model categories. This reduction
often makes otherwise intractable constructions accessible, due to
the extensive literature documenting the relation between constructions
in symmetric monoidal model categories and those in their underlying
$\infty$-categories---for instance, for operads \cite{PS18,CH20,HB23,batanin2023modelstructuresoperadsalgebras,carmona_env},
for operadic algebras \cite{Whi17,PS18,WY18,Whi22,BatWhi22,batanin2023modelstructuresoperadsalgebras,Whi24,carmona_env},
and for enrichment \cite{GH15, Haug15}. In the forthcoming work \cite{arakawa_BCHcomparison},
we apply this strategy to establish a new equivalence of two models
of enriched $\infty$-operads. The use of Theorem \ref{thm:main-inro}
is crucial, and we believe that the comparison would be very difficult
without it. There should also be plenty of other applications of this
nature.

Our proof of Theorem \ref{thm:main-inro} follows the general strategy
of Pavlov, but requires substantial new input. Pavlov\textquoteright s
argument relies on two key ingredients: The description of $\infty$-categories
via relative categories due to Barwick--Kan \cite{BK12b}, and the
theory of presentations of combinatorial model categories due to Dugger
and Low \cite{Dug01b, Low16}. To extend Pavlov's argument, we develop
symmetric monoidal analogs of these results. The main new tools are
the theory of symmetric monoidal relative categories developed earlier
by the author \cite{A25a}, and a new model structure on symmetric
cubical sets (Subsection \ref{subsec:sym_cube_set}).

\subsection*{Organization of the paper}

This paper has 6 sections in the main body and 4 sections in the appendix. 

In Section \ref{sec:definition}, we recall the definition of monoidal
model categories and their variations. In Section \ref{sec:Main-Result},
we state the main theorem, and then show that it follows from three
separate propositions. The proof of these propositions will be given
in the next three sections (Sections \ref{sec:main_L}, \ref{sec:main_flat},
and \ref{sec:main_U}). In Section \ref{sec:variation}, we give variations
of the main result of this paper, such as for non-symmetric presentably
monoidal $\infty$-categories and semi-model categories.

Appendices contain miscellaneous results that will be used in the
main body of the paper. In Appendix \ref{sec:semi}, we give an overview
of semi-model categories. In Appendix \ref{sec:MGUD}, we review Gabriel--Ulmer
duality and extend it multiplicatively. In Appendix \ref{sec:lb},
we discuss the compatibility of localization and base change. In Appendix
\ref{sec:Muro}, we explain how to replace enriched monoidal model
categories by ones with cofibrant unit.

\subsection*{Notation and convention}
\begin{itemize}
\item In addition to the ZFC axioms, we will assume the existence of three
nested Grothendieck universes whose elements are called small sets,
large sets, and very large sets.\footnote{Although Pavlov did not make this assumption in his paper \cite{Pav25},
it seems extremely inconvenient to drop this assumption due to the
nature of the statements we will prove. For example, the statement
of the main theorem itself needs to be adjusted without this.} All locally presentable categories are assumed to be large but not
very large, so that the collection of locally presentable categories
themselves form a very large set.
\item We assume that model categories are at most large. Also, we only require
the existence of finite limits and finite colimits for model categories.
If $\mathbf{M}$ is a model category, we write $\mathbf{M}_{\cof}\subset\mathbf{M}$
for the full subcategory of cofibrant objects. We also write $\mathbf{M}_{\infty}$
for the underlying $\infty$-category of $\mathbf{M}$ (i.e., localization
at weak equivalences).
\item We use the term ``regular cardinal'' to mean ``small regular cardinal.''
\item By $\infty$-categories, we mean \textit{quasicategories} in the sense
of \cite{Joyal_qcat_Kan,HTT}. We mostly follow the terminology and
notation of \cite{HTT}.
\item If $\cat C$ is a presentable $\infty$-category and $\kappa$ is
a regular cardinal, we write $\cat C_{\kappa}\subset\cat C$ for the
full subcategory spanned by the $\kappa$-compact objects of $\cat C$.
\item We let $\Fin_{\ast}$ denote the category of finite pointed sets $\inp n=\pr{\{\ast,1,\dots,n\},\ast}$
for $n\geq0$ and pointed maps.
\item We let $\SM\Cat_{\infty}$ denote the $\infty$-category of small
symmetric monoidal $\infty$-categories (defined as the underlying
$\infty$-category of the model category described in \cite[Variant 2.1.4.13]{HA}).
We write $\SM\hat{\Cat}_{\infty}$ for the $\infty$-category of large
symmetric monoidal $\infty$-categories.
\item We write $\Pr\SM\subset\SM\hat{\Cat}_{\infty}$ for the subcategory
spanned by the presentably symmetric monoidal $\infty$-categories
and symmetric monoidal functors that preserve small colimits. We define
the $\infty$-category $\Pr\Mon$ of presentably monoidal $\infty$-categories
similarly.
\item We will not notationally distinguish between ordinary categories and
their nerves. We will also regard every $\pr{2,1}$-category (i.e.,
categories enriched over groupoids) as an $\infty$-category by taking
its Duskin nerve. (Recall that this converts $\pr{2,1}$-categories
into $\infty$-categories, and it is functorial in strictly unitary
pseudofunctors \cite[\href{https://kerodon.net/tag/00AU}{Tag 00AU}]{kerodon}.)
We also do not notationally distinguish between (symmetric) monoidal
categories and the associated (symmetric) monoidal $\infty$-categories
(realized as a cocartesian fibration over $\Fin_{\ast}$ or $\Del^{\op}$).
\item Given a pair of symmetric $\infty$-monoidal categories $\cat C,\cat D$,
we let $\Fun^{\t}\pr{\cat C,\cat D}$ denote the $\infty$-category
of symmetric monoidal functors $\cat C\to\cat D$. 
\end{itemize}

\section{\label{sec:definition}Monoidal model categories and their variations}

The use of the term ``monoidal model category'' and its variations
is not entirely standardized in the literature. The goal of this section
is to record the precise definitions we will use.

\subsection{Plain case}

We start with the definition of monoidal model categories. There are
several competing definitions in the literature: Most definitions
require the pushout-product axiom, but they often differ in their
requirement on how ``flat'' the unit object should be. In some literature
(e.g., \cite{HA}), the unit object is required to be cofibrant. This
assumption is convenient for theoretical purposes but ends up excluding
many interesting examples, such as symmetric spectra. At the other
extreme (e.g., \cite{PS18}), no condition on the unit object is assumed.
In this paper, we go for a middle ground by adopting Muro's unit axiom,
which is satisfied by most monoidal model categories we encounter.
(The author is not aware of a counterexample.)
\begin{defn}
\label{def:Quillen_bifunctor}Let $\mathbf{A},\mathbf{B},\mathbf{C}$
be model categories. A functor $F\from\mathbf{A}\times\mathbf{B}\to\mathbf{C}$
is called a \textbf{left Quillen bifunctor} if it satisfies the following
pair of conditions:
\begin{enumerate}
\item The functor $F$ preserves small colimits in each variable.
\item For every pair of cofibrations $f\from A\to A'$ in $\mathbf{A}$
and $g\from B\to B'$ in $\mathbf{B}$, the map
\[
F\pr{A,B'}\amalg_{F\pr{A,B}}F\pr{A',B}\to F\pr{A,B}
\]
is a cofibration, and moreover it is a weak equivalence if $f$ or
$g$ is one.
\end{enumerate}
\end{defn}

\begin{defn}
\label{def:MMC}\hfill
\begin{itemize}
\item A\textbf{ monoidal model category} is a biclosed monoidal category
$\pr{\mathbf{M},\otimes,\mathbf{1}}$ equipped with a model structure,
satisfying the following axioms:
\begin{enumerate}
\item (\textbf{Pushout-product axiom}) The functor $\otimes\from\mathbf{M}\times\mathbf{M}\to\mathbf{M}$
is a left Quillen bifunctor (Definition \ref{def:Quillen_bifunctor}).
\item (\textbf{Muro's unit axiom} \cite{Mur15}) There is a weak equivalence
$q\from\widetilde{\mathbf{1}}\to\mathbf{1}$, with $\widetilde{\mathbf{1}}$
cofibrant, such that for every object $X\in\mathbf{M}$, the maps
$X\otimes q$ and $q\otimes X$ are weak equivalences.
\end{enumerate}
\item We say that a monoidal model category is \textbf{combinatorial} if
its underlying model category is combinatorial. If further it admits
generating sets of cofibrations and trivial cofibrations whose elements
have a cofibrant domain, it is called \textbf{tractable}.
\item If $\mathbf{M}$ and $\mathbf{N}$ are monoidal model categories,
then a \textbf{monoidal left Quillen functor} is a monoidal functor\footnote{By a monoidal functor, we mean a strong monoidal functor in the sense
of \cite[Chapter XI]{CWM}.} $\mathbf{M}\to\mathbf{N}$ whose underlying functor of model categories
is left Quillen. We will write $\Fun^{\t,LQ}\pr{\mathbf{M},\mathbf{N}}$
for the category of monoidal left Quillen functors $\mathbf{M}\to\mathbf{N}$
and monoidal natural transformations between them.
\item We will write $\CMMC$ for the (very large) category of combinatorial
monoidal model categories and monoidal left Quillen functors, and
write $\CMMC^{\mathbf{1}}\subset\CMMC$ for the full subcategory spanned
by the combinatorial monoidal model categories with a cofibrant unit. 
\item \textbf{Symmetric monoidal model categories} and \textbf{symmetric
monoidal left Quillen functors} are defined similarly. (We will use
the notation $\Fun^{\t,LQ}\pr{-,-}$ in the symmetric monoidal case,
too. This is abusive but is rarely confusing.) The categories $\CSMMC$
and $\CSMMC^{\mathbf{1}}$ are also defined similarly, using combinatorial
symmetric monoidal model categories.
\item We define categories $\TMMC,\TSMMC^{\mathbf{1}},$ etc, similarly,
using tractable symmetric and non-symmetric monoidal model categories.
\end{itemize}
\end{defn}

The reason why we adopt Muro's unit axiom is that it allows us to
turn every monoidal model category into one with a cofibrant unit.
More precisely, we have the following theorem.
\begin{thm}
\cite[Theorem 1, Proposition 12]{Mur15}\label{thm:Muro} The inclusion
$\CMMC^{\mathbf{1}}\subset\CMMC$ admits a left adjoint. Moreover,
the unit of this adjunction preserves and reflects weak equivalences
and induces an isomorphism of underlying monoidal categories. A similar
claim holds for combinatorial symmetric monoidal model categories.
\end{thm}

The categories $\CMMC$ can be upgraded to a $\pr{2,1}$-category
$\CMMC_{\pr{2,1}}$, whose mapping groupoid are given by the maximal
subgroupoid $\Fun^{\t,LQ}\pr{\mathbf{M},\mathbf{N}}^{\cong}\subset\Fun^{\t,LQ}\pr{\mathbf{M},\mathbf{N}}$.
We will use similar notations for the categories introduced in Definition
\ref{def:MMC}. These $\infty$-categories satisfy the following universal
property:
\begin{prop}
\label{prop:(2.1)}The functors 
\begin{align*}
\CMMC & \to\CMMC_{\pr{2,1}},\\
\CMMC^{{\bf 1}} & \to\CMMC^{\mathbf{1}}_{\pr{2,1}},\\
\CSMMC & \to\CSMMC_{\pr{2,1}},\\
\CSMMC^{{\bf 1}} & \to\CSMMC^{\mathbf{1}}_{\pr{2,1}}
\end{align*}
are $\infty$-categorical localizations at the left Quillen functors
whose underlying functors are equivalences of categories.
\end{prop}

\begin{proof}
We will prove the assertion for the first functor; the remaining assertions
can be proved in a similar way. According to (the dual of) \cite[Corollary 3.16]{ACK25},
it suffices to show that $\CMMC_{\pr{2,1}}$ admits weak cotensor
by $[1]$. That is, it suffices to show that for each $\mathbf{N}\in\CMMC_{\pr{2,1}}$,
there is an object $[1]\pitchfork\mathbf{N}$ admitting an isomorphism
of groupoids
\[
\Fun^{\t,LQ}\pr{\mathbf{M},[1]\pitchfork\mathbf{N}}^{\cong}\cong\Fun\pr{[1],\Fun^{\t,LQ}\pr{\mathbf{M},\mathbf{N}}^{\cong}}
\]
natural in $\mathbf{M}\in\CMMC$. For this, we consider the groupoid
$J$ with two objects $0,1$, and with exactly one morphism between
any pair of objects. We can make $\mathbf{N}^{J}$ into a monoidal
model category by using the degreewise tensor product and the model
structure coming from the equivalence of categories $\mathbf{N}\xrightarrow{\simeq}\mathbf{N}^{J}$.
Then $\mathbf{N}^{J}$ is a weak cotensor of $\mathbf{N}$ by $[1]$,
and we are done.
\end{proof}

\subsection{Enriched case}

We next turn to the definition of enriched monoidal model categories.
\begin{defn}
\label{def:V-MMC}Let $\mathbf{V}$ be a combinatorial symmetric monoidal
model category. 
\begin{itemize}
\item A \textbf{$\mathbf{V}$-monoidal model category} is a tensored and
cotensored $\mathbf{V}$-monoidal category $\mathbf{M}$ (in the sense
of \cite{Day_convolution}) which is biclosed (i.e., for each $X\in\mathbf{M}$,
the $\mathbf{V}$-functors $X\otimes-$ and $-\otimes X$ admit right
adjoints) and equipped with a model structure on its underlying category,
such that the tensor bifunctor
\[
\mathbf{V}\times\mathbf{M}\to\mathbf{M}
\]
is a left Quillen bifunctor. 
\item We say that a $\mathbf{V}$-monoidal model category is \textbf{combinatorial}
if its underlying model category is combinatorial.
\item If $\mathbf{M}$ and $\mathbf{N}$ are $\mathbf{V}$-monoidal model
categories, we write $\Fun^{\t,LQ}_{\mathbf{V}}\pr{\mathbf{M},\mathbf{N}}$
for the category of $\mathbf{V}$-monoidal functors $\mathbf{M}\to\mathbf{N}$
that are $\mathbf{V}$-cocontinuous and whose underlying functors
are left Quillen.
\item We write $\CMMC_{\mathbf{V},\pr{2,1}}$ for the $\pr{2,1}$-category
whose objects are the combinatorial $\mathbf{V}$-monoidal model categories,
with hom-groupoids given by the maximal subgroupoids of $\Fun^{\t,LQ}_{\mathbf{V}}\pr{-,-}$. 
\item We write $\CMMC_{\mathbf{V}}$ for the underlying category of $\CMMC_{\mathbf{V},\pr{2,1}}$.
\end{itemize}
We define $\mathbf{V}$-symmetric monoidal model categories similarly,
and define a $\pr{2,1}$-category $\CSMMC_{\mathbf{V},\pr{2,1}}$
and an ordinary category $\CSMMC_{\mathbf{V}}$ similarly. As in Definition
\ref{def:MMC}, we also consider full subcategories of these categories
such as $\TMMC_{\mathbf{V}}$ and $\TSMMC^{\mathbf{1}}_{\mathbf{V}}$.
\end{defn}

In ordinary algebra, a commutative algebra $A$ over a commutative
ring $k$ can be defined as a commutative ring equipped with a ring
homomorphism $k\to A$. Extending an analogy of this to the model-categorical
setting, we arrive at the notion of \textit{symmetric algebras}:
\begin{defn}
\label{def:symalg}Let $\mathbf{V}$ be a combinatorial symmetric
monoidal model category. 
\begin{itemize}
\item A \textbf{symmetric $\mathbf{V}$-algebra} is a monoidal model category
$\mathbf{M}$ equipped with a monoidal left Quillen functor $\mathbf{V}\to\mathbf{M}$.
\item We say that a symmetric $\mathbf{V}$-algebra is \textbf{combinatorial}
if its underlying model category is combinatorial.
\item We write $\CombSymAlg{\mathbf{V}}_{\pr{2,1}}=\int^{\mathbf{M}\in\CSMMC_{\pr{2,1}}}\Fun^{\t,LQ}\pr{\mathbf{V},\mathbf{M}}$
for the $\pr{2,1}$-category of combinatorial symmetric $\mathbf{V}$-algebras,
where $\int$ denotes the $2$-categorical Grothendieck construction
\cite[1.10]{Str80}. Explicitly:
\begin{itemize}
\item Objects are combinatorial symmetric $\mathbf{V}$-algebras.
\item A $1$-morphism $\pr{F\from\mathbf{V}\to\mathbf{M}}\to\pr{G\from\mathbf{V}\to\mathbf{N}}$
is a pair $\pr{H,\alpha}$, where $H\from\mathbf{M}\to\mathbf{N}$
is a symmetric monoidal functor and $\alpha:HF\To G$ is a symmetric
monoidal natural transformation.
\item A $2$-morphism $\pr{H,\alpha}\To\pr{H',\alpha'}\from\pr{F\from\mathbf{V}\to\mathbf{M}}\to\pr{G\from\mathbf{V}\to\mathbf{N}}$
is a symmetric monoidal natural isomorphism $\beta\from H\stackrel{\cong}{\To}H'$
satisfying $\alpha'\circ\beta F=\alpha$:
\[\begin{tikzcd}
	& {\mathbf{V}} &&& {\mathbf{V}} \\
	&& {} & {} \\
	{\mathbf{M}} && {\mathbf{N}} & {\mathbf{M}} && {\mathbf{N}.}
	\arrow[""{name=0, anchor=center, inner sep=0}, "F"', from=1-2, to=3-1]
	\arrow[""{name=1, anchor=center, inner sep=0}, "G", from=1-2, to=3-3]
	\arrow[""{name=2, anchor=center, inner sep=0}, "F"', from=1-5, to=3-4]
	\arrow[""{name=3, anchor=center, inner sep=0}, "G", from=1-5, to=3-6]
	\arrow["{=}"{description}, draw=none, from=2-3, to=2-4]
	\arrow["H"', from=3-1, to=3-3]
	\arrow[""{name=4, anchor=center, inner sep=0}, "{H'}", curve={height=-18pt}, from=3-4, to=3-6]
	\arrow[""{name=5, anchor=center, inner sep=0}, "H"', curve={height=18pt}, from=3-4, to=3-6]
	\arrow["\alpha"{description}, between={0.2}{0.8}, Rightarrow, from=0, to=1]
	\arrow["{\alpha'}"{description}, between={0.2}{0.8}, Rightarrow, from=2, to=3]
	\arrow["\beta"{description}, between={0.2}{0.8}, Rightarrow, from=5, to=4]
\end{tikzcd}\]
\end{itemize}
\item We write $\CombSymAlg{\mathbf{V}}$ for the underlying category of
$\CombSymAlg{\mathbf{V}}_{\pr{2,1}}$.
\end{itemize}
\end{defn}

\begin{rem}
There is a related notion of \textit{central algebras} \cite[Definition 4.1.10]{Hovey},
but they are not amenable to the techniques we use in this paper.
\end{rem}

As in Proposition \ref{prop:(2.1)}, we have:
\begin{prop}
\label{prop:(2.1)_V}Let $\mathbf{V}$ be a combinatorial symmetric
monoidal model category. The functors
\begin{align*}
\CMMC_{\mathbf{V}} & \to\CMMC_{\mathbf{V},\pr{2,1}},\\
\CMMC^{\mathbf{1}}_{\mathbf{V}} & \to\CMMC^{\mathbf{1}}_{\mathbf{V},\pr{2,1}},\\
\CSMMC_{\mathbf{V}} & \to\CSMMC_{\mathbf{V},\pr{2,1}},\\
\CSMMC^{\mathbf{1}}_{\mathbf{V}} & \to\CSMMC^{\mathbf{1}}_{\mathbf{V},\pr{2,1}},\\
\CombSymAlg{\mathbf{V}} & \to\CombSymAlg{\mathbf{V}}_{\pr{2,1}}
\end{align*}
are all localizations. 
\end{prop}

\subsection{Semi-model-categorical case}

(Left) \textbf{Semi-model categories} are a weakening of model categories
in which part of the axioms of lifting and factorization are only
required to hold for cofibrations with \textit{cofibrant} source \cite[Definition 2.1]{BW24}.
While this leads to a less attractive set of axioms, semi-model categories
are practically indistinguishable from model categories. The main
advantages of semi-model categories over ordinary model categories
is that they are much easier to construct than a full model structure.
In particular, tractable semi-model categories are fully compatible
with (left) Bousfield localization \cite{BW24}. For these reasons,
semi-model categories is becoming increasingly popular, and we have
no reason not to include them in this paper.

For the basics of semi-model categories, we refer the reader to Section
\ref{sec:semi}. In this subsection, we collect definitions related
to the interplay of semi-model categories and monoidal structures.
\begin{defn}
A semi-model category is said to be \textbf{tractable} if it is locally
presentable as a category and admits generating sets of cofibrations
of cofibrant objects and trivial cofibrations of cofibrant objects.
\end{defn}

\begin{defn}
Let $\mathbf{M},\mathbf{N},\mathbf{P}$ be tractable semi-model categories.
A \textbf{left Quillen bifunctor} $F\from\mathbf{M}\times\mathbf{N}\to\mathbf{P}$
is a functor that preserves small colimits in each variable, and which
satisfies the pushout-product axiom for cofibrations of cofibrant
objects and trivial cofibrations of cofibrant objects.
\end{defn}

\begin{defn}
The definition of monoidal model categories carries over to the semi-model
categorical case. We define a category $\TMMC_{\semi}$ of tractable
monoidal semi-model categories, and define categories such as $\TMMC^{\mathbf{1}}_{\semi}$,
$\TSMMC_{\semi}$, and $\TSMMC^{\mathbf{1}}_{\semi}$ similarly, as
in Definition \ref{def:MMC}. Also, for a tractable symmetric monoidal
semi-model category $\mathbf{V}$, there is a notion of $\mathbf{V}$-monoidal
model categories and $\mathbf{V}$-symmetric monoidal model categories.
The categories $\TMMC^{\mathbf{1}}_{\mathbf{V},\semi}$ and $\TSMMC_{\mathbf{V},\semi}$,
etc, are defined exactly as in Definition \ref{def:V-MMC}.
\end{defn}

\begin{rem}
Theorem \ref{thm:Muro} and Proposition \ref{prop:(2.1)} remains
valid for semi-model categories, with the same proof. 
\end{rem}

\section{\label{sec:Main-Result}Main result}

In this section, we state the main result of this paper and explain
how we will prove it.

To state our main result, we need a bit of terminology and notation.
\begin{defn}
\cite[Definition 0.1]{A25a} A \textbf{symmetric monoidal relative
category} is a symmetric monoidal category $\mathbf{C}$ equipped
with a subcategory $\mathbf{W}\subset\mathbf{C}$ that contains all
isomorphisms and is stable under tensor products. Morphisms of $\mathbf{W}$
are called \textbf{weak equivalences}.

We write $\SMRelCat$ for the category of small symmetric monoidal
relative categories and symmetric monoidal relative functors that
preserve weak equivalences. We define the category $\SMRelCatlarge$
of \textit{large} symmetric monoidal relative categories similarly.
\end{defn}

\begin{example}
Let $\mathbf{M}$ be a symmetric monoidal model category. Weak equivalences
of $\mathbf{M}$ are generally not stable under tensor products, so
$\mathbf{M}$ is generally \textit{not} a symmetric monoidal relative
category. Nonetheless, its full subcategory $\mathbf{M}^{\flat}\subset\mathbf{M}$
of cofibrant objects and the objects isomorphic to the unit object
is a symmetric monoidal relative category.
\end{example}

For the next definition, we recall that localization of symmetric
monoidal relative categories determines a functor $L\from\SMRelCat\to\SM\Cat_{\infty}$
\cite[Notation 2.1]{A25a}.
\begin{defn}
\label{def:loc}We define a functor
\[
\pr -_{\infty}\from\CSMMC\to\Pr\SM
\]
as the codomain restriction of the composite
\[
\CSMMC\xrightarrow{\pr -^{\flat}}\SMRelCatlarge\xrightarrow{L}\SM\hat{\Cat}_{\infty}.
\]
(Note that this is well-defined by \cite[Proposition 1.3.4.22 and Corollary 1.3.4.26]{HA}.)
If $\mathbf{M}$ is a combinatorial symmetric monoidal model category,
we refer to $\mathbf{M}_{\infty}$ as its \textbf{underlying symmetric
monoidal $\infty$-category}.
\end{defn}

We now arrive at the statement of the main result.
\begin{thm}
\label{thm:main}The functor $\pr -_{\infty}$ of Definition \ref{def:loc}
is a localization at Quillen equivalences, i.e., it induces a categorical
equivalence
\[
\CSMMC[\Quilleq^{-1}]\xrightarrow{\simeq}\Pr\SM.
\]
\end{thm}

We will prove Theorem \ref{thm:main} as follows: Fix a left proper
tractable symmetric monoidal model category $\mathbf{S}$ with a cofibrant
unit, whose underlying symmetric monoidal $\infty$-category is equivalent
to the cartesian monoidal $\infty$-category $\mathcal{S}$ of $\infty$-groupoids.\footnote{Since $\mathcal{S}$ is the initial presentably symmetric monoidal
$\infty$-category \cite[Corollary 3.2.1.9, Example 4.8.1.20]{HA},
such an equivalence is unique up to a contractible space of choices
if it exists.} (One example is the cartesian model category of simplicial sets with
the Kan--Quillen model structure.) We also define a category $\SMRelCatlarge_{\Pr\SM}$
by the pullback
\[\begin{tikzcd}
	{\mathsf{SMRel}\hat{\mathsf{Cat}}_{\mathcal{P}\mathsf{r}\mathcal{SM}}} & {\mathsf{SMRel}\hat{\mathsf{Cat}}} \\
	{\mathcal{P}\mathsf{r}\mathcal{SM}} & {\mathcal{SM}\hat{\mathcal{C}\mathsf{at}}_\infty.}
	\arrow[from=1-1, to=1-2]
	\arrow[from=1-1, to=2-1]
	\arrow["\lrcorner"{description, pos=0}, draw=none, from=1-1, to=2-2]
	\arrow["L", from=1-2, to=2-2]
	\arrow[from=2-1, to=2-2]
\end{tikzcd}\]We then contemplate the following commutative diagram 
\[\begin{tikzcd}
	{\mathsf{CombSym}\mathbf{S}\text{-}\mathsf{Alg}} & {\mathsf{CombSMMC}^{\mathbf{1}}} & {\mathsf{CombSMMC}} \\
	{\mathsf{SMRel}\hat{\mathsf{Cat}}_{\mathcal{P}\mathsf{r}\mathcal{SM}}} && {\mathcal{P}\mathsf{r}\mathcal{SM}\hat{\mathcal{C}\mathsf{at}}_\infty,}
	\arrow["U", from=1-1, to=1-2]
	\arrow["{(-)_\flat}"', from=1-1, to=2-1]
	\arrow["\iota", hook, from=1-2, to=1-3]
	\arrow["{(-)_\infty}", from=1-3, to=2-3]
	\arrow["L"', from=2-1, to=2-3]
\end{tikzcd}\]where $U$ is the forgetful functor and $\iota$ is the inclusion.
We will then prove the following propositions:
\begin{prop}
\label{prop:main_L}The functor $L\from\SMRelCatlarge_{\Pr\SM}\to\Pr\SM\hat{\Cat}_{\infty}$
induces a categorical equivalence
\[
\SMRelCatlarge_{\Pr}[\loceq^{-1}]\to\Pr\SM\hat{\Cat}_{\infty},
\]
where $\loceq$ denotes the subcategory of \textbf{local equivalences},
i.e., morphisms inverted by $L$.
\end{prop}

\begin{prop}
\label{prop:main_flat}The functor $\pr -_{\flat}$ induces a categorical
equivalence
\[
\CombSymAlg{\mathbf{S}}[\Quilleq^{-1}]\xrightarrow{\simeq}\SMRelCatlarge_{\Pr\SM}[\loceq^{-1}].
\]
\end{prop}

\begin{prop}
\label{prop:main_U}The functor $U$ induces a categorical equivalence
\[
\CombSymAlg{\mathbf{S}}[\Quilleq^{-1}]\xrightarrow{\simeq}\CSMMC^{\mathbf{1}}[\Quilleq^{-1}].
\]
\end{prop}

The proof of these propositions will be given in the next three sections
(Section \ref{sec:main_L}, \ref{sec:main_flat}, and \ref{sec:main_U}).
Since we know that $\iota$ induces a categorical equivalence upon
localizing at Quillen equivalences (Theorem \ref{thm:Muro}), these
proposition will prove Theorem \ref{thm:main}. 

\section{\label{sec:main_L}Proof of Proposition \ref{prop:main_L}}

In this section, we prove Proposition \ref{prop:main_L}, which asserts
that the functor
\[
\SMRelCatlarge_{\Pr\SM}[\loceq^{-1}]\to\Pr\SM
\]
is an equivalence. In fact, we will prove the following more general
assertion:
\begin{lem}
\label{lem:main_L}For every conservative functor $\mathcal{X}\to\SM\Cat_{\infty}$
of $\infty$-categories, the functor 
\[
\SMRelCat_{\mathcal{X}}=\mathcal{X}\times_{\SM\Cat_{\infty}}\SMRelCat\to\mathcal{X}
\]
is a localization.
\end{lem}

Replacing $\SM\Cat_{\infty}$ by $\SM\hat{\Cat}_{\infty}$ and then
substituting $\Pr\SM$ for $\mathcal{X}$, we obtain Proposition \ref{prop:main_L}. 

For the proof of Lemma \ref{lem:main_L}, we need some terminology
and notation.
\begin{itemize}
\item A \textbf{relative $\infty$-category} is an $\infty$-category equipped
with a subcategory $\cat W\subset\cat C$ containing all equivalences.
Morphisms in $\cat W$ are typically called \textbf{weak equivalences}.
When $\cat C$ is (the nerve of) an ordinary category, we call such
a pair a\textbf{ relative category}. This should not be confused with
the definition of Barwick--Kan \cite{BK12b}, which only requires
that $\mathcal{W}$ contains all objects.
\item If $\cat C$ and $\cat D$ are relative $\infty$-categories, a\textbf{
relative functor} $f\from\cat C\to\cat D$ is a functor between the
underlying $\infty$-categories that preserves weak equivalences.
We say that such an $f$ is a \textbf{homotopy equivalence} if there
is a relative functor $g\from\cat D\to\cat C$ and zig-zags of natural
weak equivalences connecting $f\circ g$ and $g\circ f$ to the identity
functors.
\item We write $\RelCat$ for the category of small relative categories
and functors preserving weak equivalences. For disambiguation, we
will write $\RelCat_{\BK}$ for the slightly larger category of Barwick--Kan's
relative categories that are small. 
\item We call a morphism $f:\mathcal{C}\to\mathcal{D}$ of $\RelCat_{\BK}$
a \textbf{local equivalence} if it induces an equivalence between
the ($\infty$-categorical) localizations. 
\item Recall that a symmetric monoidal category is called a \textbf{permutative
category} if its underlying monoidal category is strict (i.e., its
associators and the unitors are the identity maps). We write $\PermRelCat\subset\SMRelCat$
for the subcategory of objects whose underlying symmetric monoidal
categories are permutative categories, and morphisms whose underlying
symmetric monoidal functors are strict.
\item We will regard the categories $\RelCat,\RelCat_{\BK},\PermRelCat,\SMRelCat$
as relative categories whose weak equivalences are the local equivalences.
\item We write $\Fin_{\ast}$ for the category of finite pointed sets and
pointed maps. We write $\Fun^{\Seg}\pr{\Fin_{\ast},\RelCat}\subset\Fun^{\Seg}\pr{\Fin_{\ast},\RelCat}$
for the full subcategory spanned by the functors $F\from\Fin_{\ast}\to\RelCat$
satisfying the following \textbf{Segal condition}: For every $n\geq0$,
the map $\rho_{i}\from\inp n\to\inp 1$ defined by 
\[
\rho_{i}\pr j=\begin{cases}
\ast & \text{if }i\neq j,\\
1 & \text{if }i=j,
\end{cases}
\]
induces a local equivalence 
\[
F\inp n\xrightarrow{\simeq}\prod_{1\leq i\leq n}F\inp 1.
\]
We define $\Fun^{\Seg}\pr{\Fin_{\ast},\RelCat_{\BK}}\subset\Fun^{\Seg}\pr{\Fin_{\ast},\RelCat_{\BK}}$
similarly.
\end{itemize}
\begin{proof}
[Proof of Lemma \ref{lem:main_L}]Since symmetric monoidal categories
are functorially equivalent to permutative categories \cite[Proposition 4.2]{May78},
using Proposition \ref{prop:hoeq_pb}, we may replace $\SMRelCat$
by $\PermRelCat$. 

According to \cite[Theorem 1.1]{A25a}, there is a homotopy equivalence
$\Fact\from\PermRelCat\to\Fun^{\Seg}\pr{\Fin_{\ast},\RelCat}$ of
relative categories. The proof of loc. cit. further shows that the
localization of $\Fun^{\Seg}\pr{\Fin_{\ast},\RelCat}$ is equivalent
to $\SM\Cat_{\infty}$, and the diagram
\[\begin{tikzcd}
	{\mathsf{PermRelCat}} & {\operatorname{Fun}^{\mathrm{Seg}}(\mathsf{Fin}_\ast,\mathsf{RelCat})} \\
	{\mathsf{SMRelCat}} & {\mathcal{SM}\mathcal{C}\mathsf{at}_{\infty}.}
	\arrow["{\operatorname{Fact}}", from=1-1, to=1-2]
	\arrow[hook, from=1-1, to=2-1]
	\arrow["{\text{localization}}", from=1-2, to=2-2]
	\arrow["L"', from=2-1, to=2-2]
\end{tikzcd}\]commutes up to natural equivalence. Thus, by Proposition \ref{prop:hoeq_pb},
it will suffice to show that the functor
\[
\Fun^{\Seg}\pr{\Fin_{\ast},\RelCat}_{\mathcal{X}}\to\mathcal{X}
\]
is a localization. Since the inclusion $\RelCat\hookrightarrow\RelCat_{\BK}$
is a homotopy equivalence of relative categories, we can replace $\RelCat$
by $\RelCat_{\BK}$ by Proposition \ref{prop:hoeq_pb} again.

We now recall from \cite[Theorem 6.1]{BK12b} and \cite[Theorem 1.8]{BK12}
that $\RelCat_{\BK}$ and its local equivalences is part of a combinatorial
model structure. Therefore, Theorem \ref{thm:abs_loc_model_gen} and
\cite[Theorem 11.6.1]{Hirschhorn} show that the localization
\[
\Fun\pr{\Fin_{\ast},\RelCat_{\BK}}\to\Fun\pr{\Fin_{\ast},\RelCat_{\BK}}[\weq^{-1}]
\]
is stable under pullback. In particular, the functor
\[
\Fun^{\Seg}\pr{\Fin_{\ast},\RelCat_{\BK}}_{\cat X}\to\mathcal{X}
\]
is a localization, which was to be proved.
\end{proof}

\section{\label{sec:main_flat}Proof of Proposition \ref{prop:main_flat}}
\begin{notation}
Throughout this section, we fix a left proper tractable symmetric
monoidal model category $\mathbf{S}$ with a cofibrant unit, whose
underlying symmetric monoidal $\infty$-category is equivalent to
the cartesian monoidal $\infty$-category $\mathcal{S}$ of $\infty$-groupoids.
\end{notation}

In this section, we give a proof of Proposition \ref{prop:main_flat},
which asserts that the functor
\[
\CombSymAlg{\mathbf{S}}[\Quilleq^{-1}]\xrightarrow{\simeq}\SMRelCatlarge_{\Pr\SM}[\loceq^{-1}]
\]
is an equivalence. This will need two ingredients: A relative-categorical
version of the multiplicative Gabriel--Ulmer duality, and some formal
cardinality argument of model categories. We will tackle each of these
in the next two subsections, and then prove Proposition \ref{prop:main_flat}
in the final subsection.

\subsection{Relative-categorical multiplicative Gabriel--Ulmer duality}

In this subsection, we state and prove a relative-categorical multiplicative
Gabriel--Ulmer duality (Proposition \ref{prop:rel_GUdual}). We refer
the reader to Section \ref{sec:MGUD} for a review of Gabriel--Ulmer
duality. 

To state the main result, we need a bit of preliminaries.
\begin{notation}
Let $\kappa$ be an uncountable regular cardinal. We write $\SM\Cat_{\infty}\pr{\kappa}$
for the subcategory of $\SM\Cat_{\infty}$ spanned by the symmetric
monoidal $\infty$-categories that are compatible with $\kappa$-small
colimits (i.e., those $\infty$-categories that have $\kappa$-small
colimits and whose tensor product preserves such colimits in each
variable). 
\end{notation}

\begin{construction}
\label{const:MInd}Let $\kappa$ be a regular cardinal. For each $\mathcal{C}\in\SMRelCat_{\SM\Cat_{\infty}\pr{\kappa}}$,
we define an $\mathbf{S}$-algebra $\MInd^{\mathbf{S}}_{\kappa}\pr{L\pr{\mathcal{C}}}$
as follows: 
\begin{itemize}
\item Its underlying category is $\Fun\pr{\mathcal{C}^{\op},\mathbf{S}}$.
\item The symmetric monoidal structure comes from Day's convolution product.
\item The symmetric monoidal functor $\mathbf{S}\to\Fun\pr{\mathcal{C}^{\op},\mathbf{S}}$
is given by $X\mapsto\cat C\pr{-,\mathbf{1}}\cdot X$, where $\mathbf{1}\in\cat C$
denote the monoidal unit and the dot ``$\cdot$'' denotes copowering
by sets.
\item The model structure is the left Bousfield localization of the projective
model structure on $\Fun\pr{\mathcal{C}^{\op},\mathbf{S}}$ whose
weak equivalences are the maps inverted by the composite
\[
\Fun\pr{\mathcal{C}^{\op},\mathbf{S}}\to\Fun\pr{\mathcal{C}^{\op},\mathcal{S}}\to\Fun\pr{L\pr{\mathcal{C}}^{\op},\mathcal{S}}\to\Ind_{\kappa}\pr{L\pr{\mathcal{C}}}.
\]
Here, the first functor postcomposes with the localization $\mathbf{S}\to\mathcal{S}$,
the second functor is the left adjoint Kan extension along the localization
functor $\mathcal{C}^{\op}\to L\pr{\mathcal{C}}^{\op}$, and the third
functor is the left adjoint to the inclusion $\Ind_{\kappa}\pr{L\pr{\mathcal{C}}}\hookrightarrow\Fun\pr{L\pr{\mathcal{C}},\mathcal{S}}$. 
\end{itemize}
We sometimes use the notation $\MInd^{\mathbf{S}}_{\kappa}\pr{L\pr{\mathcal{C}}}$
for essentially small symmetric monoidal relative categories whose
symmetric monoidal localization is idempotent complete and is compatible
with $\kappa$-small colimits.
\end{construction}

\begin{lem}
\label{lem:MInd_welldefined}Construction \ref{const:MInd} is well-defined.
More precisely, in the situation of Construction \ref{const:MInd},
the following holds:
\begin{enumerate}
\item The model category $\MInd^{\mathbf{S}}_{\kappa}\pr{L\pr{\mathcal{C}}}$
exists.
\item The model category $\MInd^{\mathbf{S}}_{\kappa}\pr{L\pr{\mathcal{C}}}$
satisfies the pushout-product axiom, and the monoidal unit is cofibrant.
\item The symmetric monoidal functor $\mathbf{S}\to\MInd^{\mathbf{S}}_{\kappa}\pr{L\pr{\mathcal{C}}}$
is left Quillen.
\end{enumerate}
\end{lem}

\begin{proof}
Part (1) is a consequence of Proposition \ref{prop:bousfield}. 

Next, we prove part (2). Choose generating sets $I,J$ of cofibrations
and trivial cofibrations of $\mathbf{S}$. According to \cite[Theorem 11.6.1]{Hirschhorn},
the sets $I_{\mathcal{C}}=\{\mathcal{C}\pr{-,C}\cdot i\}_{i\in I,\,C\in\mathcal{C}}$
and $J_{\mathcal{C}}=\{\mathcal{C}\pr{-,C}\cdot j\}_{j\in J,\,C\in\mathcal{C}}$
generate cofibrations and trivial cofibrations of the projective model
structure on $\Fun\pr{\mathcal{C}^{\op},\mathbf{S}}$. It follows
immediately that:
\begin{itemize}
\item The monoidal unit of $\Fun\pr{\mathcal{C}^{\op},\mathbf{S}}$ is cofibrant
(because $\mathbf{S}$ has a cofibrant unit). 
\item The Day convolution product satisfies the pushout-product axiom for
the projective model structure.
\item The projective model structure is tractable.
\end{itemize}
Therefore, to prove (2), it will suffice to prove the following: For
every cofibrant object $X\in\MInd^{\mathbf{S}}_{\kappa}\pr{L\pr{\mathcal{C}}}$,
the functor 
\[
X\otimes-\from\MInd^{\mathbf{S}}_{\kappa}\pr{L\pr{\mathcal{C}}}_{\cof}\to\MInd^{\mathbf{S}}_{\kappa}\pr{L\pr{\mathcal{C}}}_{\cof}
\]
preserves weak equivalences. But this is clear, because the composite
\[
\Fun\pr{\mathcal{C}^{\op},\mathbf{S}}_{\cof}\to\Fun\pr{\mathcal{C}^{\op},\mathcal{S}}\to\Fun\pr{L\pr{\mathcal{C}}^{\op},\mathcal{S}}\to\Ind_{\kappa}\pr{L\pr{\mathcal{C}}}
\]
is symmetric monoidal.

Part (3) follows from the fact that its right adjoint is the evaluation
at the unit object, which is evidently right Quillen. The claim follows.
\end{proof}

\begin{defn}
\label{def:kappa_pres}Let $\mathcal{C}$ be a monoidal $\infty$-category,
and let $\kappa$ be an uncountable regular cardinal. We say that
$\mathcal{C}$ is \textbf{$\kappa$-presentably monoidal} if it satisfies
the following conditions:
\begin{enumerate}
\item The underlying $\infty$-category of $\mathcal{C}$ is $\kappa$-presentable.
\item The tensor product of $\cat C$ preserves small colimits in each variable.
\item For every $n\geq0$ and every collection of $\kappa$-presentable
objects $X_{1},\dots,X_{n}\in\mathcal{C}$, the object $X_{1}\otimes\cdots\otimes X_{n}$
is $\kappa$-presentable. (In particular, the unit object is $\kappa$-presentable).
\end{enumerate}
We write $\kappa\-\Pr\SM\subset\SM\hat{\Cat}_{\infty}$ for the subcategory
of $\kappa$-presentably monoidal $\infty$-categories and monoidal
functors that preserves small colimits and $\kappa$-presentable objects.
\end{defn}

\begin{notation}
For each replete subcategory\footnote{Recall that a subcategory $\cat C'\subset\cat C$ of an $\infty$-category
is called \textbf{replete} if every equivalence $f\from X\xrightarrow{\simeq}X'$
in $\cat C$ such that $X'\in\cat C'$ is a morphism of $\cat C'$.} $\mathcal{X}\subset\SM\Cat_{\infty}$, we will write $\SMRelCat_{\mathcal{X}}=\mathcal{X}\times_{\SM\Cat_{\infty}}\SMRelCat$,
and define $\SMRelCat_{\pr{2,1},\mathcal{X}}\subset\SMRelCat_{\pr{2,1}}$
for the sub $\pr{2,1}$-category whose mapping groupoids are the components
corresponding to the morphisms in $\SMRelCat_{\mathcal{X}}$.
\end{notation}

We can now state the main result of this subsection.
\begin{prop}
\label{prop:rel_GUdual}Let $\kappa$ be an uncountable regular cardinal.
The strictly unitary pseudofunctor
\begin{align*}
\SMRelCat_{\SM\Cat_{\infty}\pr{\kappa}} & \to\SMRelCatlarge_{\kappa\-\Pr\SM,\pr{2,1}}\\
\mathcal{C} & \mapsto\MInd^{\mathbf{S}}_{\kappa}\pr{L\pr{\mathcal{C}}}_{\flat}
\end{align*}
induces a categorical equivalence
\[
\SMRelCat_{\SM\Cat_{\infty}\pr{\kappa}}[\loceq^{-1}]\xrightarrow{\simeq}\SMRelCatlarge_{\kappa\-\Pr\SM,\pr{2,1}}[\loceq^{-1}].
\]
\end{prop}

We need a lemma for the proof of Proposition \ref{prop:rel_GUdual}.
We will use the following notation: For each replete subcategory $\mathcal{X}\subset\SM\Cat_{\infty}$,
we will write $\SMRelCat_{\pr{2,1},\mathcal{X}}\subset\SMRelCat_{\pr{2,1}}$
for the sub $\pr{2,1}$-category whose mapping groupoids are the union
of the components of the mapping groupoids of $\SMRelCat_{\pr{2,1}}$
corresponding to the morphisms in $\SMRelCat_{\mathcal{X}}$. We then
have:
\begin{lem}
\label{lem:SMRelCatX(2.1)loc}For every replete subcategory $\mathcal{X}\subset\SM\Cat_{\infty}$,
the inclusion $\SMRelCat_{\mathcal{X}}\hookrightarrow\SMRelCat_{\pr{2,1},\mathcal{X}}$
is a localization. 
\end{lem}

\begin{proof}
According to \cite[Corollary 3.15]{ACK25}, it suffices to show that
the $\pr{2,1}$-category $\SMRelCat_{\pr{2,1},\mathcal{X}}$ admits
a weak cotensor by $[1]$. Given an object $\mathcal{C}\in\SMRelCat_{\pr{2,1}}$,
a weak cotensor by $[1]$ is given by $\mathcal{C}^{J}$, where the
tensor product is given by the degreewise tensor product, and the
weak equivalences are the natural transformations whose components
are weak equivalences.
\end{proof}

We can now prove Proposition \ref{prop:rel_GUdual}.
\begin{proof}
[Proof of Proposition \ref{prop:rel_GUdual}]The Yoneda embedding
determines a pseudonatural transformation depicted as 
\[\begin{tikzcd}
	{\mathsf{SMRelCat}_{\mathcal{SM}\mathcal{C}\mathsf{at}_\infty^{}(\kappa)}} && {\mathsf{SMRel}\widehat{\mathsf{Cat}}_{\kappa\text{-}\mathcal{P}\mathsf{r}\mathcal{SM},(2,1)}} \\
	& {\mathsf{SMRel}\widehat{\mathsf{Cat}}_{(2,1)}}
	\arrow["{\mathbf{Ind}^{\mathbf{S}}_\kappa(L(-))_\flat}", from=1-1, to=1-3]
	\arrow[""{name=0, anchor=center, inner sep=0}, hook, from=1-1, to=2-2]
	\arrow[""{name=1, anchor=center, inner sep=0}, from=1-3, to=2-2]
	\arrow[between={0.2}{0.8}, Rightarrow, from=0, to=1]
\end{tikzcd}\]Using Lemmas \ref{lem:main_L} and \ref{lem:SMRelCatX(2.1)loc} to
identify the localization of each category at local equivalences,
we obtain a functor $I:\SM\Cat_{\infty}\pr{\kappa}\to\kappa\-\Pr\SM$
and a natural transformation depicted as
\[\begin{tikzcd}
	{\mathcal{SM}\mathcal{C}\mathsf{at}_\infty^{}(\kappa)} && {\kappa\text{-}\mathcal{P}\mathsf{r}\mathcal{SM}} \\
	& {\mathcal{SM}\widehat{\mathcal{C}\mathsf{at}}_\infty.}
	\arrow["I", from=1-1, to=1-3]
	\arrow[""{name=0, anchor=center, inner sep=0}, hook, from=1-1, to=2-2]
	\arrow[""{name=1, anchor=center, inner sep=0}, hook', from=1-3, to=2-2]
	\arrow["\alpha", between={0.2}{0.8}, Rightarrow, from=0, to=1]
\end{tikzcd}\]By construction, the pair $\pr{I,\alpha}$ satisfies the hypothesis
of Corollary \ref{cor:MGUD_2}. It follows from this corollary that
$I$ is an equivalence, and we are done.
\end{proof}

\subsection{\label{subsec:strongkappa}Strongly $\kappa$-combinatorial model
categories}

Combinatorial model categories enjoy the following curious property,
first articulated by Dugger in \cite{Dug01b} and later expanded by
Low \cite{Low16}: Beyond sufficiently large regular cardinals, the
distinction between ordinary-categorical notions and $\infty$-categorical
notions gets blurry. In this subsection, we record several results
that embody this principle.

The following definition is essentially due to Low \cite[Definition 5.1]{Low16}:
\begin{defn}
\label{def:strongly_comb}Let $\kappa$ be a regular cardinal. We
say that a model category $\mathbf{M}$ is \textbf{strongly $\kappa$-combinatorial
}if there is a regular cardinal $\kappa_{0}\vartriangleleft\kappa$
with the following properties:
\begin{enumerate}
\item $\mathbf{M}$ is locally $\kappa_{0}$-presentable.
\item $\mathbf{M}_{\kappa}$ is closed under finite limits in $\mathbf{M}$.
\item Each hom-set in $\mathbf{M}_{\kappa_{0}}$ is $\kappa$-small.
\item There are $\kappa$-small sets of morphisms in $\mathbf{M}_{\kappa_{0}}$
that cofibrantly generate the model structure on $\mathbf{M}$.
\end{enumerate}
\end{defn}

\begin{rem}
We can also define \textbf{strongly $\kappa$-tractable model categories}
as strongly $\kappa$-combinatorial model categories satisfying conditions
(1) through (4) of Definition \ref{const:MInd} and the following
additional condition: In (4), the elements in the generating sets
must have a cofibrant domain. Likewise, we have the notion of \textbf{strongly
$\kappa$-tractable semi-model categories}.
\end{rem}

The following results summarize the basic properties of strongly $\kappa$-combinatorial
model categories.
\begin{prop}
\cite[Propositions 5.6, 5.12]{Low16}\label{prop:Low16a} For every
combinatorial model category $\mathbf{M}$, there is a regular cardinal
$\kappa$ such that $\mathbf{M}$ is strongly $\kappa$-combinatorial.
Moreover, in this situation, $\mathbf{M}_{\kappa}$ is a model category
whose cofibrations, fibrations, and weak equivalences are precisely
those of $\mathbf{M}$ of $\kappa$-compact objects.
\end{prop}

\begin{prop}
\cite[Remark 5.2]{Low16}\label{prop:Low16b} Let $\kappa'\vartriangleright\kappa$
be regular cardinals. Every strongly $\kappa$-combinatorial model
category is strongly $\kappa'$-combinatorial.
\end{prop}

\begin{prop}
\label{prop:strongly_comb_ind}Let $\kappa$ be a regular cardinal,
and let $\mathbf{M}$ be a strongly $\kappa$-combinatorial model
category. Then the $\infty$-category $\mathbf{M}_{\kappa,\infty}=\pr{\mathbf{M}_{\kappa}}_{\infty}$
admits $\kappa$-small colimits, and the functor
\[
\iota\from\mathbf{M}_{\kappa,\infty}\to\mathbf{M}_{\infty}
\]
exhibits $\mathbf{M}_{\infty}$ as an $\Ind_{\kappa}$-completion
of $\mathbf{M}_{\kappa,\infty}$.
\end{prop}

\begin{proof}
The fact that $\mathbf{M}_{\kappa,\infty}$ admits $\kappa$-small
colimits follows from the argument of \cite[Proposition 7.7.4]{HCHA}.
Next, to show that $\iota$ induces a categorical equivalence $\Ind_{\kappa}\pr{\mathbf{M}_{\kappa,\infty}}\xrightarrow{\simeq}\mathbf{M}_{\infty}$,
we use \cite[Proposition 5.3.5.11]{HTT}. We must prove the following:
\begin{enumerate}
\item The functor $\iota$ is fully faithful.
\item The essential image of $\iota$ generates $\mathbf{M}_{\infty}$ under
$\kappa$-filtered colimits.
\item The image of $\iota$ lies in $\pr{\mathbf{M}_{\infty}}_{\kappa}$.
\end{enumerate}

Claim (1) follows from Proposition \ref{prop:Low16a} and \cite[Corrollary 3.1]{ACK25},
which asserts that the classical derived mapping spaces (computed
using simplicial resolutions objects) computes the mapping spaces
of the underlying $\infty$-categories of model categories.

For (2), recall that $\kappa$-filtered colimits in $\mathbf{M}$
are already homotopy colimits (Corollary \ref{cor:RR15_3.1}). Thus,
we only have to show that $\mathbf{M}_{\kappa}$ generates $\mathbf{M}$
under $\kappa$-filtered colimits. This is clear, because $\mathbf{M}$
is locally $\kappa$-presentable. 

For (3), suppose we are given a $\kappa$-filtered $\infty$-category
$\mathcal{I}$ and a colimit diagram $\overline{F}\from\mathcal{I}^{\rcone}\to\mathbf{M}_{\infty}$.
We must show that, for every object $X\in\mathbf{M}_{\kappa}$, the
diagram $\mathbf{M}_{\infty}\pr{X,-}\circ\overline{F}\from\mathcal{I}^{\rcone}\to\mathcal{S}$
is a colimit diagram. Using \cite[\href{https://kerodon.net/tag/02QA}{Tag 02QA}]{kerodon},
we may assume that $\mathcal{I}$ is (the nerve of) a poset. In this
case, the functor $\Fun\pr{\mathcal{I}^{\rcone},\mathbf{M}}\to\Fun\pr{\mathcal{I}^{\rcone},\mathbf{M}_{\infty}}$
is a localization \cite[Theorem 7.9.8]{HCHA}, so we may assume that
$F$ lifts to a diagram $\overline{G}:\mathcal{I}^{\rcone}\to\mathbf{M}$.
Without loss of generality, we may assume that $\overline{G}$ takes
values in the full subcategory of fibrant objects. Set $G=\overline{G}\vert\mathcal{I}$.
Since $\cat I$ is $\kappa$-filtered, the map 
\[
\colim_{\mathcal{I}}G\to\overline{G}\pr{\infty}
\]
is a weak equivalence. Moreover, since fibrations of $\mathbf{M}$
are stable under $\kappa$-filtered colimits, the object $\colim_{\mathcal{I}}G\in\mathbf{M}$
is fibrant. Thus, we may assume that $\overline{G}$ is a strict colimit
diagram.

Now choose a cosimplicial resolution $X^{\bullet}$ of $X$ in $\mathbf{M}_{\kappa}$.
By \cite[Corrollary 3.1]{ACK25}, the functor $\mathbf{M}\pr{X^{\bullet},-}:\mathbf{M}_{\fib}\to\SS$
descends to the functor $\mathbf{M}_{\infty}\to\mathcal{S}$ corepresented
by $X$. Therefore, we are reduced to showing that the diagram $\mathbf{M}\pr{X^{\bullet},\overline{G}}:\mathcal{I}^{\rcone}\to\SS$
is a homotopy colimit diagram. Since $X^{\bullet}$ carries each object
to a $\kappa$-compact object, this diagram is a colimit diagram.
Since filtered colimits are homotopy colimits in $\SS$, we are done.
\end{proof}

\begin{cor}
\label{cor:modcat_Indcomp}Let $\kappa$ be a regular cardinal, and
let $\mathbf{M}$ be a symmetric $\mathbf{S}$-algebra whose underlying
model category is strongly $\kappa$-combinatorial. The inclusion
$\mathbf{M}_{\kappa,\cof}\hookrightarrow\mathbf{M}$ induces a left
Quillen equivalence of $\mathbf{S}$-algebras
\begin{align*}
\theta\from\MInd^{\mathbf{S}}_{\kappa}\pr{L\pr{\mathbf{M}_{\kappa,\cof}}} & \xrightarrow{\simeq}\mathbf{M}\\
X & \mapsto\int^{M\in\mathbf{M}_{\kappa,\cof}}M\otimes X\pr M.
\end{align*}
\end{cor}

\begin{proof}
We first show that $\theta$ induces a left Quillen functor
\[
\theta'\from\Fun\pr{\mathbf{M}^{\op}_{\kappa,\cof},\mathbf{S}}\to\mathbf{M},
\]
where $\Fun\pr{\mathbf{M}^{\op}_{\kappa,\cof},\mathbf{S}}$ carries
the projective model structure. Since $\mathbf{M}$ is strongly $\kappa$-combinatorial,
it has a generating set $I$ of cofibrations consisting of maps in
$\mathbf{M}_{\kappa}$. The maps $\{\mathbf{M}\pr{-,M}\otimes i\}_{M\in\mathbf{M}_{\kappa,\cof},\,i\in I}$
generate projective cofibrations of $\Fun\pr{\mathbf{M}^{\op}_{\kappa,\cof},\mathbf{S}}$
\cite[Theorem 11.6.1]{Hirschhorn}. The images of these maps under
$\theta'$ are simply the maps $\{M\otimes i\}_{M\in\mathbf{M}_{\kappa,\cof},\,i\in I}$,
which are all cofibrations. Hence, $\theta'$ carries projective cofibrations
to cofibrations. A similar argument shows that $\theta'$ carries
projective trivial cofibrations to trivial cofibrations. Thus $\theta'$
is left Quillen, as claimed.

We now consider the following diagram, which commutes up to natural
equivalence:
\[\begin{tikzcd}
	{\mathbf{M}_{\kappa,\mathrm{cof}}} & {\mathbf{M}_{\kappa,\infty}} \\
	{\operatorname{Fun}(\mathbf{M}_{\kappa,\mathrm{cof}}^{\mathrm{op}},\mathbf{S})_{\infty}} & {\mathbf{M}_{\infty}} \\
	{\mathbf{Ind}_{\kappa}^{\mathbf{S}}(L(\mathbf{M}_{\kappa,\mathrm{cof}}))_{\infty}.}
	\arrow[from=1-1, to=1-2]
	\arrow[from=1-1, to=2-1]
	\arrow["\iota", from=1-2, to=2-2]
	\arrow["{\theta'_\infty}", from=2-1, to=2-2]
	\arrow[from=2-1, to=3-1]
\end{tikzcd}\]According to Proposition \ref{prop:strongly_comb_ind}, the functor
$\iota$ exhibits $\mathbf{M}_{\infty}$ as an $\Ind_{\kappa}$-completion
of $\mathbf{M}_{\kappa,\infty}$. Also, by what we have just shown
in the previous paragraph, the functor $\theta_{\infty}'$ preserves
small colimits. It follows that $\theta'_{\infty}$ can be identified
with the functor 
\[
\Fun\pr{\mathbf{M}^{\op}_{\kappa,\cof},\mathcal{S}}\to\Ind_{\kappa}\pr{\mathbf{M}_{\kappa,\cof}}
\]
described in Construction \ref{const:MInd}. By the definition of
weak equivalences of $\MInd^{\mathbf{S}}_{\kappa}\pr{L\pr{\mathbf{M}_{\kappa,\cof}}}$,
this means that $\theta$ preserves and reflects weak equivalences.
Since we already know that $\theta'$ is a left Quillen functor, this
is enough to conclude that $\theta$ is a left Quillen functor. The
resulting functor 
\[
\theta_{\infty}\from\MInd^{\mathbf{S}}_{\kappa}\pr{L\pr{\mathbf{M}_{\kappa,\cof}}}_{\infty}\to\mathbf{M}_{\infty}
\]
is an equivalence, because both sides are the localization of $\Fun\pr{\mathbf{M}^{\op}_{\kappa,\cof},\mathbf{S}}_{\infty}$
at the same class of morphisms. The proof is now complete.
\end{proof}

We conclude this subsection by taking monoidal structures into account.
\begin{defn}
Let $\kappa$ be a regular cardinal. A \textbf{monoidally $\kappa$-combinatorial
model category} is a monoidal model category $\mathbf{M}$ satisfying
the following conditions:
\begin{itemize}
\item As a model category, $\mathbf{M}$ is strongly $\kappa$-combinatorial.
\item $\kappa$-compact objects are stable under finite tensor products.
(In particular, the unit object is $\kappa$-compact.)
\end{itemize}
We write $\CMMC\pr{\kappa}\subset\CMMC$ for the subcategory spanned
by the monoidally $\kappa$-combinatorial model category and those
maps that preserve $\kappa$-compact objects.
\end{defn}

\begin{lem}
\label{lem:monoidally_comb}\hfill 
\begin{enumerate}
\item For every pair of regular cardinals $\kappa\lcone\lambda$, we have
$\CMMC\pr{\kappa}\subset\CMMC\pr{\lambda}$.
\item Every morphism of $\CMMC$ belongs to $\CMMC\pr{\kappa}$ for some
regular cardinal $\kappa$. 
\end{enumerate}
\end{lem}

\begin{proof}
We first show that every combinatorial monoidal model category $\mathbf{M}$
is monoidally $\kappa$-combinatorial for some $\kappa$, and that
such an $\mathbf{M}$ is monoidally $\lambda$-combinatorial for every
$\lambda\rcone\kappa$. By Proposition \ref{prop:Low16a}, there is
some regular cardinal $\kappa_{0}$ such that $\mathbf{M}$ is strongly
$\kappa_{0}$-combinatorial as a model category. Find a regular cardinal
$\kappa\rcone\kappa_{0}$ such that, if $X,Y\in\mathbf{M}_{\kappa_{0}}$,
then $X\otimes Y\in\mathbf{M}_{\kappa}$. Since every $\kappa$-compact
object is a $\kappa$-small colimit of $\kappa_{0}$-compact object
\cite[Theorem 2.3.11]{MP89}, this ensures that $\kappa$-compact
objects is stable under tensor product. Hence, $\mathbf{M}$ is monoidally
$\kappa$-combinatorial. A similar argument, using Proposition \ref{prop:Low16b},
shows that every monoidally $\kappa$-combinatorial model category
is monoidally $\lambda$-combinatorial.

We now prove (1) and (2). Part (1) follows from the result in the
previous paragraph and \cite[Remark 2.20]{AR1994}. For (2), let $F\from\mathbf{M}\to\mathbf{N}$
be a morphism of $\CMMC$. By the result in the previous paragraph
and part (1), we can find some $\kappa_{0}$ such that $\mathbf{M}$
and $\mathbf{N}$ are monoidally $\kappa_{0}$-combinatorial. Choose
$\kappa\rcone\kappa_{0}$ such that $F\pr{\mathbf{M}_{\kappa_{0}}}\subset\mathbf{N}_{\kappa}$.
We claim that $F$ belongs to $\CMMC\pr{\kappa}$. 

Since every $\kappa$-compact object is a $\kappa$-small colimit
of $\kappa_{0}$-compact object \cite[Theorem 2.3.11]{MP89}, our
choice of $\kappa$ and $\kappa_{0}$ ensures that $F\pr{\mathbf{M}_{\kappa}}\subset\mathbf{N}_{\kappa}$.
We also know from (1) that $\mathbf{M}$ and $\mathbf{N}$ are monoidally
$\kappa$-combinatorial. Therefore, $F$ belongs to $\CMMC\pr{\kappa}$,
as desired.
\end{proof}

\subsection{Proof of Proposition \ref{prop:main_flat}}

We now arrive at the proof of Proposition \ref{prop:main_flat}. 

\begin{notation}
We write $\urCard$ for the (large) poset of (small) uncountable regular
cardinals, ordered by $\lcone$. Given a functor $F:\urCard\to\hat{\mathsf{Cat}}$,
we write $\int^{\urCard}F\to\urCard$ for the Grothendieck construction
of $F$. Likewise, given a functor $G:\urCard^{\op}\to\hat{\mathsf{Cat}}$,
we write $\int_{\urCard}G\to\urCard$ for its Grothendieck construction.
\end{notation}

\begin{notation}
We let $\combsymalg{\mathbf{S}}\subset\CombSymAlg{\mathbf{S}}$ denote
the full subcategory of $\mathbf{S}$-algebras $\mathbf{M}$ with
the following property: For every regular cardinal $\kappa$ such
that (the underlying category of) $\mathbf{M}$ is $\kappa$-presentable,
the full subcategory $\mathbf{M}_{\kappa}\subset\mathbf{M}$ of $\kappa$-compact
objects is \textit{literally} small (not just essentially small).

We define a full subcategory $\combsymalg{\mathbf{S}}_{\pr{2,1}}\subset\combsymalg{\mathbf{S}}_{\pr{2,1}}$
similarly. 
\end{notation}

\begin{rem}
\label{rem:(2.1)_V}The functor $\combsymalg{\mathbf{S}}\to\combsymalg{\mathbf{S}}_{\pr{2,1}}$
is a localization. This follows from the argument in Proposition \ref{prop:(2.1)_V}
and the following observation \cite[Proposition 2.23]{ZhenLin_Universe}:
Let $\kappa$ be a regular cardinal, and let $\mathbf{C}$ be a locally
$\kappa$-presentable category. For every category $\mathbf{I}$ with
only finitely many morphisms, the category $\Fun\pr{\mathbf{I},\mathbf{C}}$
is locally $\kappa$-presentable, and its $\kappa$-compact objects
are exactly those diagrams $\mathbf{I}\to\mathbf{C}$ carrying each
object to a $\kappa$-compact object.
\end{rem}

\begin{proof}
[Proof of Proposition \ref{prop:main_flat}]Consider the following
diagram:
\[\begin{tikzcd}[column sep = small]
	{\mathsf{combsym}\mathbf{S}\text{-}\mathsf{alg}[\mathrm{Quill.eq}^{-1}]} & {\mathsf{CombSym}\mathbf{S}\text{-}\mathsf{Alg}[\mathrm{Quill.eq}^{-1}]} & {\mathsf{SMRel}\widehat{\mathsf{Cat}}_{\mathcal{P}\mathsf{r}\mathcal{SM}}} \\
	{\mathsf{combsym}\mathbf{S}\text{-}\mathsf{alg}_{(2,1)}[\mathrm{Quill.eq}^{-1}]} & {\mathsf{CombSym}\mathbf{S}\text{-}\mathsf{Alg}_{(2,1)}[\mathrm{Quill.eq}^{-1}]} & {\mathsf{SMRel}\widehat{\mathsf{Cat}}_{\mathcal{P}\mathsf{r}\mathcal{SM},(2,1)}}
	\arrow[from=1-1, to=1-2]
	\arrow["{\alpha'}", from=1-1, to=2-1]
	\arrow[from=1-2, to=1-3]
	\arrow["\alpha"', from=1-2, to=2-2]
	\arrow["\beta"', from=1-3, to=2-3]
	\arrow["\gamma", from=2-1, to=2-2]
	\arrow[from=2-2, to=2-3]
\end{tikzcd}\]By Proposition \ref{prop:(2.1)_V}, the map $\alpha$ is a localization,
and we saw in Remark \ref{rem:(2.1)_V} that $\alpha'$ is also a
localization. The map $\beta$ is a localization by Lemma \ref{lem:SMRelCatX(2.1)loc}.
Furthermore, the map $\gamma$ is an equivalence (as it is an equivalence
even before localizing). Thus, it suffices to show that bottom composite
\[
\combsymalg{\mathbf{S}}_{\pr{2,1}}[\Quilleq^{-1}]\to\SMRelCatlarge_{\pr{2,1}}[\loceq^{-1}]
\]
is an equivalence. 

For each uncountable regular cardinal $\kappa$, let $\combsymalg{\mathbf{S}}\pr{\kappa}$
denote the fiber product
\[
\combsymalg{\mathbf{S}}\times_{\CMMC}\CMMC\pr{\kappa}.
\]

We now consider the following diagram:
\[\begin{tikzcd}
	{\int^{\mathsf{urCard}}\mathsf{combsym}\mathbf{S}\text{-}\mathsf{alg}(-)} & {\int_{\mathsf{urCard}}\mathsf{SMRelCat}_{\mathcal{SMC}\mathsf{at}_\infty(-)}} \\
	{\mathsf{combsym}\mathbf{S}\text{-}\mathsf{alg}_{(2,1)}} & {\mathsf{SMRel}\widehat{\mathsf{Cat}}_{\mathcal{P}\mathsf{r}\mathcal{SM},(2,1)}}
	\arrow["F", from=1-1, to=1-2]
	\arrow["U"', from=1-1, to=2-1]
	\arrow["{\mathbf{Ind}^{\mathbf{S}}}"{description}, from=1-2, to=2-1]
	\arrow["{(-)_\flat\circ \mathbf{Ind}^{\mathbf{S}}}", from=1-2, to=2-2]
	\arrow["{(-)_\flat}"', from=2-1, to=2-2]
\end{tikzcd}\]Here $U$ and $F$ are functors and $\MInd^{\mathbf{S}}$ is a strictly
unitary pseudofunctor, defined by the formulas
\begin{align*}
F\pr{\kappa,\mathbf{M}} & =\pr{\kappa,\mathbf{M}_{\kappa,\cof}},\\
U\pr{\kappa,\mathbf{M}} & =\mathbf{M},\\
\MInd^{\mathbf{S}}\pr{\kappa,\mathcal{C}} & =\MInd^{\mathbf{S}}_{\kappa}\pr{L\pr{\mathcal{C}}}.
\end{align*}
The lower triangle commutes by construction; the upper triangle does
not commute on the nose, but Corollary \ref{cor:modcat_Indcomp} gives
a pseudonatural Quillen equivalence $\MInd^{\mathbf{S}}\circ F\xrightarrow{\simeq}U$.
Thus, to prove the claim, it suffices to show that $U$ and $\pr -_{\flat}\circ\MInd^{\mathbf{S}}$
induce equivalences of $\infty$-categories when localized at the
maps whose images in $\Pr\SM$ are equivalences.

We start from the claim on $U$. We claim more strongly that $U$
is a localization. According to Lemma \ref{lem:monoidally_comb},
the category $\combsymalg{\mathbf{S}}$ is the colimit of $\combsymalg{\mathbf{S}}\pr{\kappa}$
as $\kappa$ ranges over $\urCard$. It follows from \cite[\href{https://kerodon.net/tag/02UU}{Tag 02UU}]{kerodon}
that the functor 
\[
\int^{\urCard}\combsymalg{\mathbf{S}}\pr -\to\combsymalg{\mathbf{S}}
\]
is a localization. Since the functor $\combsymalg{\mathbf{S}}\to\combsymalg{\mathbf{S}}_{\pr{2,1}}$
is also a localization (Remark \ref{rem:(2.1)_V}), we deduce that
$U$ is also a localization. 

Next, for $\pr -_{\flat}\circ\MInd^{\mathbf{S}}$, we factor it as
\begin{align*}
\int_{\urCard}\SMRelCat_{\SM\Cat_{\infty}\pr -} & \xrightarrow{\Phi}\int_{\urCard}\SMRelCatlarge_{\pr -\-\Pr\SM,\pr{2,1}}\\
 & \xrightarrow{\Psi}\SMRelCatlarge_{\Pr\SM,\pr{2,1}},
\end{align*}
where $\int^{\urCard}\SMRelCatlarge_{\pr -\-\Pr\SM,\pr{2,1}}$ denotes
the cartesian fibration associated with the functor $\kappa\mapsto\SMRelCatlarge_{\kappa\-\Pr\SM,\pr{2,1}}$,
concretely realized as (the dual of) Lurie's relative nerve \cite[Definition 3.2.5.2]{HTT}.
The functor $\Phi$ is induced by the functors $\{\MInd^{\mathbf{S}}_{\kappa}\pr{L\pr -}_{\flat}\}_{\kappa}$,
and $\Psi$ is the forgetful functor. As in the previous paragraph,
the functor $\Psi$ is already a localization, so it suffices to verify
that $\Phi$ induces a localization upon localizing at the maps whose
images in $\Pr\SM$ are equivalences. This follows from (the dual
of) \cite[\href{https://kerodon.net/tag/02LW}{Tag 02LW}]{kerodon}
and Proposition \ref{prop:rel_GUdual}, which show more strongly that
it induces a localization when we localize at fiberwise local equivalences.
The proof is now complete.
\end{proof}

\section{\label{sec:main_U}Proof of Proposition \ref{prop:main_U}}

In this section, we take up the proof of Proposition \ref{prop:main_U}.
In light of Propositions \ref{prop:main_L} and \ref{prop:main_flat},
it will suffice to prove this for one \textit{specific} choice of
$\mathbf{S}$. To this end, we will construct a new model category
of \textit{symmetric cubical sets} in Subsection \ref{subsec:sym_cube_set}.
We then use this to prove Proposition \ref{prop:main_U} in Subsection
\ref{subsec:main_U}.

\subsection{\label{subsec:sym_cube_set}Symmetric cubical sets}

In this subsection, we construct the symmetric monoidal model category
of symmetric cubical sets (Theorem \ref{thm:sym_cube}). It has two
essential features: Firstly, it models the homotopy theory of spaces,
in the sense that its underlying symmetric monoidal $\infty$-category
is equivalent to the cartesian monoidal $\infty$-category of spaces.
Secondly, every combinatorial symmetric monoidal model category with
cofibrant unit can be enriched over symmetric cubical sets in an essentially
unique manner (Corollary \ref{cor:symcub}). Together, these features
make the category of symmetric cubical sets an ideal choice of $\mathbf{S}$.
\begin{rem}
Part of the contents of this subsection is similar in spirit to Isaacson's
papers \cite{isaacson2009symmetriccubicalsets,Isa11}. In these papers,
he considered symmetric cubical sets by using a slightly bigger category
of symmetric cubes. He then showed that every combinatorial symmetric
monoidal model category with cofibrant unit and which satisfies the
monoid axiom can be enriched over symmetric cubical sets \cite[Theorem 10.1]{isaacson2009symmetriccubicalsets}.
However, his approach does not seem to give an essential uniqueness
of this enrichment.
\end{rem}

We start by recalling the classical cube category and the Grothendieck
model structure.
\begin{defn}
The \textbf{box category} $\square$ has the following descriptions:
\begin{itemize}
\item Objects are the posets $[1]^{n}$, where $n\geq0$.
\item A morphism $f:[1]^{n}\to[1]^{m}$ is a poset map that erases coordinates
and inserts $0$ and $1$ without changing the order of the coordinates.
Equivalently, they are  generated under composition by the following
maps:
\begin{enumerate}
\item The \textbf{face map} $\delta^{i,\varepsilon}=\delta^{i,\varepsilon}_{n}:[1]^{n}\to[1]^{n+1}$
for $n\geq0$ and $1\leq i\leq n+1$, which inserts $\varepsilon\in\{0,1\}$
in the $i$th coordinate.
\item The \textbf{degeneracy map} $\sigma^{i}_{n}:[1]^{n+1}\to[1]^{n}$
for $n\geq0$ and $1\leq i\leq n+1$, which deletes the $i$th coordinate.
\end{enumerate}
\end{itemize}
We will write $\square_{\leq1}\subset\square$ for the full subcategory
spanned by the objects $[1]^{n}$ with $n\leq1$. 

The category $\square$ is naturally a Reedy category. (In fact, an
Eilenberg--Zilber category in the sense of \cite[Definition 1.3.1]{HCHA}.)
Also, the cartesian product of posets makes $\square$ into a monoidal
category. The category $\Set^{\square^{\op}}$ of \textbf{cubical
sets} admits an induced monoidal structure, given by Day's convolution
product. The presheaf represented by $[1]^{n}$ will be denote by
$\square^{n}$.
\end{defn}

The category $\Set^{\square^{\op}}$ can be endowed with a model structure
which models the homotopy theory of spaces. To state it more precisely,
set $\partial\square^{n}=\bigcup_{i,\varepsilon}\delta^{i,\varepsilon}\pr{\square^{n-1}}$
and $\sqcap^{n}_{i,\varepsilon}=\bigcup_{\pr{j,\eta}\neq\pr{i,\varepsilon}}\delta^{i,\varepsilon}\pr{\square^{n-1}}$.
We consider the following sets of morphisms of cubical sets: 
\[
I=\{\partial\square^{n}\to\square^{n}\mid n\geq0\},\,J=\{\sqcap^{n}_{i,\varepsilon}\to\square^{n}\mid n\geq1,\,1\leq i\leq n,\,\varepsilon\in\{0,1\}\}.
\]

\begin{thm}
\cite{cisinski_test}, \cite[Theorem 6.2, Theorem 8.6, Theorem 8.8]{jardine_cathtpy}\label{thm:cis}
The category $\Set^{\square^{\op}}$ of \textbf{cubical sets} has
a combinatorial monoidal model structure called the \textbf{Grothendieck
model structure}, whose:
\begin{itemize}
\item cofibrations are the monomorphisms.
\item weak equivalences are preserved and detected by the triangulation
functor $\Set^{\square^{\op}}\to\Set^{\Del^{\op}}$, which is the
left adjoint carrying $\square^{n}\mapsto\pr{\Delta^{1}}^{n}$.
\end{itemize}
Moreover, $I$ and $J$ cofibrantly generate this model structure,
and $T$ is a left Quillen equivalence with respect to the Kan--Quillen
model structure on simplicial sets.
\end{thm}

Having recalled the basics of cubical sets, we now develop the theory
of symmetric cubical sets. Our goal is to consider the symmetric monoidal
version of Theorem \ref{thm:cis} (Theorem \ref{thm:sym_cube}) and
to prove its model-categorical universal property (Corollary \ref{cor:symcub}).
\begin{defn}
We define the symmetric monoidal category of \textbf{symmetric box
category} $\square_{\Sigma}$ by adjoining to the box category the
permutation map $\pi_{p}:[1]^{n}\to[1]^{n}$ for each $n\geq0$ and
$p\in\Sigma_{n}$, defined by $\pi_{p}\pr{x_{1},\dots,x_{n}}=\pr{x_{p^{-1}\pr 1},\dots,x_{p^{-1}\pr n}}$.
There is an inclusion $i:\square\to\square_{\Sigma}$. We denote by
$\square^{n}_{\Sigma}$ the presheaves on $\square_{\Sigma}$ represented
by $[1]^{n}$. Note that $\square_{\leq1}$ is a full subcategory
of $\square_{\Sigma}$.
\end{defn}

The following proposition is immediate from the definitions:
\begin{prop}
We have the following \textbf{cocubical relations} for maps in $\square_{\Sigma}$:
\begin{itemize}
\item $\delta^{j,\eta}\delta^{i,\varepsilon}=\delta^{i,\varepsilon}\delta^{j-1,\eta}$
for $j\leq i$.
\item $\sigma^{j}\delta^{i,\varepsilon}=\begin{cases}
\delta^{i,\varepsilon}\sigma^{j-1} & \text{if }i<j,\\
\id & \text{if }i=j,\\
\delta^{i-1}\sigma^{j} & \text{if }i>j.
\end{cases}$
\item $\sigma^{j}\sigma^{i}=\sigma^{i}\sigma^{j+1}$ if $i\leq j$.
\item $\pi_{p}\pi_{q}=\pi_{pq}$.
\item $\pi_{p}\delta^{i,\varepsilon}=\delta^{p\pr i}\pi_{q}$, where $q$
denotes the composite
\[
\{1,\dots,n\}\cong\{1,\dots,n+1\}\setminus\{i\}\xrightarrow{p}\{1,\dots,n+1\}\setminus\{p\pr i\}\cong\{1,\dots,n\}.
\]
\item $\sigma^{i}\pi_{p}=\pi_{q}\sigma^{p^{-1}\pr i}$, where $q$ denotes
the composite
\[
\{1,\dots,n-1\}\cong\{1,\dots n\}\setminus\{p^{-1}\pr i\}\xrightarrow{p}\{1,\dots,n\}\setminus\{i\}\cong\{1,\dots,n-1\}.
\]
\end{itemize}
Moreover, every map in $\square_{\Sigma}$ can be written uniquely
as
\[
\delta^{i_{1},\varepsilon_{1}}\cdots\delta^{i_{n}\varepsilon_{n}}\pi_{p}\sigma^{j_{1}}_{l}\cdots\sigma^{j_{m}},
\]
where $i_{1}>\cdots i_{n}$, $j_{1}<\cdots<j_{m}$, and $p\in\Sigma_{l}$.
\end{prop}

\begin{cor}
\label{cor:cube}The symmetric box category $\square_{\Sigma}$ is
freely generated by the face maps, degeneracy maps, and permutation
maps, subject to the cocubical relations. A similar claim holds for
the classical box category.
\end{cor}

Corollary \ref{cor:cube} is quite elementary, but it has an important
consequence. Namely, it can be used to show that $\square_{\Sigma}$
is the free symmetric monoidal category generated by an ``interval'':
\begin{cor}
\label{cor:sym_cube}Let $\pr{\mathcal{C},\otimes,\mathbf{1}}$ be
a symmetric monoidal category. The evaluation at the unit object $[1]\in\square_{\Sigma}$
determines a categorical equivalence
\[
\theta\from\Fun^{\t}\pr{\square_{\Sigma},\mathcal{C}}\xrightarrow{\simeq}\Fun\pr{\square_{\leq1},\mathcal{C}}\times_{\Fun\pr{\{[1]^{0}\},\mathcal{C}}}\{\mathbf{1}\}.
\]
A similar claim holds for monoidal categories and the plain cube category.
\end{cor}

\begin{proof}
The functor $\theta$ needs a bit of explanation. If $F:\square_{\Sigma}\to\mathcal{C}$
is a symmetric monoidal functor, then its image in $\Fun\pr{\square_{\leq1},\mathcal{C}}\times_{\Fun\pr{\{[1]^{0}\},\mathcal{C}}}\{\mathbf{1}\}$
is given by the functor $\square_{\leq1}\to\mathcal{C}$ obtained
by modifying the value of $F$ at $[1]^{0}$ to $\mathbf{1}$ using
the structure map $\mathbf{1}\xrightarrow{\cong}F\pr{[1]^{0}}$ of
$F$.

With this in mind, we will give an explicit inverse equivalence of
$\theta$. Given a functor $F:\square_{\leq1}\to\mathcal{C}$ carrying
$[1]^{0}$ to the unit object, we can define a new functor $\widetilde{F}\from\square_{\Sigma}\to\mathcal{C}$
by $\widetilde{F}\pr{[1]^{n}}=F\pr{[1]}^{\otimes n}$, with structure
maps given by that of $F$ and the braiding of $\mathcal{C}$ (Corollary
\ref{cor:cube}). The coherence maps of $\mathcal{C}$ make $\widetilde{F}$
into a symmetric monoidal functor. The functor $F\mapsto\widetilde{F}$
is an inverse equivalence of $\theta$.
\end{proof}

\begin{cor}
\label{cor:sym_cube_day}Let $\pr{\mathcal{C},\otimes,\mathbf{1}}$
be a cocomplete symmetric monoidal category whose tensor product preserves
small colimits in each variable. The evaluation at the unit object
$\square^{0}_{\Sigma}\in\Set^{\square^{\op}_{\Sigma}}$ induces a
categorical equivalence
\[
\Fun^{\t,cc}\pr{\Set^{\square^{\op}_{\Sigma}},\mathcal{C}}\xrightarrow{\simeq}\Fun\pr{\square_{\leq1},\mathcal{C}}\times_{\Fun\pr{\{[1]^{0}\},\mathcal{C}}}\{\mathbf{1}\},
\]
where $\Fun^{\t,cc}\pr{\Set^{\square^{\op}_{\Sigma}},\mathcal{C}}$
denotes the category of symmetric monoidal categories whose underlying
functor preserves small colimits, and symmetric monoidal natural transformations
between them.
\end{cor}

\begin{proof}
This follows from Corollary \ref{cor:sym_cube} and the universal
property of Day convolution product monoidal structure.
\end{proof}

We can now state our first main result of this section. For the statement
theorem, we will write $\partial\square^{n}_{\Sigma}\subset\square^{n}_{\Sigma}$
for the subpresheaf consisting of the maps $[1]^{k}\to[1]^{n}$ that
inserts at least one $0$ or $1$. Equivalently, we have $i_{!}\pr{\partial\square^{n}}=\partial\square^{n}_{\Sigma}$.
\begin{thm}
\label{thm:sym_cube}The category $\Set^{\square^{\op}_{\Sigma}}$
has a tractable model structure whose weak equivalences and fibrations
are preserved and detected by the forgetful functor $i^{*}:\Set^{\square^{\op}_{\Sigma}}\to\Set^{\square^{\op}}$.
Moreover:

\begin{enumerate}[label = (\Roman*)]

\item The adjunction
\[
i_{!}:\Set^{\square^{\op}}\adj\Set^{\square^{\op}_{\Sigma}}:i^{*}
\]
is a Quillen equivalence.

\item The model structure on $\Set^{\square^{\op}_{\Sigma}}$ is
symmetric monoidal with respect to the Day convolution.

\item The model structure is left proper.

\item If $\mathbf{M}$ is a symmetric monoidal model category and
$F:\Set^{\square^{\op}_{\Sigma}}\to\mathbf{M}$ is a symmetric monoidal
functor which is also a left adjoint, then $F$ is left Quillen precisely
when the following conditions are satisfied:
\begin{itemize}
\item The map $F\pr{\partial\square^{1}_{\Sigma}\to\square^{1}_{\Sigma}}$
is a cofibration.
\item The map $F\pr{\square^{1}_{\Sigma}}\to F\pr{\square^{0}_{\Sigma}}$
is a weak equivalence.
\item The monoidal unit of $\mathbf{M}$ is cofibrant.
\end{itemize}
\item The assignment $\square^{n}_{\Sigma}\mapsto\pr{\Delta^{1}}^{n}$
determines a symmetric monoidal left Quillen equivalence $\Set^{\square^{\op}_{\Sigma}}\to\Set^{\Del^{\op}}$.

\end{enumerate}
\end{thm}

The proof of Theorem \ref{thm:sym_cube} relies on the following lemma.
\begin{lem}
\label{lem:commutes}The diagram
\[\begin{tikzcd}
	{\mathsf{Set}^{\square^{\mathrm{op}}}} & {\mathcal{S}^{\square^{\mathrm{op}}}} \\
	{\mathsf{Set}^{\square^{\mathrm{op}}}[\mathrm{weq}^{-1}]} & {\mathcal{S}}
	\arrow[hook, from=1-1, to=1-2]
	\arrow[from=1-1, to=2-1]
	\arrow["{\operatorname{colim}_{\square^{\mathrm{op}}}}", from=1-2, to=2-2]
	\arrow["\simeq", from=2-1, to=2-2]
\end{tikzcd}\]commutes up to natural equivalence. 
\end{lem}

\begin{proof}
We first give a model for $\colim_{\square^{\op}}$. Consider the
functor 
\[
\Phi\from\Fun\pr{\square^{\op},\Set^{\square^{\op}}}\to\Set^{\square^{\op}},\,Y\mapsto\int^{[1]^{n}\in\square}Y_{n}\otimes\square^{n}.
\]
Since the cocubical object $\square^{\bullet}\in\Fun\pr{\square,\Set^{\square^{\op}}}$
is Reedy cofibrant, \cite[Proposition A.2.9.26]{HTT} shows that the
functor $\Phi$ is left Quillen for the Reedy model structure. The
right adjoint of this functor, given by $K\mapsto K^{\square^{\bullet}}$,
is weakly equivalent to that of the diagonal functor on the full subcategory
of fibrant objects. Therefore, $\Phi$ is a model of the homotopy
colimit, in the sense that the induced functor
\[
\cat S^{\square^{\op}}\to\cat S
\]
is equivalent to $\colim_{\square^{\op}}$.

Now since every cubical set $X$ is Reedy cofibrant when regarded
as a levelwise discrete cubical object in $\Set^{\square^{\op}}$,
the above argument shows that the composite of the functor $\Phi'\from\Set^{\square^{\op}}\to\Set^{\square^{\op}},\,X\mapsto\Phi\pr X$
and the localization functor $\Set^{\square^{\op}}\to\cat S$ is equivalent
to the composite
\[
\Set^{\square^{\op}}\to\cat S^{\square^{\op}}\to\cat S.
\]
On the other hand, $\Phi'$ is naturally isomorphic to the identity
functor by the coYoneda lemma. The claim follows.
\end{proof}

\begin{proof}
[Proof of Theorem \ref{thm:sym_cube}]We start by showing that the
model structure on $\Set^{\square^{\op}}$ transfers to a model structure
on $\Set^{\square^{\op}_{\Sigma}}$, using \cite[Theorem 11.3.2]{Hirschhorn}.
Let $I=\{\partial\square^{n}\to\square^{n}\mid n\geq0\}$ and $J=\{\sqcap^{n}_{i,\varepsilon}\to\square^{n}\mid n\geq1,\,1\leq i\leq n,\,\varepsilon\in\{0,1\}\}$
denote the generating cofibrations and trivial cofibrations of $\Set^{\square^{\op}}$.
We wish to show that $i^{*}$ takes $i_{!}J$-cell complexes to weak
equivalences. For this, it will suffice to show that each of the map
in $i^{*}i_{!}J$ is a trivial cofibration. The cofibration part is
clear, because $i_{!}$ and $i^{*}$ preserve monomorphisms. Consequently,
it will suffice to show that each of the maps in $i^{*}i_{!}J$ is
a weak equivalence. We will prove this by showing that the unit $\eta:\id\to i^{*}i_{!}$
is a natural weak equivalence.

By the standard argument using the skeletal filtration of presheaves
\cite[Theorem 1.3.8]{HCHA}, it will suffice to show that $\eta_{\square^{n}}:\square^{n}\to i^{*}i_{!}\square^{n}$
is a weak equivalence for all $n$. In other words, our task it to
prove that each $i^{*}i_{!}\square^{n}$ is weakly contractible. By
Lemma \ref{lem:commutes} and \cite[Corollary 3.3.4.6]{HTT}, this
is equivalent to the condition that the category $\square_{/i^{*}i_{!}\square^{n}}\cong\square\times_{\square_{\Sigma}}\square_{\Sigma/\square^{n}_{\Sigma}}$
be weakly contractible for all $n$. Thus, it suffices to show that
the functor $i:\square\to\square_{\Sigma}$ is homotopy initial. 

The homotopy initiality of $i$ is proved as follows: Consider the
commutative diagram
\[\begin{tikzcd}
	& {\square_{+}} \\
	\square && {\square _\Sigma,}
	\arrow["j"', from=1-2, to=2-1]
	\arrow["k", from=1-2, to=2-3]
	\arrow["i"', from=2-1, to=2-3]
\end{tikzcd}\]where $\square_{+}\subset\square$ denotes the subcategory spanned
by the face maps. Since homotopy initial functors have the right cancellation
property \cite[Corollary 4.1.9]{CisinskiHCHA}, it will suffice to
show that $j$ and $k$ are homotopy initial. We will show that $j$
is homotopy initial; the proof for $k$ is similar. Our task is to
show that the category $\square_{+}\times_{\square}\square_{/\square^{n}}$
is weakly contractible for every $n\geq0$. But the inclusion $\square_{+}\times_{\square}\square_{/\square^{n}}\hookrightarrow\square_{/\square^{n}}$
admits a homotopy inverse on the level of classifying spaces, given
by the functor that factors a map as a composition of degeneracy maps
followed by a composition of face maps. So we are reduced to showing
that $\square_{/\square^{n}}$ is weakly contractible, which is clear.

We now prove (I) through (V). For (I), observe that we have just shown
that the derived unit is an isomorphism. Since $\mathbb{R}i^{*}$
is conservative by construction, the triangle identities show that
the derived counit is also an isomorphism. Hence, the adjunction is
a Quillen equivalence. Part (II) follows from the fact that the model
structure on $\Set^{\square^{\op}}$ is monoidal for the Day convolution
and that $i_{!}$ is a monoidal functor. Part (III) follows from the
left properness of $\Set^{\square^{\op}}$, since $i^{*}$ preserves
cofibrations and small colimits and detects weak equivalences. Part
(IV) is proved exactly as in \cite[Corollary 1.5]{Law17}. Finally,
part (V) follows from parts (I, IV) and Theorem \ref{thm:cis}.
\end{proof}

Combining Corollary \ref{cor:sym_cube_day} and Theorem \ref{thm:sym_cube},
we get the following model-categorical universal property of symmetric
cubical sets:
\begin{cor}
\label{cor:symcub}Every cofibrantly generated symmetric monoidal
model category $\mathbf{M}$ with a cofibrant unit is a symmetric
$\Set^{\square^{\op}_{\Sigma}}$-algebra in an essentially unique
way. More precisely, the category $\Fun^{\t,LQ}\pr{\Set^{\square^{\op}_{\Sigma}},\mathbf{M}}$
is weakly contractible.
\end{cor}

\begin{proof}
Define a full subcategory $\cat X$ of 
\[
\Fun\pr{\square_{\leq1},\mathbf{M}}_{\mathbf{1}}=\Fun\pr{\square_{\leq1},\mathbf{M}}\times_{\Fun\pr{\{[1]^{0}\},\mathbf{M}}}\{\mathbf{1}\}
\]
as follows: An object of $F\in\Fun\pr{\square_{\leq1},\mathbf{M}}_{\mathbf{1}}$
can be identified with a diagram of the form
\[\begin{tikzcd}
	{\mathbf{1}} \\
	& {F([1])} & {\mathbf{1}.} \\
	{\mathbf{1}}
	\arrow[from=1-1, to=2-2]
	\arrow["{\operatorname{id}}", curve={height=-6pt}, from=1-1, to=2-3]
	\arrow[from=2-2, to=2-3]
	\arrow[from=3-1, to=2-2]
	\arrow["{\operatorname{id}}"', curve={height=6pt}, from=3-1, to=2-3]
\end{tikzcd}\]We declare that such a diagram belongs to $\cat X$ if and only if
the map $\mathbf{1}\amalg\mathbf{1}\to F\pr{[1]}$ is a cofibration
and the map $F\pr{[1]}\to\mathbf{1}$ is a weak equivalence. 

According to Corollary \ref{cor:sym_cube_day} and Theorem \ref{thm:sym_cube},
there is an equivalence of categories
\[
\Fun^{\t,LQ}\pr{\Set^{\square^{\op}_{\Sigma}},\mathbf{M}}\xrightarrow{\simeq}\cat X.
\]
Therefore, it suffices to show that $\cat X$ is weakly contractible.
To this end, we define a functor $\Phi\from\Fun\pr{\square_{\leq1},\mathbf{M}}_{\mathbf{1}}\to\Fun\pr{\square_{\leq1},\mathbf{M}}_{\mathbf{1}}$
as follows: It carries an object $\pr{\mathbf{1}\amalg\mathbf{1}\to I\to\mathbf{1}}$
to an object $\pr{\mathbf{1}\amalg\mathbf{1}\to I'\to\mathbf{1}}$,
where $I'$ is obtained by (functorially) factoring the map $\mathbf{1}\amalg\mathbf{1}\to I$
as a cofibration $\mathbf{1}\amalg\mathbf{1}\xmono{}{}I'$ followed
by a weak equivalence $I'\xrightarrow{\simeq}I$. Then $\Phi$ restricts
to a functor $\Phi_{\cat X}\from\cat X\to\cat X$. The functor $\Phi_{\cat X}$
admits natural transformations to the identity functor and the constant
functor at $\Phi\pr{\mathbf{1}\amalg\mathbf{1}\to\mathbf{1}\to\mathbf{1}}$.
Therefore, the identity functor of $\mathcal{X}$ can be connected
by a zig-zag of natural transformation to a constant functor. Hence
$\mathcal{X}$ is weakly contractible, as claimed.
\end{proof}

\subsection{\label{subsec:main_U}Proof of Proposition \ref{prop:main_U}}

We can now give a proof of Proposition \ref{prop:main_U}.
\begin{proof}
[Proof of Proposition \ref{prop:main_U}]As discussed at the beginning
of the section, it will suffice to prove the claim for a specific
choice of $\mathbf{S}$.

Take $\mathbf{S}=\Set^{\square^{\op}_{\Sigma}}$. The functor $U\from\CombSymAlg{\mathbf{S}}\to\CSMMC^{\mathbf{1}}$
is a cocartesian fibration, and its fibers are weakly contractible
by Corollary \ref{cor:symcub}. It follows from \cite[\href{https://kerodon.net/tag/02LY}{Tag 02LY}]{kerodon}
that $U$ is already a localization. In particular, it induces an
equivalence
\[
\CombSymAlg{\mathbf{S}}[\Quilleq^{-1}]\xrightarrow{\simeq}\CSMMC^{\mathbf{1}}[\Quilleq^{-1}],
\]
as claimed.
\end{proof}

\section{\label{sec:variation}Variations of the main theorem}
\begin{notation}
Throughout this section, we fix a left proper tractable symmetric
monoidal model category $\mathbf{S}$ with a cofibrant unit, whose
underlying symmetric monoidal $\infty$-category is equivalent to
the cartesian monoidal $\infty$-category $\mathcal{S}$ of $\infty$-groupoids.
\end{notation}

In Section \ref{sec:definition}, we introduced many variations of
categories of symmetric monoidal model categories. In this section,
we will identify the localizations of most of them. We also identify
the localizations of their subcategories of left Quillen equivalences. 

This section has two main results, one in the symmetric monoidal case
(Theorem \ref{thm:var_SM}), and the other in the non-symmetric monoidal
case (Theorem \ref{thm:var_Mon}). Let us start by stating the symmetric
monoidal case.
\begin{thm}
\label{thm:var_SM}Consider the diagram of categories introduced in
Section \ref{sec:definition}:
\[\begin{tikzcd}
	& {\mathsf{TractSMMC}_{\mathbf{S},\mathrm{semi}}} \\
	{\mathsf{TractSMMC}_{\mathbf{S}}} & {\mathsf{TractSMMC}^{}} & {\mathsf{TractSMMC}_{\mathrm{semi}}} \\
	{\mathsf{CombSMMC}_{\mathbf{S}}} & {\mathsf{CombSMMC}^{}}
	\arrow[from=1-2, to=2-3]
	\arrow[from=2-1, to=1-2]
	\arrow[from=2-1, to=2-2]
	\arrow[from=2-1, to=3-1]
	\arrow[from=2-2, to=2-3]
	\arrow[from=2-2, to=3-2]
	\arrow[from=3-1, to=3-2]
\end{tikzcd}\]
\begin{enumerate}
\item All of the vertices in the diagram localize to $\Pr\SM$ via the underlying
symmetric monoidal $\infty$-category functor.
\item Assertion (1) remains valid if we replace each vertex by the full
subcategory consisting of those objects with a cofibrant unit.
\item Consider the subcategories of each vertex in the diagram spanned by
Quillen equivalences. They localize to the maximal sub $\infty$-groupoid
$\Pr\SM^{\simeq}\subset\Pr\SM$ via the underlying symmetric monoidal
$\infty$-category functor.
\item Assertion (3) remains valid if we replace each vertex by the subcategory
consisting of those objects with a cofibrant unit and Quillen equivalences
between them.
\end{enumerate}
\end{thm}

\begin{rem}
\label{rem:informally}Informally, part (3) of Theorem \ref{thm:var_SM}
asserts the following: 
\begin{itemize}
\item presentably symmetric monoidal $\infty$-categories are presented
by combinatorial symmetric monoidal model categories;
\item such presentations are unique up to zig-zags of Quillen equivalences;
\item such zig-zags are themselves unique; and so on.
\end{itemize}
More precisely, for every $\cat C\in\Pr\SM$, the $\infty$-category
\[
\CSMMC^{QE}\times_{\Pr\SM^{\simeq}}\pr{\Pr\SM^{\simeq}}_{/\cat C}
\]
is weakly contractible, where $\CSMMC^{QE}\subset\CSMMC$ denotes
the subcategory of Quillen equivalences. This follows from the fact
that a functor into an $\infty$-groupoid is a localization if and
only if it is final \cite[Proposition 7.1.10]{HCHA}.

The same remark applies to combinatorial simplicial symmetric monoidal
model categories. It follows that every combinatorial symmetric monoidal
model category is symmetric-monoidally Quillen equivalent to a simplicial
one. This gives a solution to a special case of Hovey's 10th problem
\cite[Problem 10]{Hoveylist_model}, which asks whether every monoidal
model category is monoidally Quillen equivalent to a simplicial one.
\end{rem}

Distressingly, some of our methods so far do not extend to the monoidal
setting. Because of this, the monoidal version of Theorem \ref{thm:var_SM}
takes the following form:
\begin{thm}
\label{thm:var_Mon}Consider the diagram of categories introduced
in Section \ref{sec:definition}:
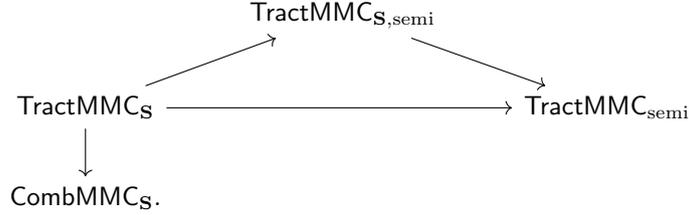
\[\begin{tikzcd}
	& {\mathsf{TractMMC}_{\mathbf{S},\mathrm{semi}}} \\
	{\mathsf{TractMMC}_{\mathbf{S}}} && {\mathsf{TractMMC}_{\mathrm{semi}}} \\
	{\mathsf{CombMMC}_{\mathbf{S}}.}
	\arrow[from=1-2, to=2-3]
	\arrow[from=2-1, to=1-2]
	\arrow[from=2-1, to=2-3]
	\arrow[from=2-1, to=3-1]
\end{tikzcd}\]
\begin{enumerate}
\item All of the vertices in the diagram localize to $\Pr\Mon$ via the
underlying monoidal $\infty$-category functor.
\item Assertion (1) remains valid if we replace each vertex by the full
subcategory consisting of those objects with a cofibrant unit.
\item Consider the subcategories of each vertex in the diagram spanned by
Quillen equivalences. They localize to the maximal sub $\infty$-groupoid
$\Pr\Mon^{\simeq}\subset\Pr\Mon$ via the underlying monoidal $\infty$-category
functor.
\item Assertion (3) remains valid if we replace each vertex by the subcategory
consisting of those objects with a cofibrant unit and Quillen equivalences
between them.
\end{enumerate}
\end{thm}

\begin{rem}
\label{rem:Hovey}Part (3) of Theorem \ref{thm:var_Mon} implies that
every tractable monoidal semi-model is monoidally Quillen equivalent
(as a semi-model category) to a tractable simplicial monoidal model
category. This gives another partial solution to Hovey's 10th problem.
\end{rem}

We now turn to the proof of these theorems. A key step in the proof
is a functorial replacement of semi-monoidal model categories by simplicial
ones. The basic idea is that of Dugger \cite{Dug01c}, who observed
that if $\mathbf{M}$ is nice model category, the category $s\mathbf{M}$
of simplicial objects can be made into a simplicial model category
by taking the Bousfield localization of the Reedy model structure.
In the situation at hand, his idea does not apply directly, so we
need a bit more work. 

We will start by recalling Dugger's idea. Most of the following lemma
can be found in \cite{Dug01c}. We present a slightly more streamlined
proof than loc. cit.
\begin{lem}
\label{lem:Dugger}Let $\mathbf{M}$ be a tractable semi-model category.
The Reedy semi-model category $s\mathbf{M}_{\Reedy}$ of simplicial
objects in $\mathbf{M}$ admits a tractable left Bousfield localization,
denoted by $s\mathbf{M}_{\Reedy,\loc}$, with the following properties:
\begin{enumerate}
\item Fibrant objects of $s\mathbf{M}_{\Reedy,\loc}$ are the Reedy fibrant
simplicial objects $X$ such that the map $X_{0}\to X_{n}$ is a weak
equivalence for all $n\geq0$.
\item The diagonal functor $\delta\from\mathbf{M}\to s\mathbf{M}_{\Reedy,\loc}$
is a right Quillen equivalence.
\item $s\mathbf{M}_{\Reedy,\loc}$ is a $\SS$-enriched semi-model category.
\end{enumerate}
\end{lem}

\begin{proof}
Consider the functor 
\[
\pr{\mathbf{M}^{\Del^{\op}}}_{\infty}\simeq\pr{\mathbf{M}_{\infty}}^{\Del^{\op}}\xrightarrow{\colim}\mathbf{M}_{\infty},
\]
where the first equivalence is that of \cite[Theorem 7.9.8]{HCHA}.
Since $\Del^{\op}$ is weakly contractible, the right adjoint of this
functor is fully faithful \cite[Corollary 4.4.4.10]{HTT}. Thus, applying
Proposition \ref{prop:semi_bousfield} to this functor, we obtain
a tractable semi-model structure $s\mathbf{M}_{\Reedy,\loc}$ on $s\mathbf{M}$
satisfying conditions (1) and (2). 

For (3), we must describe the simplicial enrichment of $s\mathbf{M}$.
If $X$ and $Y$ are simplicial objects in $s\mathbf{M}$, we define
the mapping simplicial object $\Map\pr{X,Y}$ as follows: Given a
simplicial set $S$, we define a new simplicial object $S\otimes X\in s\mathbf{M}$
by $\pr{S\otimes X}_{n}=\coprod_{S_{n}}X_{n}$. We then define $\Map\pr{X,Y}_{n}=\Hom_{s\mathbf{M}}\pr{\Delta^{n}\otimes X,Y}$.
This makes $s\mathbf{M}$ into a simplicial category.

We must show that with this enrichment, $s\mathbf{M}_{\Reedy,\loc}$
is an $\SS$-enriched model category. Let $\pr -^{S}\from s\mathbf{M}\to s\mathbf{M}$
denote the right adjoint of the functor $S\otimes-$. We must prove
the following: 

\begin{enumerate}[label=(\roman*)]

\item If $i\from K\to L$ is a cofibration in $\SS$ and $f\from X\to Y$
is a Reedy cofibration of Reedy cofibrant objects of $s\mathbf{M}$,
then the induced map
\[
i\square f\from\pr{K\otimes Y}\amalg_{K\otimes X}\pr{L\otimes X}\to L\otimes Y
\]
is a Reedy cofibration.

\item In the situation of (i), if further $f$ is a weak equivalence
of $s\mathbf{M}_{\Reedy,\loc}$, then so is $i\square f$.

\item In the situation of (i), if further $i$ is anodyne, then $i\square f$
is a weak equivalence of $s\mathbf{M}_{\Reedy,\loc}$.

\end{enumerate}

To prove these assertions, choose generating sets $I$ and $J$ of
cofibrations and trivial cofibrations of $\mathbf{M}$ whose elements
have a cofibrant domain. Also, for each $n\geq0$, write $i_{n}\from\partial\Delta^{n}\hookrightarrow\Delta^{n}$
for the inclusion. The class of cofibrations trivial cofibrations
of Reedy semi-model structure on $s\mathbf{M}$ is generated by the
sets $\widetilde{I}=\{i_{n}\square e\mid n\geq0,\,e\in I\}$ and $\widetilde{J}=\{i_{n}\square e\mid n\geq0,\,e\in J\}$,
which are maps of cofibrant objects (see \cite[Lemma 3.34]{Bar10}).
Here we identified maps in $\mathbf{M}$ with those in $s\mathbf{M}$
via the diagonal functor $\mathbf{M}\to s\mathbf{M}$. In order to
prove (i), we may assume that $f$ is an element of $\widetilde{I}$,
say $f=i_{n}\square e$. In this case, the claim is clear, because
$i\square i_{n}$ is a cofibration of simplicial sets and the maps
$\{i_{m}\}_{m\geq0}$ generate cofibrations of simplicial sets.

For (ii), by two out of three and part (i), it suffices to show that
the maps $K\otimes f$ and $L\otimes f$ are weak equivalences. We
will show that $K\otimes f$ is one; the proof for $L\otimes f$ is
similar. By part (i) and \ref{rem:7.14_JT07}, this will follow if
we show that $\pr -^{K}\from s\mathbf{M}_{\Reedy,\loc}\to s\mathbf{M}_{\Reedy,\loc}$
preserves fibrations of fibrant objects. The explicit generating set
$\widetilde{J}$ makes it clear that $\pr -^{K}$ carries Reedy fibrations
to Reedy fibrations. Since a map of fibrant objects of $s\mathbf{M}_{\Reedy,\loc}$
is a fibration if and only if it is a Reedy fibration (Proposition
\ref{prop:3.3.16_Hir}), we are therefore reduced to showing that
$\pr -^{K}$ preserves fibrant objects of $s\mathbf{M}_{\Reedy,\loc}$.

Let $X\in s\mathbf{M}_{\Reedy,\loc}$ be fibrant. We already know
that $X^{K}$ is Reedy fibrant, so it suffices to show that the map
$\theta\from\pr{X^{K}}_{0}\to\pr{X^{K}}_{n}$ is a weak equivalence
for all $n$. For this, set $\hom\pr{S,X}=\lim_{\Delta^{n}\to S\in\pr{\Del_{/S}}^{\op}}X_{n}$
for each simplicial set $S$. We can then identify $\theta$ with
the map $\hom\pr{K\times\Delta^{0},X}\to\hom\pr{K\times\Delta^{n},X}$.
To prove that this is a weak equivalence, it suffices to show that
the map
\[
\hom\pr{\Delta^{0},X}\to\hom\pr{\Delta^{m},X}
\]
is a weak equivalence (see the argument of \cite[Proposition 4.6.1]{HCHA})
for every $m\geq0$. This follows from our assumption on $X$.

Finally, for (iii), by part (i) and \ref{rem:7.14_JT07}, it suffices
to show that for each fibration $p\from E\to B$ of fibrant objects
of $s\mathbf{M}_{\Reedy,\loc}$, the map
\[
i\pitchfork p\from E^{L}\to E^{K}\times_{B^{K}}B^{L}
\]
is a trivial fibration. We already know from (ii) that this is a fibration,
so it suffices to show that the map $E^{L}\to E^{K}$ and $B^{L}\to B^{K}$
are weak equivalences. Using (the argument of) \cite[Proposition 4.6.1]{HCHA},
it will suffice to prove this in the case where $i$ is a map of the
form $\Delta^{0}\to\Delta^{n}$. In this case, the claim is proved
as in the previous paragraph, using the fact that $E$ and $B$ are
fibrant in $s\mathbf{M}_{\Reedy,\loc}$.
\end{proof}

Unfortunately, when $\mathbf{M}$ is a monoidal semi-model category,
it seems difficult to compare $s\mathbf{M}_{\Reedy,\loc}$ with $\mathbf{M}$
via a monoidal left Quillen functor. The problem is that the diagonal
functor $\mathbf{M}\to s\mathbf{M}$ is rarely left Quillen. Thus,
what we are going to do is to increase the class of cofibrations to
force this to be true.

To state the next lemma, it will be useful to introduce the following
terminology. 
\begin{defn}
We say that a semi-model category $\mathbf{M}$ is \textbf{$\kappa$-tractable}
if it satisfies the following pair of conditions:
\begin{enumerate}
\item The underlying category of $\mathbf{M}$ is locally $\kappa$-presentable.
\item The (essentially small) class of cofibrations of $\kappa$-compact
cofibrant objects of $\mathbf{M}$ generates the class of cofibrations
of $\mathbf{M}$.
\end{enumerate}
\end{defn}

\begin{lem}
\label{lem:injloc}Let $\kappa$ be an uncountable regular cardinal,
and let $\mathbf{M}$ be a $\kappa$-tractable semi-model category.
There is a $\kappa$-tractable semi-model structure on $s\mathbf{M}$,
denoted by $s\mathbf{M}_{\kappa\-\inj,\loc}$ with the following properties:
\begin{enumerate}
\item Cofibrations are generated by the maps $X\to Y$ such that, for each
$n\geq0$, the map $X_{n}\to Y_{n}$ is a cofibration of $\kappa$-compact
cofibrant objects of $\mathbf{M}$.
\item Weak equivalences are the hocolim-equivalences, i.e., maps inducing
weak equivalences between the homotopy colimits. 
\item The diagonal functor $\delta\from\mathbf{M}\to s\mathbf{M}_{\inj,\loc}$
is a left Quillen equivalence.
\item $s\mathbf{M}_{\inj,\loc}$ is a $\SS$-enriched semi-model category.
\item If $\mathbf{M}$ is a monoidal semi-model category, then $s\mathbf{M}_{\inj,\loc}$
is also one for the degreewise tensor product.
\end{enumerate}
\end{lem}

\begin{proof}
We first construct a $\kappa$-tractable semi-model structure satisfying
conditions (1) through (3), using Theorem \ref{thm:Smith}. (Note
that, since $\kappa$ is uncountable, $\kappa$-compact objects of
$s\mathbf{M}$ are simply the levelwise $\kappa$-compact simplicial
objects \cite[Proposition 2.23]{ZhenLin_Universe}.) Call a map $X\to Y$
of $s\mathbf{M}$ an \textbf{$\kappa$-injective cofibration} if it
is a cofibration in the sense of condition (1). We also say that a
map of $s\mathbf{M}$ is a \textbf{$\kappa$-injective fibration}
if it has the right lifting property for $\kappa$-injective cofibrations.
We must verify the following:

\begin{enumerate}[label=(\roman*)]

\item The full subcategory of $\Fun\pr{[1],s\mathbf{M}}$ spanned
by the hocolim-equivalences is an accessible subcategory.

\item Every $\kappa$-injective fibration is a hocolim-equivalence.

\item Given a pushout diagram 
\[\begin{tikzcd}
	X & {X'} \\
	Y & {Y'}
	\arrow[from=1-1, to=1-2]
	\arrow["i"', from=1-1, to=2-1]
	\arrow["{i'}", from=1-2, to=2-2]
	\arrow[from=2-1, to=2-2]
\end{tikzcd}\]of $\kappa$-injectively cofibrant objects, if $i$ is a hocolim-equivalence
and a $\kappa$-injective cofibration, so is $i'$.

\item Given a limit ordinal $\alpha$ and a transfinite sequence
$X_{\bullet}\from\alpha\to s\mathbf{M}$ of $\kappa$-injective cofibrations
of $\kappa$-injectively cofibrant objects, if the map $X_{\beta}\to X_{\beta+1}$
is a hocolim-equivalence for each $\beta<\alpha$, then so is the
map $X_{0}\to\colim_{\beta<\alpha}X_{\beta}$.

\end{enumerate}

For part (i), by \cite[Proposition 5.4.6.6]{HTT}, it suffices to
show that the composite
\[
\mathbf{M}^{\Del^{\op}}\to\pr{\mathbf{M}^{\Del^{\op}}}_{\infty}\simeq\pr{\mathbf{M}_{\infty}}^{\Del^{\op}}\xrightarrow{\colim}\mathbf{M}_{\infty}
\]
is accessible. The first functor is accessible by Corollary \ref{cor:RR15_3.1},
and the second functor is accessible because it is a left adjoint.
Thus, their composite is accessible.

For part (ii), let $p\from X\to Y$ be a $\kappa$-injective fibration
of $s\mathbf{M}$. We claim that $p$ is levelwise a weak equivalence
(hence a hocolim-equivalence). Since $\mathbf{M}$ is $\kappa$-tractable,
it suffices to show that every lifting problem as depicted in the
left-hand diagram below admits a solution, provided that $i$ is a
cofibration of $\kappa$-compact objects of $\mathbf{M}$:
\[\begin{tikzcd}
	A & {X_n} && {\Delta^n\otimes \delta(A)} & X \\
	B & {Y_n} && {\Delta ^n\otimes \delta (B)} & Y.
	\arrow[from=1-1, to=1-2]
	\arrow["i"', from=1-1, to=2-1]
	\arrow["p", from=1-2, to=2-2]
	\arrow[from=1-4, to=1-5]
	\arrow["{\Delta^n\otimes \delta (i)}"', from=1-4, to=2-4]
	\arrow["p", from=1-5, to=2-5]
	\arrow[dashed, from=2-1, to=1-2]
	\arrow[from=2-1, to=2-2]
	\arrow[dashed, from=2-4, to=1-5]
	\arrow[from=2-4, to=2-5]
\end{tikzcd}\]The lifting problem is equivalent to the one on the right, which is
solvable because $\Delta^{n}\otimes\delta\pr i$ is a $\kappa$-injective
cofibration. 

Par (iii) is clear, because such squares are homotopy cocartesian
in $\mathbf{M}^{\Del^{\op}}$ (as it is levelwise homotopy cocartesian).
Part (iv) is proved similarly.

Next, we prove (4). Let $i\from K\to L$ be a monomorphism of simplicial
sets, and let $f\from X\to Y$ be a generating $\kappa$-injective
cofibration $s\mathbf{M}_{\kappa\-\inj,\loc}$. We must show that
the map
\[
i\square f\from\pr{K\otimes Y}\amalg_{K\otimes X}\pr{L\otimes X}\to L\otimes Y
\]
is a $\kappa$-injective cofibration, and that it is a weak equivalence
if either $i$ or $f$ is one. For the claim on cofibration, it suffices
to prove this in the case where $i$ is the boundary inclusion $\partial\Delta^{n}\subset\Delta^{n}$
for some $n\geq0$. In this case, $i\square f$ is another generating
$\kappa$-injective cofibration, so there is nothing to prove. For
the acyclic part of the claim, find a commutative diagram 
\[\begin{tikzcd}
	{X'} & {Y'} \\
	X & Y
	\arrow["{f'}", from=1-1, to=1-2]
	\arrow["p"', from=1-1, to=2-1]
	\arrow["q", from=1-2, to=2-2]
	\arrow["f", from=2-1, to=2-2]
\end{tikzcd}\]where $p$ and $q$ are trivial Reedy fibrations, $X'$ is Reedy cofibrant,
and $f'$ is a Reedy cofibration. We then have a diagram 
\[\begin{tikzcd}
	{(K\otimes Y')\amalg_{K\otimes X'}(L\otimes X')} & {L\otimes Y'} \\
	{(K\otimes Y)\amalg_{K\otimes X}(L\otimes X)} & {L\otimes Y}
	\arrow[from=1-1, to=1-2]
	\arrow[from=1-1, to=2-1]
	\arrow[from=1-2, to=2-2]
	\arrow[from=2-1, to=2-2]
\end{tikzcd}\]whose vertical arrows are levelwise weak equivalences (because the
pushouts are levelwise homotopy pushouts). Thus, it suffices to show
that $i\square f'$ is a weak equivalence. This follows from Lemma
\ref{lem:Dugger}.

Finally, we prove (5). Suppose that $\mathbf{M}$ is a monoidal model
category. We must show that $s\mathbf{M}_{\kappa\-\inj,\loc}$ is
also a monoidal model category. The only nontrivial part is the acyclic
part of the pushout-product axiom. For this, it suffices to show that
for each cofibrant object $X\in s\mathbf{M}_{\kappa\-\inj,\loc}$,
the functors $X\otimes-,\,-\otimes X\from s\mathbf{M}_{\kappa\-\inj,\loc}\to s\mathbf{M}_{\kappa\-\inj,\loc}$
preserve weak equivalences of cofibrant objects. This follows from
the fact that $\Del^{\op}$ is a sifted $\infty$-category \cite[Lemma 5.5.8.4]{HTT},
because this guarantees that $\hocolim_{\Del^{\op}}\pr{X\otimes Y}\simeq\pr{\hocolim_{\Del^{\op}}X}\otimes\pr{\hocolim_{\Del^{\op}}Y}$
whenever $X,Y\in s\mathbf{M}$ are levelwise cofibrant. The proof
is now complete.
\end{proof}

\begin{proof}
[Proof of Theorems \ref{thm:var_SM} and \ref{thm:var_Mon}]We consider
the symmetric monoidal case; the monoidal case is similar. We start
with Part (1). We treat each case separately:
\begin{itemize}
\item $\CSMMC$: This is the content of Theorem \ref{thm:main}.
\item $\TSMMC$: The proof of Theorem \ref{thm:main} goes through by replacing
$\CSMMC$ by $\TSMMC$.
\item $\TSMMC_{\mathbf{S}}$ and $\CSMMC_{\mathbf{S}}$ and $\TSMMC_{\mathbf{S},\semi}$:
The proof of Proposition \ref{prop:main_flat} goes through by replacing
$\CombSymAlg{\mathbf{S}}$ by these categories. (In the semi-model
categorical case, we replace various model-categorical results by
the results in Section \ref{sec:semi}. Also, in this case, we only
need $\mathbf{S}$ to be a left proper tractable symmetric monoidal
semi-model category.) Thus, the claim for these categories follows
from Proposition \ref{prop:main_L}.
\item $\TSMMC_{\semi}$: By the previous point, it suffices to show that
the forgetful functor 
\[
\TSMMC_{\SS,\semi}\to\TSMMC_{\semi}
\]
induces an equivalence upon localizing at Quillen equivalences. 

For each uncountable regular cardinal $\kappa$, we let $\TSMMC_{\semi}\pr{\kappa}\subset\TSMMC_{\semi}$
denote the subcategory spanned by the objects whose underlying semi-model
categories are $\kappa$-tractable, and the morphisms that preserve
$\kappa$-compact objects. We define $\TSMMC_{\SS,\semi}\pr{\kappa}$
similarly. Lemma \ref{lem:injloc} gives us a functor 
\begin{align*}
\Phi\from\int^{\urCard}\TSMMC_{\semi}\pr - & \to\TSMMC_{\SS,\semi}\pr{\kappa},\\
\pr{\kappa,\mathbf{M}} & \mapsto s\mathbf{M}_{\kappa\-\inj,\loc},
\end{align*}
which fits into the (non-commutative) diagram 
\[\begin{tikzcd}
	{\int^{\mathsf{urCard}}\mathsf{TractSMMC}_{\mathsf{sSet},\mathrm{semi}}(-)} & {\int^{\mathsf{urCard}}\mathsf{TractSMMC}_{\mathrm{semi}}(-)} \\
	{\mathsf{TractSMMC}_{\mathsf{sSet},\mathrm{semi}}} & {\mathsf{TractSMMC}_{\mathrm{semi}}.}
	\arrow["U", from=1-1, to=1-2]
	\arrow["\pi"', from=1-1, to=2-1]
	\arrow["\Phi"{description}, from=1-2, to=2-1]
	\arrow["{\pi'}", from=1-2, to=2-2]
	\arrow["{U'}"', from=2-1, to=2-2]
\end{tikzcd}\]Here $\pi,\pi',U,U'$ are all forgetful functors. Since $\TSMMC_{\SS,\semi}$
is the colimit of $\TSMMC_{\SS,\semi}\pr{\kappa}$ as $\kappa$ ranges
over $\urCard$, the functor $\pi$ and $\pi'$ are localizations
\cite[\href{https://kerodon.net/tag/02UU}{Tag 02UU}]{kerodon}. Thus,
it suffices to show that the diagram commutes up to equivalence when
we localize these categories by Quillen equivalences.
\end{itemize}
By Lemma \ref{lem:injloc}, the lower triangle commutes up to natural
Quillen equivalence. For the upper triangle, recall from the previous
bullet point that the composite
\[
\TSMMC_{\SS,\semi}\xrightarrow{U'}\TSMMC_{\semi}\to\Pr\SM
\]
is a localization at Quillen equivalences. We also have a natural
Quillen equivalence $U'\pi\xrightarrow{\simeq}U'\Phi U$ by Lemma
\ref{lem:injloc}, so this proves the claim on the lower triangle.

Part (2) follows from part (1) and Muro's theorem \cite[Theorem 1]{Mur15}
or its enriched version (Proposition \ref{prop:Mur15}). Parts (3)
and (4) are proved similarly, noting that the proof of Theorem \ref{thm:main}
is unaffected if we restrict our attention to the relevant subcategories.
This completes the proof of the theorem.
\end{proof}

\appendix

\section{\label{sec:semi}Review of semi-model categories}

Semi-model categories are a generalization of model categories, first
introduced by Hovey and Spitzweck \cite{hovey1998monoidalmodelcategories, spitzweck2001operadsalgebrasmodulesgeneral}
and later developed further by Barwick \cite{Bar10}. (See \cite{BW24}
for a historical account and examples of semi-model categories.) For
the most part, they behave very much like (and sometimes better than)
model categories. There are, however, some differences, and some results
about model-categories will be false without making some modifications
to their statements. In this section, we review the definition of
semi-model categories and give precise statements of results on semi-model
categories that directly generalize the model-categorical counterparts.

\subsection{Definition}

In this subsection, we recall the definition of semi-model categories.
\begin{defn}
Let $\mathbf{M}$ be a locally presentable category, and let $I$
be a set of morphisms of $\mathbf{M}$. An \textbf{$I$-cofibration}
is a retract of a transfinite composition of pushouts of elements
of $I$. An \textbf{$I$-injective} is a morphism having the right
lifting property for the maps in $I$. We write $I\-\cof$ and $I\-\inj$
for the classes of $I$-cofibrations and $I$-injectives, respectively.
\end{defn}

\begin{defn}
\cite[Definition 2.1]{BW24} A \textbf{semi-model structure} on a
category $\mathbf{M}$ consists of classes $\pr{W,F,C}$ of weak equivalences,
fibrations, and cofibrations, satisfying the following axioms:

\begin{enumerate}[label=M\arabic*]

\item Fibrations are closed under pullback.

\item Weak equivalences have the two out of three property.

\item The classes $W,F,C$ contain isomorphisms and is closed under
composition and retracts.

\item 

\begin{enumerate}[label=\roman*]

\item Cofibrations have the left lifting property for trivial fibrations
(i.e., elements of $W\cap F$).

\item Trivial cofibrations with cofibrant domains have the left lifting
property for fibrations. (Here, trivial cofibrations refer to elements
in $C\cap W$, and cofibrant objects are those objects such that the
map from the initial object is a cofibration.)

\end{enumerate}

\item 

\begin{enumerate}[label=\roman*]

\item Every morphism in $\mathbf{M}$ can be functorially factored
as a cofibration followed by a trivial fibration.

\item Every morphism whose domain is cofibrant can be functorially
factored as a trivial cofibration followed by a fibration.

\end{enumerate}

\end{enumerate}

A\textbf{ semi-model category} is a finitely bicomplete\footnote{In \cite{BW24}, bicompleteness is assumed, but we drop this to have
a semi-model categorical version of the arguments in \ref{subsec:strongkappa}.} category $\mathbf{M}$ equipped with a semi-model structure. Its
\textbf{underlying $\infty$-category} is the ($\infty$-categorical)
localization at weak equivalences.

We say that a semi-model category $\mathbf{M}$ is \textbf{combinatorial}
if it satisfies the following pair of conditions:
\begin{itemize}
\item $\mathbf{M}$ is locally presentable as a category.
\item There are (small) sets $I,J$ of morphisms of $\mathbf{M}$, such
that the classes of fibrations and trivial fibrations are exactly
the classes of $J$-injectives and $I$-injectives. Elements of the
sets $I$ and $J$ are called \textbf{generating cofibrations} and
\textbf{generating trivial cofibrations}.
\end{itemize}
If further the maps in $I$ and $J$ have cofibrant domains, then
we say that $\mathbf{M}$ is \textbf{tractable}.
\end{defn}

\begin{rem}
\label{rem:lifting}Let $\mathbf{M}$ be a semi-model category, and
let $f$ be a morphism of $\mathbf{M}$. By the retract argument \cite[Lemma 1.1.9]{Hovey},
one has the following:
\begin{itemize}
\item If $f$ is a map of cofibrant objects and has the left lifting property
for fibrations, then it is a trivial cofibration.
\item $f$ is a cofibration if and only if it has the left lifting property
for trivial fibrations.
\item $f$ is a trivial fibration if and only if it has the right lifting
property for cofibrations.
\end{itemize}
This implies the following:
\begin{itemize}
\item A semi-model structure is completely determined by the class of fibrations
and weak equivalences. 
\item In a tractable model category $\mathbf{M}$ with generating sets $I,J$
of cofibrations and trivial cofibrations with cofibrant domains, the
maps in $I$ are cofibrations, and the maps in $J$ are trivial cofibrations.
\end{itemize}
\end{rem}

\begin{rem}
\label{rem:7.14_JT07} Let $\mathbf{M}$ be a semi-model category.
The argument of \cite[Lemma 7.14]{JT07} shows that a cofibration
of cofibrant objects of $\mathbf{M}$ is a trivial cofibration if
and only if it has the left lifting property for fibrations of fibrant
objects.
\end{rem}

\begin{defn}
A functor $F\from\mathbf{M}\to\mathbf{N}$ of tractable semi-model
categories is called a \textbf{left Quillen functor} $\mathbf{M}\to\mathbf{N}$
if it satisfies the following conditions:
\begin{enumerate}
\item $F$ is a left adjoint.
\item $F$ carries cofibrations of cofibrant objects to cofibrations.
\item $F$ carries trivial cofibrations of cofibrant objects to trivial
cofibrations.
\end{enumerate}
Equivalently, $F$ is left Quillen if its right adjoint preserves
fibrations and trivial fibrations. 
\end{defn}

\begin{rem}
\label{rem:und_semi}Let $\mathbf{M}$ be a semi-model category. By
\cite[Proposition 7.7.4]{HCHA}, the underlying $\infty$-category
of $\mathbf{M}$ has small colimits. If $\mathbf{M}$ is tractable,
then the semi-model categorical version of Proposition \ref{prop:strongly_comb_ind}
shows that the underlying $\infty$-category of $\mathbf{M}$ is presentable.
\end{rem}

\subsection{Reedy semi-model structure}

Given a model category $\mathbf{M}$ and a Reedy category $\cat C$,
the functor category $\Fun\pr{\cat C,\mathbf{M}}$ admits a model
structure, called the Reedy model structure. In this subsection, we
will see that this directly generalizes to the semi-model categorical
setting (Proposition \ref{prop:Reedy}). As a consequence of this,
we will obtain a formula for the mapping space of the underlying $\infty$-category
of a semi-model category using (co)simplicial resolutions (Corollary
\ref{cor:ACK25}).
\begin{prop}
[Reedy semi-model structure]\label{prop:Reedy} Let $\mathbf{M}$
be a bicomplete semi-model category, and let $\cat C$ be a Reedy
category. There is a semi-model structure on $\Fun\pr{\cat C,\mathbf{M}}$,
called the \textbf{Reedy semi-model structure}, with the following
properties:
\begin{enumerate}
\item Weak equivalences are the natural weak equivalences.
\item A map $\alpha\from F\to G$ is a fibration (resp. trivial fibration)
if and only if, for each $C\in\cat C$, the map
\[
FC\to GC\times_{M_{C}G}M_{C}F
\]
is a fibration (resp. trivial fibration) of $\mathbf{M}$. Here, $M_{C}F$
denotes the matching object of $F$ at $C$.
\item A map $\alpha\from F\to G$ is a cofibration if and only if, for each
$C\in\cat C$, the map
\[
L_{C}G\amalg_{L_{C}F}FC\to GC
\]
is a cofibration of $\mathbf{M}$. Here, $L_{C}F$ denotes the latching
object of $F$ at $C$.
\item Given a cofibrant $F\in\mathbf{M}^{\cat I}$, a map $\alpha\from F\to G$
is a trivial cofibration if and only if, for each $C\in\cat C$, the
map
\[
L_{C}G\amalg_{L_{C}F}FC\to GC
\]
is a trivial cofibration of $\mathbf{M}$.
\end{enumerate}
Moreover, if $\mathbf{M}$ is tractable, so is the Reedy semi-model
structure on $\Fun\pr{\cat C,\mathbf{M}}$.
\end{prop}

\begin{proof}
See the proof of \cite[Theorems 15.3.4 and 15.3.15]{Hirschhorn} or
\cite[Theorem 5.2.5]{Hovey} for the construction of the semi-model
structure. For tractability, see \cite[Lemma 3.34]{Bar10}.
\end{proof}

We are particularly interested in the Reedy semi-model structure for
(co)simplicial objects, as it can be used to model derived mapping
spaces (Corollary \ref{cor:ACK25}.) If $\mathbf{M}$ is a tractable
semi-model category, a \textbf{simplicial resolution} in $\mathbf{M}$
is a Reedy fibrant simplicial object $X_{\bullet}$ such that the
map $X_{0}\to X_{n}$ is a weak equivalence for every $n\geq0$. Dually,
a cosimplicial resolution in $\mathbf{M}$ is a Reedy cofibrant cosimplicial
object $X^{\bullet}$ such that the map $X^{n}\to X^{0}$ is a weak
equivalence for every $n\geq0$. We then have the following proposition:
\begin{prop}
\label{prop:16.5.2}Let $\mathbf{M}$ be a semi-model category.
\begin{enumerate}
\item Let $i\from A\to B$ be a cofibration of cofibrant objects of $\mathbf{M}$,
and let $p\from X_{\bullet}\to Y_{\bullet}$ be a Reedy fibration
of simplicial resolutions of $\mathbf{M}$. The map
\[
\mathbf{M}\pr{B,X_{\bullet}}\to\mathbf{M}\pr{B,Y_{\bullet}}\times_{\mathbf{M}\pr{A,Y_{\bullet}}}\mathbf{M}\pr{A,X_{\bullet}}
\]
is a Kan fibration, which is a trivial fibration if $i$ or $p$ is
a weak equivalence.
\item Let $i\from A^{\bullet}\to B^{\bullet}$ be a Reedy cofibration of
cosimplicial resolutions in $\mathbf{M}$, and let $p\from X\to Y$
be a fibration of $\mathbf{M}$. The map
\[
\mathbf{M}\pr{B^{\bullet},X}\to\mathbf{M}\pr{B^{\bullet},Y}\times_{\mathbf{M}\pr{A^{\bullet},Y}}\mathbf{M}\pr{A^{\bullet},X}
\]
is a Kan fibration, which is a trivial fibration if $i$ or $p$ is
one.
\end{enumerate}
\end{prop}

\begin{proof}
Identical to \cite[Theorem 16.5.2]{Hirschhorn}.
\end{proof}

\begin{cor}
\label{cor:ACK25}Let $\mathbf{M}$ be a semi-model category, let
$L\from\mathbf{M}\to\cat M$ be the localization at weak equivalences,
and let $X_{\bullet}$ be a simplicial resolution in $\mathbf{M}$.
There is a natural equivalence
\[
\mathbf{M}\pr{A,X_{\bullet}}\simeq\cat M\pr{L\pr A,L\pr{X_{0}}}
\]
of functors $\mathbf{M}^{\op}_{\cof}\to\SS[\weq^{-1}]\simeq\cat S$,
which carries a vertex $A\to X_{0}\in\mathbf{M}\pr{A,X_{\bullet}}$
to its image in $\cat M$.
\end{cor}

\begin{proof}
Identical to \cite[Corollary 3.1]{ACK25}.
\end{proof}

\begin{cor}
\label{cor:ACK25_1}Let $\mathbf{M}$ be a semi-model category. The
weak equivalences of $\mathbf{M}$ are exactly the maps whose images
in the underlying $\infty$-category $\cat M$ are equivalences. Moreover,
if two objects in $\mathbf{M}$ have equivalent images in $\cat M$,
then they are connected by a zig-zag of weak equivalences.
\end{cor}

\begin{proof}
This is immediate from \ref{cor:ACK25}.
\end{proof}

\begin{cor}
\label{cor:7.9.8}Let $\mathbf{M}$ be a semi-model category, and
let $\cat C$ be a small category. Postcomposing with the functor
$\mathbf{M}\to\mathbf{M}[\weq^{-1}]$ induces an equivalence of $\infty$-categories
\[
\Fun\pr{\cat C,\mathbf{M}}[\weq^{-1}]\xrightarrow{\simeq}\Fun\pr{\cat C,\mathbf{M}[\weq^{-1}]}.
\]
\end{cor}

\begin{proof}
This follows from \cite[Theorem 7.9.8]{HCHA} and Corollary \ref{cor:ACK25_1}.
\end{proof}

Corollary \ref{cor:ACK25} has another important consequence:
\begin{prop}
\label{prop:A.2.3.1}Let $\mathbf{M}$ be a semi-model category, and
let $L\from\mathbf{M}\to\cat M$ be the localization at weak equivalences.
Consider an extension problem in $\mathbf{M}$ depicted as 
\[\begin{tikzcd}
	A & X. \\
	B
	\arrow["f", from=1-1, to=1-2]
	\arrow["i"', from=1-1, to=2-1]
	\arrow[dashed, from=2-1, to=1-2]
\end{tikzcd}\]Suppose that:
\begin{enumerate}
\item $i$ is a cofibration of cofibrant objects.
\item $X$ is fibrant. 
\item There is a morphism $g\from B\to X$ and a homotopy $L\pr g\circ L\pr i\simeq L\pr f$
in $\cat M$.
\end{enumerate}
Then there is a map $g'\from B\to X$ such that $g'i=f$ in $\mathbf{M}$
and a homotopy $L\pr g\simeq L\pr{g'}$ in $\cat M$.
\end{prop}

\begin{proof}
A closer inspection of the proof of Proposition \ref{prop:Reedy}
shows that the constant cosimplicial object $\delta\pr A$ at $A$
admits a trivial Reedy fibration $A^{\bullet}\to\delta\pr A$ from
a cosimplicial resolution, such that its $0$th degree is the identity
map of $A$. By the same token, the composite $A^{\bullet}\to\delta\pr A\to\delta\pr B$
can be factored as 
\[\begin{tikzcd}
	{A^\bullet} & {\delta(A)} \\
	{B^\bullet } & {\delta(B)}
	\arrow[from=1-1, to=1-2]
	\arrow["{i'}"', from=1-1, to=2-1]
	\arrow["{\delta(i)}", from=1-2, to=2-2]
	\arrow["p"', from=2-1, to=2-2]
\end{tikzcd}\]where $i'$ is a Reedy cofibration and $p$ is a Reedy fibration whose
$0$th degree is the identity map of $B$. By the dual of Corollary
\ref{cor:ACK25}, we can think of a homotopy $L\pr gL\pr i\simeq L\pr f$
as a path $\alpha\from gi\to f$ in the Kan complex $\mathbf{M}\pr{A^{\bullet},X}$.
By Proposition \ref{prop:16.5.2}, the map
\[
\mathbf{M}\pr{B^{\bullet},X}\to\mathbf{M}\pr{A^{\bullet},X}
\]
is a Kan fibration, so the path $\alpha$ can be lifted to a morphism
$g\to g'$ in $\mathbf{M}\pr{B^{\bullet},X}$. The map $g'$ has the
desired property.
\end{proof}

\subsection{Homotopy colimits in Semi-model categories}

Let $\mathbf{M}$ be a semi-model category, let $\cat C$ be a small
category, and let $F\from\cat C^{\rcone}\to\mathbf{M}$ be a diagram.
We say that $F$ is a \textbf{homotopy colimit diagram }if its image
in the underlying $\infty$-category of $\mathbf{M}$ is a colimit
diagram. In this subsection, we prove several results to detect homotopy
colimit diagrams (Proposition \ref{prop:hocolim} and \ref{cor:RR15_3.1}).

We start with the following result on cofibrantly generating semi-model
structures:
\begin{prop}
\label{prop:cofibrantgeneration}Let $\mathbf{M}$ be a locally presentable
category, and let $I$ and $J$ be sets of maps of $\mathbf{M}$.
Call $I$-cofibrations cofibrations, and suppose that the domains
of maps in $I$ and $J$ are cofibrant (i.e., the map from the initial
object is a cofibration). Let $W$ be a class of maps that is closed
under retracts and satisfies the two out of three property. Suppose
that:
\begin{enumerate}
\item $I\-\inj=J\-\inj\cap W$.
\item $\pr{J\-\cof}_{c}\subset I\-\cof\cap W$, where the subscript $c$
indicates the subclass of maps with cofibrant domains.
\end{enumerate}
Then $\mathbf{M}$ becomes a tractable semi-model category with generating
sets $I$ and $J$ and with weak equivalences $W$.
\end{prop}

\begin{proof}
Declare a map to be a fibration if it is a $J$-injective. We claim
that this, together with the cofibrations and weak equivalences given,
determines a semi-model structure on $\mathbf{M}$. Axioms M1 through
M3 hold trivially. Axiom M4-i holds by (1), and axiom M4-ii holds
by definition. Axiom M5-i and M5-ii hold by (1) and (2), respectively,
using the small object argument.
\end{proof}

\begin{cor}
[Projective semi-model structure]\label{cor:projective} Let $\mathbf{M}$
be a tractable semi-model category, and let $\cat C$ be a small category.
The functor category $\Fun\pr{\cat C,\mathbf{M}}$ admits a tractable
semi-model structure, called the \textbf{projective semi-model structure},
whose weak equivalences and fibrations are the natural transformations
whose components are weak equivalences and fibrations, respectively.
\end{cor}

\begin{proof}
Let $I,J$ be generating sets of cofibrations and trivial cofibrations
of $\mathbf{M}$ with cofibrant domains. Replacing $I$ by $I\cup J$,
we may assume that $J\subset I$. Consider the following sets of maps
of $\Fun\pr{\cat C,\mathbf{M}}$:
\[
I_{\cat C}=\{\cat C\pr{C,-}\cdot i\mid C\in\cat C,\,i\in I\},\,J_{\cat C}=\{\cat C\pr{C,-}\cdot j\mid C\in\cat C,j\in J\}.
\]
Note that $I_{\cat C}$-injectives are exactly the natural transformations
whose components are trivial fibrations, and $J_{\cat C}$-injectives
are the natural transformations whose components are fibrations. We
claim that the sets $I_{\cat C}$ and $J_{\cat C}$ and the class
of natural weak equivalences satisfy conditions (1) and (2) of Proposition
\ref{prop:cofibrantgeneration}. 

Condition (1) is clear from the characterization of $I_{\cat C}$-injectives
and $J_{\cat C}$-injectives we gave in the previous paragraph. For
(2), observe that the components of a $J_{\cat C}$-cofibration with
a cofibrant domain are trivial cofibrations. It follows that every
$J_{\cat C}$-cofibrations with a cofibrant domain is a natural weak
equivalence. Also, since $J_{\cat C}\subset I_{\cat C}$, every $J_{\cat C}$-cofibration
is an $I_{\cat C}$-cofibration. Hence condition (2) is satisfied,
and the proof is complete.
\end{proof}

Corollary \ref{cor:projective} allows us to detect homotopy colimit
diagrams in semi-model categories. 
\begin{prop}
\label{prop:hocolim}Let $\mathbf{M}$ be a tractable semi-model category
or a combinatorial model category, let $\cat C$ be a small category,
and let $F\from\cat C^{\rcone}\to\mathbf{M}$ be a colimit diagram.
If $F\vert\cat C$ is projectively cofibrant, then $F$ is a homotopy
colimit diagram.
\end{prop}

\begin{proof}
Let $i\from\cat C\hookrightarrow\cat C^{\rcone}$ denote the inclusion,
and consider the adjunction
\[
i_{!}\from\Fun\pr{\cat C,\mathbf{M}}\adj\Fun\pr{\cat C^{\rcone},\mathbf{M}}\from i^{*}
\]
between the left Kan extension functor $i_{!}$ and the restriction
functor $i^{*}$. By the theory of derived adjunctions \cite[Theorem 2.2.11]{cathtpy},
this adjunction induces an adjunction 
\[
\mathbb{L}i_{!}\from\ho\pr{\Fun\pr{\cat C,\mathbf{M}}}\adj\ho\pr{\Fun\pr{\cat C^{\rcone},\mathbf{M}}}\from\mathbb{R}i^{*}
\]
of homotopy categories. Explicitly, $\mathbb{R}i^{*}=\ho\pr{i^{*}}$
and $\mathbb{L}i_{!}=\ho\pr{i_{!}\circ Q}$, where $Q$ is a projective
cofibrant replacement functor of $\Fun\pr{\cat C,\mathbf{M}}$. (Note
that $i_{!}\circ Q$ preserves weak equivalences by Ken Brown's lemma
\cite[Corollary 7.4.14]{HCHA}.)

Now using Corollary \ref{cor:7.9.8}, we can rewrite the above adjunction
as
\[
\mathbb{L}i_{!}\from\ho\pr{\Fun\pr{\cat C,\cat M}}\adj\ho\pr{\Fun\pr{\cat C^{\rcone},\cat M}}\from\ho\pr{i^{*}},
\]
where $\cat M=\mathbf{M}[\weq^{-1}]$ denotes the underlying $\infty$-category
of $\mathbf{M}$. By the uniqueness of left adjoints, it follows that
$\mathbb{L}i_{!}\cong\ho\pr{i_{!}}$, where $i_{!}\from\Fun\pr{\cat C,\mathcal{M}}\to\Fun\pr{\cat C^{\rcone},\cat M}$
denotes the left Kan extension functor. (It exists by Remark \ref{rem:und_semi}.)
Thus, the image of $F$ in $\Fun\pr{\cat C^{\rcone},\cat M}$ lies
in the essential image of the left Kan extension functor $i_{!}$,
meaning that it is a colimit diagram in $\cat M$.
\end{proof}

\begin{cor}
\label{cor:RR15_3.1}Let $\mathbf{M}$ be a tractable semi-model category
or a combinatorial model category, let $\kappa$ be a regular cardinal.
Suppose that $\mathbf{M}$ admits a generating set of cofibrations
consisting of maps of $\kappa$-compact objects. Then:
\begin{enumerate}
\item Weak equivalences of $\mathbf{M}$ are stable under $\kappa$-filtered
colimits.
\item The localization functor $L\from\mathbf{M}\to\mathbf{M}[\weq^{-1}]$
preserves $\kappa$-filtered colimits.
\end{enumerate}
\end{cor}

\begin{proof}
For both (1) and (2), it suffices to consider diagrams indexed by
$\kappa$-filtered categories, because every $\kappa$-filtered $\infty$-category
admits a final functor from a $\kappa$-filtered category \cite[\href{https://kerodon.net/tag/02QA}{Tag 02QA}]{kerodon}.

Part (2) follows from part (1) and Proposition \ref{prop:hocolim},
so it suffices to prove (1). For this, let $\cat C$ be a $\kappa$-filtered
category, and let $\alpha\from F\xrightarrow{\simeq}G$ be a natural
weak equivalence of $\Fun\pr{\cat C,\mathbf{M}}$. We wish to show
that $\colim_{\cat C}\alpha$ is a weak equivalence. For this, find
a commutative diagram
\[\begin{tikzcd}
	{F'} & {G'} \\
	F & G
	\arrow["{\alpha'}", from=1-1, to=1-2]
	\arrow[from=1-1, to=2-1]
	\arrow[from=1-2, to=2-2]
	\arrow["\alpha", from=2-1, to=2-2]
\end{tikzcd}\]where the vertical arrows are projective trivial fibrations, and $F'$
and $G'$ are projectively cofibrant. Since the functor $\colim_{\cat C}$
is left Quillen for the projective (semi-)model structure, Ken Brown's
lemma \cite[Corollary 7.4.14]{HCHA} shows that it preserves weak
equivalences of projectively cofibrant objects. Thus $\colim_{\cat C}\alpha'$
is a weak equivalence. Hence, by two out of three property of weak
equivalences, it suffices to show that $\colim_{\cat C}$ carries
projective trivial fibrations to trivial fibrations. This follows
from our assumption on $\mathbf{M}$.
\end{proof}

\subsection{Jeff Smith's theorem and Bousfield localization}

In a tractable semi-model category, fibrations are exactly the maps
having the right lifting property for trivial cofibrations of cofibrant
objects. Therefore, a tractable semi-model structure is completely
determined by the class of cofibrations and weak equivalences. This
prompts the following question: Given a locally presentable category
and a class of cofibrations and weak equivalences, when does it determine
a tractable semi-model structure? 

Surprisingly, there is a complete answer to this question (Theorem
\ref{thm:Smith}). Using this, we will show that semi-model categories
enjoy excellent properties for Bousfield localization (Proposition
\ref{prop:semi_bousfield}).

We start by giving a precise answer to the question above, summarized
in the following theorem. In the model-categorical setting, the theorem
is often attributed to Jeff Smith.
\begin{thm}
\label{thm:Smith}Let $\mathbf{M}$ be a locally presentable category
equipped with two classes $\pr{C,W}$ of maps, whose elements are
called cofibrations and weak equivalences. Then $\pr{C,W}$ is part
of a tractable semi-model structure on $\mathbf{M}$ if and only if
the following conditions are satisfied:
\begin{enumerate}
\item The class $W$, considered as a full subcategory of $\Fun\pr{[1],\mathbf{M}}$,
is accessible and accessibly embedded.
\item The class $W$ is closed under retracts and has the two out of three
property.
\item If a map has the right lifting property for cofibrations, then it
is a weak equivalence.
\item There is a small set $I$ of maps of cofibrant objects such that $C=I\-\cof$.
\item Given a pushout diagram
\[\begin{tikzcd}
	A & {A'} \\
	B & {B'}
	\arrow[from=1-1, to=1-2]
	\arrow["i"', from=1-1, to=2-1]
	\arrow["{i'}", from=1-2, to=2-2]
	\arrow[from=2-1, to=2-2]
\end{tikzcd}\]in $\mathbf{M}$, if $A$ and $A'$ cofibrant (i.e., maps from the
initial objects are cofibrations) and $i\in C\cap W$, then $i'\in C\cap W$.
\item The maps in $C\cap W$ with cofibrant domains are closed under transfinite
composition.
\end{enumerate}
\end{thm}

\begin{proof}
Sufficiency is proved in \cite[Theorem B]{BW24}.\footnote{In \cite{BW24}, it is required that weak equivalences with cofibrant
domains span an accessible and accessibly embedded full subcategory.
However, the proof goes through without the cofibrancy assumption,
as verified via private communication with David White.} For necessity, suppose that $\pr{C,W}$ is part of a tractable semi-model
structure on $\mathbf{M}$. We must check that conditions (1) through
(6) are satisfied. Conditions (2), (3), (6) hold trivially, and conditions
(4) and (5) hold by Remark \ref{rem:7.14_JT07}. For part (1), we
recall that from Corollary \ref{cor:RR15_3.1} that the functor 
\[
\mathbf{M}\to\mathbf{M}[\weq^{-1}]
\]
is accessible. Since the target of this functor is presentable (Remark
\ref{rem:und_semi}), the accessibility of $W$ follows from \cite[Proposition 5.4.6.6]{HTT}
and Corollary \ref{cor:ACK25_1}.
\end{proof}

We can use Theorem \ref{thm:Smith} to construct Bousfield localizations
of semi-model categories.
\begin{defn}
Let $\mathbf{M}$ be a semi-model category. A \textbf{left Bousfield
localization }of $\mathbf{M}$ is a semi-model category $\mathbf{M}'$
having the same underlying category and cofibrations as $\mathbf{M}$,
but which has more weak equivalences than $\mathbf{M}$.
\end{defn}

\begin{prop}
\label{prop:semi_bousfield}Let $\mathbf{M}$ be a tractable semi-model
category with underlying $\infty$-category $\mathcal{M}$. Let $F:\mathcal{M}\to\mathcal{N}$
be a functor admitting a fully faithful right adjoint $G$, where
$\mathcal{N}$ is presentable. There is a tractable left Bousfield
localization $\mathbf{N}$ of $\mathbf{M}$ with the following properties:

\begin{enumerate}[label=(\alph*)]

\item The weak equivalences are the maps inverted by the composite
$\mathbf{M}\to\mathcal{M}\xrightarrow{F}\mathcal{N}$.

\item The induced functor $\mathbf{N}\to\mathcal{N}$ is a localization
at weak equivalences.

\item The fibrant objects of $\mathbf{N}$ are the fibrant objects
of $\mathbf{M}$ whose image in $\cat M$ lies in the essential image
of $G$.

\end{enumerate}
\end{prop}

\begin{proof}
We first check that there is a tractable left Bousfield localization
of $\mathbf{M}$ satisfying condition (a). For this, we will use Theorem
\ref{thm:Smith} to the class $C$ of cofibrations of $\mathbf{M}$
and the class $W$ of maps inverted by the composite $\mathbf{M}\to\mathcal{M}\xrightarrow{F}\mathcal{N}$.
We must check conditions (1) through (6) of loc. cit. Condition (1)
follows from Corollary \ref{cor:RR15_3.1} and \cite[Proposition 5.4.6.6]{HTT}.
Conditions (2), (3), (4) hold trivially. Condition (5) and (6) follow
from Proposition \ref{prop:hocolim}, because the relevant diagrams
are homotopy colimit diagrams. (Note that the projective semi-model
structures for spans and transfinite sequences agree with the Reedy
semi-model structures. Thus projectively cofibrant diagrams are easy
to identify in these cases.)

We next check that the semi-model category $\mathbf{N}$ we just constructed
satisfies conditions (b) and (c).

For condition (b), notice that $F$ is a localization at the maps
it inverts \cite[\href{https://kerodon.net/tag/04JL}{Tag 04JL}]{kerodon}.
It follows that the composite $\mathbf{M}\to\mathcal{M}\xrightarrow{F}\mathcal{N}$
is a localization at $W$, as claimed.

Finally, for (c), note that \cite[Theorem 2.2.11]{cathtpy}, the adjunction
$\id\from\mathbf{M}\adj\mathbf{N}\from\id$ gives rise to the derived
adjunction
\[
\mathbb{L}\id\from\ho\pr{\mathcal{M}}\adj\ho\pr{\mathcal{N}}\from\mathbb{R}\id.
\]
Explicitly, $\mathbb{L}\id=\ho\pr F$, and $\mathbb{R}\id=\ho\pr R$,
where $R$ is a fibrant replacement functor of $\mathbf{N}$. By the
uniqueness of right adjoints, $\mathbb{R}\id$ is isomorphic to $\ho\pr G$.
This implies the following:
\begin{itemize}
\item If $X\in\mathbf{M}$ is fibrant in $\mathbf{N}$, then its image in
$\cat M$ lies in the essential image of $G$. 
\item If $X\in\mathbf{M}$ is an object whose image in $\cat M$ lies in
the essential image of $G$, then $X$ is connected by a zig-zag of
weak equivalences of $\mathbf{M}$ to a fibrant object of $\mathbf{N}$
(by Corollary \ref{cor:ACK25_1}).
\end{itemize}
The claim then follows from Proposition \ref{prop:A.2.3.1}.
\end{proof}

A model-categorical version of Proposition \ref{prop:semi_bousfield}
is well-known to the experts, but we could not find a reference. We
record it here for future reference:
\begin{prop}
\label{prop:bousfield}Let $\mathbf{M}$ be a combinatorial left proper
model category with underlying $\infty$-category $\mathcal{M}$,
and let $F:\mathcal{M}\to\mathcal{N}$ be a functor admitting a fully
faithful right adjoint $G$, where $\mathcal{N}$ is presentable.
There is a left Bousfield localization $\mathbf{N}$ of $\mathbf{M}$
(in the sense of model categories) with the following properties:

\begin{enumerate}[label=(\alph*)]

\item The weak equivalences are the maps inverted by the composite
$\mathbf{M}\to\mathcal{M}\xrightarrow{F}\mathcal{N}$.

\item The induced functor $\mathbf{N}\to\mathcal{N}$ is a localization
at weak equivalences.

\item The fibrant objects of $\mathbf{N}$ are the fibrant objects
of $\mathbf{M}$ whose image in $\cat M$ lies in the essential image
of $G$.

\end{enumerate}
\end{prop}

\begin{proof}
Identical to Proposition \ref{prop:semi_bousfield}, using \cite[Proposition 2.2]{Bar10}
instead of Theorem \ref{thm:Smith}. (To check (6), we use \cite[Proposition 17.9.3]{Hirschhorn}).
\end{proof}

We conclude this subsection with the following recognition result
for fibrations in Bousfield localization:
\begin{prop}
\label{prop:3.3.16_Hir}Let $\mathbf{M}$ be a tractable semi-model
category, let $L\mathbf{M}$ be a tractable left Bousfield localization
of $\mathbf{M}$, and let $p\from X\to Y$ be a map of fibrant objects
of $L\mathbf{M}$. Then $p$ is a fibration of $L\mathbf{M}$ if and
only if it is a fibration of $\mathbf{M}$.
\end{prop}

\begin{proof}
In a tractable semi-model category, fibrations are exactly the maps
having the right lifting property for the trivial cofibrations of
cofibrant objects. Thus, the ``only if'' part is clear. For the
``if'' part, suppose that $p$ is a fibration of $\mathbf{M}$.
Consider a lifting problem
\[\begin{tikzcd}
	A & X \\
	B & Y
	\arrow[from=1-1, to=1-2]
	\arrow["i"', from=1-1, to=2-1]
	\arrow["p", from=1-2, to=2-2]
	\arrow[dashed, from=2-1, to=1-2]
	\arrow[from=2-1, to=2-2]
\end{tikzcd}\]where $i$ is a trivial cofibration of $L\mathbf{M}$ and $A$ is
cofibrant. We must show that there is a dashed filler rendering the
diagram commutative. Factor $i$ as $A\xrightarrow{j}A'\xrightarrow{q}B$,
where $j$ is a trivial cofibration of $\mathbf{M}$ and $q$ is (necessarily)
a weak equivalence of $L\mathbf{M}$. Since $p$ is a fibration of
$\mathbf{M}$, there is a dashed filler depicted as 
\[\begin{tikzcd}
	A & A & X \\
	{A'} & B & Y
	\arrow[equal, from=1-1, to=1-2]
	\arrow["j"', from=1-1, to=2-1]
	\arrow[from=1-2, to=1-3]
	\arrow["i"', from=1-2, to=2-2]
	\arrow["p", from=1-3, to=2-3]
	\arrow[dashed, from=2-1, to=1-3]
	\arrow[from=2-1, to=2-2]
	\arrow[from=2-2, to=2-3]
\end{tikzcd}\]Applying Proposition \ref{prop:A.2.3.1} to the semi-model category
$\pr{L\mathbf{M}}_{/Y}$, we find that the original lifting problem
also admits a solution. The proof is now complete.
\end{proof}

\section{\label{sec:MGUD}Multiplicative Gabriel--Ulmer duality}

The Gabriel--Ulmer duality \cite{GabrielUlmer71} is one precise
formulation of the slogan that ``locally presentable categories are
controlled by small data.'' In the $\infty$-categorical setting,
one formulation of this is that for every regular cardinal $\kappa$,
the $\Ind_{\kappa}$-completion functor induces an equivalence of
$\infty$-categories
\[
\Ind_{\kappa}\from\Cat^{\idem}_{\infty}\pr{\kappa}\xrightarrow{\simeq}\Pr^{L}\pr{\kappa}.
\]
Here $\Cat^{\idem}_{\infty}\pr{\kappa}$ denotes the $\infty$-category
of small idempotent complete $\infty$-categories with $\kappa$-small
colimits and functors preserving $\kappa$-small colimits, and $\Pr^{L}\pr{\kappa}$
denotes the $\infty$-category of $\kappa$-presentable $\infty$-categories\footnote{An $\infty$-category is \textbf{$\kappa$-presentable }if it is $\kappa$-accessible
and presentable.} and functors that preserve $\kappa$-compact objects and small colimits.

The goal of this section is to enhance this equivalence to that of
symmetric monoidal $\infty$-categories (Theorem \ref{thm:MGU}) when
$\kappa$ is uncountable. As a corollary of this, we obtain a Gabriel--Ulmer
duality for $\kappa$-presentably symmetric monoidal $\infty$-categories
(Corollary \ref{cor:MGUD}) for an uncountable $\kappa$.

We start by reviewing symmetric monoidal structures on $\Pr^{L}$
and $\Cat_{\infty}$.

\begin{recollection}
\cite[Corollary 4.8.1.4]{HA}\label{recoll:4.8.1.4} Let $\cat K$
be a small set of small simplicial sets. We write $\Cat_{\infty}\pr{\cat K}\subset\Cat_{\infty}$
for the subcategory whose objects are the $\infty$-categories admitting
colimits indexed by the elements in $\cat K$, whose morphisms are
the functors that preserve such colimits. 

The $\infty$-category $\Cat_{\infty}\pr{\cat K}$ can be enhanced
to a symmetric monoidal $\infty$-category, which we denote by $\Cat_{\infty}\pr{\cat K}^{\t}$.
An active map $\pr{\cat A_{1},\dots,\cat A_{k}}\to\cat B$ is given
by a functor $\cat A_{1}\times\cdots\times\cat A_{k}\to\cat B$ preserving
colimits indexed by the elements of $\cat K$ in each variable. Such
a map is cocartesian if and only if, for any other $\cat C\in\Cat_{\infty}\pr{\kappa}$,
the functor
\[
\Fun^{\cat K}\pr{\cat B,\cat D}\to\Fun^{\cat K}\pr{\cat A_{1}\times\dots\times\cat A_{k},\cat C}
\]
is an equivalence. Here, $\Fun^{\cat K}$ denotes the $\infty$-category
of functors preserving colimits indexed by the elements in $\cat K$
in each variable. 

In the case where $\cat K$ is the set of (representatives of isomorphism
classes of) all $\kappa$-small simplicial sets, where $\kappa$ is
a regular cardinal, we write $\Cat_{\infty}\pr{\cat K}=\Cat_{\infty}\pr{\kappa}$
and $\Fun^{\cat K}=\Fun^{\kappa}$.
\end{recollection}

\begin{recollection}
\cite[Proposition 4.8.1.15]{HA} Applying Recollection \ref{recoll:4.8.1.4}
to the very large $\infty$-category $\hat{\Cat}_{\infty}$ of large
$\infty$-categories and the collection $\cat U$ of all small simplicial
sets, we obtain the symmetric monoidal $\infty$-category $\hat{\Cat}_{\infty}\pr{\cat U}^{\t}$
of large $\infty$-categories with small colimits, and functors preserving
such colimits. The full subcategory $\Pr^{L}\subset\hat{\Cat}_{\infty}\pr{\cat U}$
is stable under tensor products, so it inherits a symmetric monoidal
structure. We denote the resulting symmetric monoidal $\infty$-category
by $\pr{\Pr^{L}}^{\t}$.
\end{recollection}

We can relate $\Cat_{\infty}\pr{\kappa}^{\t}$ with $\pr{\Pr^{L}}^{\t}$
by $\Ind_{\kappa}$-completion (Corollary \ref{cor:4.8.1.8}). For
this, we recall the following propositions:
\begin{prop}
\cite[Remark 4.8.1.8]{HA}\label{prop:4.8.1.8-1} Let $\cat K\subset\cat L$
be an inclusion of small sets of small simplicial sets. The inclusion
\[
\Cat_{\infty}\pr{\cat L}^{\t}\subset\Cat_{\infty}\pr{\cat K}^{\t}
\]
admits a left adjoint relative to $\Fin_{\ast}$ (in the sense of
\cite[Definition 7.3.2.2]{HA}), and the left adjoint $\Cat_{\infty}\pr{\cat K}^{\t}\to\Cat_{\infty}\pr{\cat L}^{\t}$
is symmetric monoidal.
\end{prop}

\begin{prop}
\cite[Propositions 5.3.5.10 and 5.5.1.9]{HTT}\label{prop:Ind_univ}
Let $\cat A$ be a small $\infty$-category with $\kappa$-small colimits.
The functor $\cat A\to\Ind_{\kappa}\pr{\cat A}$ has the following
universal property: For every $\infty$-category $\cat Z$ with small
colimits, the functor
\[
\Fun^{L}\pr{\Ind_{\kappa}\pr{\cat A},\cat Z}\to\Fun^{\kappa}\pr{\cat A,\cat Z}
\]
is an equivalence, where the left-hand side denotes the $\infty$-category
of functors preserving small colimits.
\end{prop}

\begin{cor}
\label{cor:4.8.1.8} Let $\kappa$ be a regular cardinal. The assignment
$\cat A\mapsto\Ind_{\kappa}\pr{\cat A}$ can be enhanced to a symmetric
monoidal functor
\[
\Ind_{\kappa}\from\Cat_{\infty}\pr{\kappa}^{\t}\to\pr{\Pr^{L}}^{\t}.
\]
\end{cor}

\begin{proof}
In light of Proposition \ref{prop:Ind_univ}, this is a special case
of Corollary \ref{cor:4.8.1.8}.
\end{proof}

We now use the functor of Corollary \ref{cor:4.8.1.8} to obtain the
classical Gabriel--Ulmer duality. To state it, write $\overline{\Cat_{\infty}}\pr{\kappa}\subset\hat{\Cat}_{\infty}$
for the subcategory whose objects are the essentially small $\infty$-categories
with $\kappa$-small colimits, and whose morphisms are the functors
that preserve $\kappa$-small colimits. We also write $\Cat^{\idem}_{\infty}\pr{\kappa}\subset\Cat_{\infty}\pr{\kappa}$
and $\overline{\Cat^{\idem}_{\infty}}\pr{\kappa}\subset\overline{\Cat_{\infty}}\pr{\kappa}$
for the full subcategories of idempotent complete objects.
\begin{cor}
\label{cor:GU_plain}Let $\kappa$ be a regular cardinal. The functor
of Corollary \ref{cor:4.8.1.8} induces an equivalence of $\infty$-categories
\[
\Ind_{\kappa}\from\Cat^{\idem}_{\infty}\pr{\kappa}\xrightarrow{\simeq}\Pr^{L}\pr{\kappa}.
\]
 The inverse equivalence is given by the forgetful functor
\[
\pr -_{\kappa}\from\Pr^{L}\pr{\kappa}\to\overline{\Cat^{\idem}_{\infty}}\pr{\kappa}\simeq\Cat^{\idem}_{\infty}\pr{\kappa}.
\]
\end{cor}

\begin{proof}
By construction, the functor $\Ind_{\kappa}\from\Cat^{\idem}_{\infty}\pr{\kappa}\to\Pr^{L}$
is equipped with a natural transformation $\{\eta_{\cat A}\from\cat A\to\Ind_{\kappa}\pr{\cat A}\}_{\cat A\in\Cat^{\idem}_{\infty}\pr{\kappa}}$
of functors $\Cat_{\infty}\pr{\kappa}\to\hat{\Cat}_{\infty}$, where
each $\eta_{\cat A}$ exhibits $\Ind_{\kappa}\pr{\cat A}$ as an $\Ind_{\kappa}$-completion
of $\cat A$. According to \cite[Lemma 5.4.2.4]{HTT}, each $\eta_{\cat A}$
restricts to an equivalence $\cat A\xrightarrow{\simeq}\Ind_{\kappa}\pr{\cat A}_{\kappa}$.
Therefore, the functor $\Ind_{\kappa}\from\Cat^{\idem}_{\infty}\pr{\kappa}\to\Pr^{L}$
takes values in $\Pr^{L}\pr{\kappa}$.

To show that the resulting functor $\Ind_{\kappa}\from\Cat^{\idem}_{\infty}\pr{\kappa}\to\Pr^{L}\pr{\kappa}$
is an equivalence, we first observe that the functor $\pr -_{\kappa}\from\Pr^{L}\pr{\kappa}\to\overline{\Cat^{\idem}_{\infty}}\pr{\kappa}\simeq\Cat^{\idem}_{\infty}\pr{\kappa}$
is a right adjoint of $\Ind_{\kappa}$ with unit $\eta$. Indeed,
for each $\cat A\in\Cat^{\idem}_{\infty}\pr{\kappa}$ and $\cat X\in\Pr^{L}\pr{\kappa}$,
Proposition \ref{prop:Ind_univ} and the aforementioned equivalence
$\cat A\xrightarrow{\simeq}\Ind_{\kappa}\pr{\cat A}_{\kappa}$ gives
an equivalence
\[
\Fun_{\Pr^{L}\pr{\kappa}}\pr{\Ind_{\kappa}\pr{\cat A},\cat X}\xrightarrow{\simeq}\Fun_{\kappa}\pr{\cat A,\cat X_{\kappa}},
\]
where the left-hand side is the full subcategory of $\Fun\pr{\Ind_{\kappa}\pr{\cat A},\cat X}$
spanned by the functors in $\Pr^{L}\pr{\kappa}$. Now for each $\cat X\in\Pr^{L}\pr{\kappa}$,
the counit $\Ind_{\kappa}\pr{\cat X_{\kappa}}\to\cat X$ is also an
equivalence by \cite[Propositions 5.3.5.11 and 5.4.2.2.]{HTT}. Therefore,
$\Ind_{\kappa}$ and $\pr -_{\kappa}$ is part of an adjoint equivalence.
\end{proof}

We now upgrade Corollary \ref{cor:GU_plain} to a symmetric monoidal
equivalence (Theorem \ref{thm:MGU}) and look at its consequences.
\begin{defn}
Let $\kappa$ be a regular cardinal. We write $\Pr^{L}\pr{\kappa}^{\t}\subset\pr{\Pr^{L}}^{\t}$
for the suboperad spanned by the $\kappa$-presentable $\infty$-categories
and those maps $\pr{\mathcal{C}_{1},\dots,\mathcal{C}_{n}}\to\mathcal{C}$
whose corresponding functors $F:\mathcal{C}_{1}\times\cdots\times\mathcal{C}_{n}\to\mathcal{C}$
have the following property:
\begin{itemize}
\item Given an object $\pr{X_{1},\dots,X_{n}}\in\mathcal{C}_{1}\times\cdots\times\mathcal{C}_{n}$,
if each $X_{i}$ is $\kappa$-presentable, then so is $F\pr{X_{1},\dots,X_{n}}$.
\end{itemize}
\end{defn}

\begin{thm}
\label{thm:MGU}Let $\kappa$ be a regular cardinal. Then:
\begin{enumerate}
\item $\Pr^{L}\pr{\kappa}^{\t}$ is a symmetric monoidal $\infty$-category,
and the inclusion $\Pr^{L}\pr{\kappa}^{\t}\hookrightarrow\pr{\Pr^{L}}^{\t}$
is symmetric monoidal.
\item The symmetric monoidal functor $\Ind_{\kappa}\from\Cat_{\infty}\pr{\kappa}^{\t}\to\pr{\Pr^{L}}^{\t}$
of Corollary \ref{cor:4.8.1.8} restricts to a symmetric monoidal
functor
\[
\Ind_{\kappa}\from\Cat_{\infty}\pr{\kappa}^{\t}\to\pr{\Pr^{L}\pr{\kappa}}^{\t},
\]
which is an equivalence if $\kappa$ is uncountable.
\end{enumerate}
\end{thm}

\begin{proof}
Part (2) follows from part (1) and Corollary \ref{cor:GU_plain},
so we only have to prove (1). For this, it suffices to prove the following: 
\begin{itemize}
\item [(1$'$)]Let $\cat X,\cat Y$ be $\kappa$-presentable $\infty$-categories.
Then:
\begin{itemize}
\item [(1$'$-a)]Their tensor product $\cat X\t_{\Pr^{L}}\cat Y$ in $\Pr^{L}$
is $\kappa$-presentable.
\item [(1$'$-b)]Every object of $\pr{\cat X\t_{\Pr^{L}}\cat Y}_{\kappa}$
is a retract of an object lying in the image of the functor $\cat X_{\kappa}\times\cat Y_{\kappa}\to\pr{\cat X\t_{\Pr^{L}}\cat Y}_{\kappa}.$
\end{itemize}
\end{itemize}

To prove (1$'$-a), use Corollary \ref{cor:GU_plain} to find objects
$\cat A,\cat B\in\Cat_{\infty}\pr{\kappa}$ and equivalences $\Ind_{\kappa}\pr{\cat A}\xrightarrow{\simeq}\cat X$
and $\Ind_{\kappa}\pr{\cat B}\xrightarrow{\simeq}\cat Y$. For each
$\cat Z\in\Pr^{L}$, Proposition \ref{prop:Ind_univ} gives an equivalence
\begin{align*}
\Fun^{L}\pr{\cat X\t_{\Pr^{L}}\cat Y,\cat Z} & \simeq\Fun^{L}\pr{\cat X,\Fun^{L}\pr{\cat Y,\cat Z}}\\
 & \simeq\Fun^{\kappa}\pr{\cat A,\Fun^{\kappa}\pr{\cat B,\cat Z}}\\
 & \simeq\Fun^{\kappa}\pr{\cat A\t_{\Cat_{\infty}\pr{\kappa}}\cat B,\cat Z}\\
 & \simeq\Fun^{L}\pr{\Ind_{\kappa}\pr{\cat A\t_{\Cat_{\infty}\pr{\kappa}}\cat B},\cat Z}.
\end{align*}
It follows that $\cat X\t_{\Pr^{L}}\cat Y$ is equivalent to $\Ind_{\kappa}\pr{\cat A\t_{\Cat_{\infty}\pr{\kappa}}\cat B}$.
Since the latter is is $\kappa$-presentable, this proves (1$'$-a). 

For part (1$'$-b), note that we can identify the map $\cat X_{\kappa}\times\cat Y_{\kappa}\to\pr{\cat X\t_{\Pr^{L}}\cat Y}_{\kappa}$
with the composite
\[
\cat A\times\cat B\to\cat A\t_{\Cat_{\infty}\pr{\kappa}}\cat B\to\Ind_{\kappa}\pr{\cat A\t_{\Cat_{\infty}\pr{\kappa}}\cat B}_{\kappa}.
\]
The first functor is essentially surjective (see the proof of \cite[Proposition 4.8.1.3]{HA}).
Also, by \cite[Lemma 5.4.2.4]{HTT} the second functor identifies
its target as an idempotent completion of the source. Hence (1$'$-b)
follows.
\end{proof}

\begin{rem}
We do not know if $\Cat^{\idem}_{\infty}\pr{\omega}$ is closed under
tensor products in $\Cat_{\infty}\pr{\omega}$. This is why we imposed
the uncountability assumption in part (2) of Theorem \ref{thm:MGU}.
\end{rem}

We conclude this section with three corollaries of Theorem \ref{thm:MGU}.
For the following corollary, we make $\overline{\Cat}_{\infty}\pr{\kappa}$
into a symmetric monoidal $\infty$-category $\overline{\Cat}_{\infty}\pr{\kappa}^{\t}$
using Recollection \ref{recoll:4.8.1.4}.
\begin{cor}
\label{cor:GU}Let $\kappa$ be an uncountable regular cardinal. The
forgetful functor
\[
\pr -_{\kappa}\from\Pr^{L}\pr{\kappa}^{\t}\to\overline{\Cat}_{\infty}\pr{\kappa}^{\t}
\]
is an equivalence of symmetric monoidal $\infty$-categories.
\end{cor}

\begin{proof}
The composite $\Ind_{\kappa}\from\Cat_{\infty}\pr{\kappa}^{\t}\xrightarrow{\Ind_{\kappa}}\Pr^{L}\pr{\kappa}^{\t}\xrightarrow{\pr -_{\kappa}}\overline{\Cat}_{\infty}\pr{\kappa}^{\t}$
is equivalent to the inclusion, which is an equivalence of symmetric
monoidal $\infty$-categories. Since $\Ind_{\kappa}$ is an equivalence
by Theorem \ref{thm:MGU}, we are done.
\end{proof}

\begin{cor}
\label{cor:MGUD}Let $\kappa$ be an uncountable regular cardinal.
The forgetful functor
\[
\kappa\-\Pr\SM\to\SM\overline{\Cat}\pr{\kappa},\,\mathcal{C}^{\t}\mapsto\mathcal{C}^{\t}_{\kappa}
\]
is a categorical equivalence, where $\SM\overline{\Cat}_{\infty}\pr{\kappa}\subset\SM\hat{\Cat}_{\infty}$
denotes the subcategory of essentially small symmetric monoidal $\infty$-categories
compatible with $\kappa$-small colimits, and those monoidal functors
preserving $\kappa$-small colimits.
\end{cor}

\begin{proof}
This follows from Corollary \ref{cor:GU} and the straightening--unstraightening
equivalence.
\end{proof}

\begin{cor}
\label{cor:MGUD_2}Let $\kappa$ be an uncountable regular cardinal,
and let $I:\SM\Cat_{\infty}\pr{\kappa}\to\kappa\-\Pr\SM$ be a functor,
denoted by $\cat C^{\t}\mapsto I\pr{\cat C}^{\t}$. Suppose there
is a natural transformation depicted as
\[\begin{tikzcd}
	{\mathcal{SM}\mathcal{C}\mathsf{at}_\infty^{}(\kappa)} && {\kappa\text{-}\mathcal{P}\mathsf{r}\mathcal{SM}} \\
	& {\mathcal{SM}\widehat{\mathcal{C}\mathsf{at}}_\infty,}
	\arrow["I", from=1-1, to=1-3]
	\arrow[""{name=0, anchor=center, inner sep=0}, hook, from=1-1, to=2-2]
	\arrow[""{name=1, anchor=center, inner sep=0}, hook', from=1-3, to=2-2]
	\arrow["\alpha", between={0.2}{0.8}, Rightarrow, from=0, to=1]
\end{tikzcd}\]with the following property: For each $\mathcal{C}\in\SM\Cat_{\infty}\pr{\kappa}$,
the map $\alpha_{\mathcal{C}}:\mathcal{C}\to I\pr{\mathcal{C}}$ exhibits
$I\pr{\mathcal{C}}$ as an $\Ind_{\kappa}$-completion of $\mathcal{C}$.
Then $I$ is a categorical equivalence.
\end{cor}

\begin{proof}
The natural transformation $\alpha$ determines a natural equivalence
\[
\mathcal{C}^{\t}\xrightarrow{\simeq}I\pr{\mathcal{C}}^{\t}_{\kappa}
\]
from the inclusion $\iota:\SM\Cat_{\infty}\pr{\kappa}\hookrightarrow\SM\overline{\Cat}_{\infty}\pr{\kappa}$
to the composite $\SM\Cat_{\infty}\pr{\kappa}\xrightarrow{I}\kappa\-\Pr\SM\xrightarrow{\pr -_{\kappa}}\SM\overline{\Cat}_{\infty}\pr{\kappa}$.
Since the functors $\iota$ and $\pr -_{\kappa}$ are equivalences
(Corollary \ref{cor:MGUD}), so must be $I$.
\end{proof}

\section{\label{sec:lb}Base change of localization}

Let $L\from\mathcal{C}\to\mathcal{D}$ be a localization of $\infty$-categories,
and suppose we are given a pullback square in $\Cat_{\infty}$ 
\[\begin{tikzcd}
	{\mathcal{C}'} & {\mathcal{D}} \\
	{\mathcal{D}'} & {\mathcal{D}.}
	\arrow[from=1-1, to=1-2]
	\arrow["{L'}"', from=1-1, to=2-1]
	\arrow["L", from=1-2, to=2-2]
	\arrow[from=2-1, to=2-2]
\end{tikzcd}\]The functor $L'$ is not a localization in general. This naturally
raises a question: When is $L'$ a localization? The goal of this
subsection is to record a few results related to this question.

We start with a sufficient condition for a localization functor to
be stable under base change; such a functor is called a \textbf{universal
localization}. (A general criteria for universal localization is discussed
in \cite{Hinich24}.) The following theorem says that localization
of model categories, or mo generally, $\infty$-categories with weak
equivalences and fibrations (in the sense of \cite[Definition 7.4.12]{HCHA})
enjoy this property:
\begin{thm}
\label{thm:abs_loc_model_gen}Let $\mathcal{C}$ be an $\infty$-category
with weak equivalences and fibrations. The localization functor
\[
\mathcal{C}\to\mathcal{C}[\mathrm{weq}^{-1}]
\]
is a universal localization.
\end{thm}

The proof of Theorem \ref{thm:abs_loc_model_gen} requires the following
lemma: 
\begin{lem}
\label{lem:abs_loc}Let $\mathcal{C}$ be an $\infty$-category, let
$\mathcal{W}\subset\mathcal{C}$ be a subcategory containing all objects,
and let $L\from\mathcal{C}\to\mathcal{D}$ be a functor of $\infty$-categories
which carries every morphism in $\mathcal{W}$ to an equivalence.
If for each $n\geq0$, the functor
\[
\theta_{L,n}:\Fun\pr{[n],\mathcal{C}}\times_{\mathcal{C}^{n+1}}\mathcal{W}^{n+1}\to\Fun\pr{[n],\mathcal{D}}^{\simeq}
\]
is a weak homotopy equivalence, then $L$ is a universal localization.
\end{lem}

\begin{proof}
Let $\mathcal{D}'\to\mathcal{D}$ be a categorical fibration of $\infty$-categories,
and set $\mathcal{C}'=\mathcal{D}'\times_{\mathcal{D}}\mathcal{C}$.
We must show that the functor $L':\mathcal{C}'\to\mathcal{D}'$ is
a localization. For this, set $\mathcal{W}'=\pr{\mathcal{D}'}^{\simeq}\times_{\mathcal{D}}\mathcal{W}\subset\mathcal{C}'$.
We claim that $L'$ is a localization at $\mathcal{W}'$. According
to the generalized Mazel-Gee's localization theorem \cite[Theorem 1.7]{A23b},
it suffices to show that for each $n\geq0$, the map $\theta_{L',n}$
is a weak homotopy equivalence. But $\theta'_{n}$ is a pullback of
the map $\theta_{n}$ along the  Kan fibration $\Fun\pr{[n],\mathcal{D}'}^{\simeq}\to\Fun\pr{[n],\mathcal{D}}^{\simeq}$,
so the claim follows from the right properness of the Kan--Quillen
model structure. (If $\cat W$ is the subcategory of morphisms mapped
to equivalences in $\mathcal{D}$, we can use the original version
of Mazel-Gee's theorem \cite[Theorem 3.8]{MR4045352}. See \cite[Theorem 1.1]{AC25}
for a different proof of this theorem.) 
\end{proof}

\begin{proof}
[Proof of Theorem \ref{thm:abs_loc_model_gen}]Let $\mathcal{W}\subset\mathcal{C}$
denote the subcategory of weak equivalences. By \cite[Remark 7.5.22]{HCHA},
we may assume that $\mathcal{W}$ is saturated, i.e., that it consists
of the morphisms whose images in $L\pr{\mathcal{C}}=\mathcal{C}[\mathrm{weq}^{-1}]$
are equivalences. By Lemma \ref{lem:abs_loc}, it will suffice to
show that the functor
\[
\theta_{\mathcal{C},n}:\Fun\pr{[n],\mathcal{C}}\times_{\mathcal{C}^{n+1}}\mathcal{W}^{n+1}\to\Fun\pr{[n],L\pr{\mathcal{C}}}
\]
is a weak homotopy equivalence for every $n\geq0$. By \cite[Theorem 7.4.20]{HCHA},
the $\infty$-category $\Fun\pr{[n],\mathcal{C}}$ has the structure
of an $\infty$-category with weak equivalences and fibrations whose
subcategory of weak equivalences is given by $\Fun\pr{[n],\mathcal{C}}\times_{\mathcal{C}^{n+1}}\mathcal{W}^{n+1}$.
Moreover, by \cite[Theorem 7.6.17]{HCHA}, the functor
\[
L\pr{\Fun\pr{[n],\mathcal{C}}}\to\Fun\pr{[n],L\pr{\mathcal{C}}}
\]
is a categorical equivalence. It follows that $\Fun\pr{[n],\cat C}$
is saturated, and the map $\theta_{L,n}$ can be identified with $\theta_{\Fun\pr{[n],\mathcal{C}},0}$.
Thus, replacing $\mathcal{C}$ by $\Fun\pr{[n],\mathcal{C}}$, we
are reduced to showing that $\theta_{\mathcal{C},0}\from\mathcal{W}\to L\pr{\mathcal{C}}$
is a weak homotopy equivalence. This is the content of \cite[Lemma 7.6.9]{HCHA},
and the proof is complete.
\end{proof}

We next consider the interaction of pullback and homotopy equivalences
of relative $\infty$-categories.
\begin{prop}
\label{prop:hoeq_pb}Consider a diagram $[2]\times[1]\to\Cat_{\infty}$
depicted as 
\[\begin{tikzcd}
	&&& {\mathcal{A}} && {\mathcal{B}} \\
	{\mathcal{A}'} && {\mathcal{B}'} \\
	&&&& {\mathcal{C},} \\
	& {\mathcal{C}'}
	\arrow["f", from=1-4, to=1-6]
	\arrow["p"', from=1-4, to=3-5]
	\arrow["q", from=1-6, to=3-5]
	\arrow[from=2-1, to=1-4]
	\arrow["{f'}", from=2-1, to=2-3]
	\arrow["{p'}"', from=2-1, to=4-2]
	\arrow[from=2-3, to=1-6]
	\arrow["{q'}", from=2-3, to=4-2]
	\arrow["u"', from=4-2, to=3-5]
\end{tikzcd}\]where the squares are all cartesian. Regard $\cat A$ and $\cat B$
(resp. $\cat A'$ and $\cat B'$) as relative $\infty$-categories
whose weak equivalences are the maps whose images in $\cat C$ (resp.
$\cat C'$) are equivalences. Suppose that the following conditions
are satisfied:
\begin{enumerate}
\item The functor $f\from\cat A\to\cat B$ is a homotopy equivalence of
relative $\infty$-categories.
\item The functor $u\from\cat C'\to\cat C$ is conservative (i.e., reflects
equivalences).
\end{enumerate}
Then $f'$ is a homotopy equivalence of relative $\infty$-categories.
\end{prop}

\begin{proof}
We claim that there are maps $g\from\cat B\to\cat A$ and $i\from\cat B\to\cat B$
in $\Cat_{\infty/\cat C}$ and a diagram of the form 
\begin{equation}\label{diagram:htpy}
\begin{tikzcd}
	{\mathcal{B}\times \{0\}} \\
	& {\mathcal{B}\times \mathcal{I}} & {\mathcal{B}} \\
	{\mathcal{B}\times \{1\}}
	\arrow[hook, from=1-1, to=2-2]
	\arrow["fg", curve={height=-12pt}, from=1-1, to=2-3]
	\arrow["\Phi", from=2-2, to=2-3]
	\arrow[hook', from=3-1, to=2-2]
	\arrow["i"', curve={height=12pt}, from=3-1, to=2-3]
\end{tikzcd}
\end{equation}in $\Cat_{\infty/\cat C}$, where:

\begin{enumerate}[label=(\Roman*)]

\item $\cat I$ is a weakly contractible $\infty$-category equipped
with two distinguished objects $0,1\in\cat I$.

\item For each $B\in\cat B$, the functor $\Phi\vert\{B\}\times\cat I$
carries each morphism to a weak equivalence.

\item The functor $i$ is an equivalence of $\infty$-categories.

\item $\cat B\times\cat I$ lies over $\cat C$ via the composite
$\cat B\times\cat I\xrightarrow{\mathrm{pr}}\cat B\xrightarrow{q}\cat C$.

\end{enumerate}

Pulling back along $u\from\cat C'\to\cat C$, we obtain functors $g'\from\cat B'\to\cat A'$,
$i'\from\cat B'\to\cat B'$, and $\Phi'\from\cat B'\times\cat I\to\cat B'$.
Using condition (2), we find that $\Phi'$ witnesses the fact that
$f'$ has a right homotopy inverse. Applying the same argument to
$g$, we find that $g'$ also has a right homotopy inverse. It then
follows from the two out of six property of homotopy equivalences
that $f'$ is a homotopy equivalence of relative categories.

To prove the above claim, use condition (1) to find a relative functor
$g\from\cat B\to\cat A$ and diagrams of the form (\ref{diagram:htpy})
in $\Cat_{\infty}$ (not $\Cat_{\infty/\cat C}$ yet!) that satisfy
conditions (I), (II), and (III). Let $S$ denote the set of morphisms
of $\cat B\times\cat I$ of the form $\pr{B,x}\to\pr{B,y}$, where
$B\in\cat B$ and $x\to y$ is a morphism in $\cat I$. By condition
(I), $q\Phi$ factors through the localization $\pr{\cat B\times\cat I}[S^{-1}]\simeq\cat B\times\cat I[\cat I^{-1}]$.
Since $\cat I$ is weakly contractible, this localization can be identified
with $\cat B$, with localizing functor given by the projection $\cat B\times\cat I\to\cat B$.
Thus, there is a diagram $[1]\times[1]\to\Cat_{\infty}$ of the form
\[\begin{tikzcd}
	{\mathcal{B}\times \mathcal{I}} & {\mathcal{B}} \\
	{\mathcal{B}} & {\mathcal{C}}
	\arrow["{\Phi }", from=1-1, to=1-2]
	\arrow["{\mathrm{pr}}"', from=1-1, to=2-1]
	\arrow["q", from=1-2, to=2-2]
	\arrow["r"', from=2-1, to=2-2]
\end{tikzcd}\] Since $r\simeq\pr{q\Phi}\vert\cat B\times\{1\}=q$, we may assume
that $r=q$. Then we obtain a diagram $\sigma\from[2]\to\Cat_{\infty}$
whose boundary is depicted as 
\[\begin{tikzcd}
	{\mathcal{B}\times \mathcal{I}} & {\mathcal{B}} \\
	& {\mathcal{C}.}
	\arrow["{\Phi }", from=1-1, to=1-2]
	\arrow["{q\circ \mathrm{pr}}"', from=1-1, to=2-2]
	\arrow["q", from=1-2, to=2-2]
\end{tikzcd}\]

We now contemplate the diagrams $\pr{[1]\times[1]}^{\rcone}\to\Cat_{\infty}$
and $[3]\to\Cat_{\infty}$ depicted as
\[\begin{tikzcd}[scale cd =. 95]
	{\mathcal{B}\times\{0\}} &&& {\mathcal{A}} &&& {\mathcal{B}\times \mathcal{I}} && {\mathcal{B}} \\
	{\mathcal{B}\times \mathcal{I}} &&&& {\mathcal{B}} & {\mathcal{B}\times \{1\}} \\
	&&& {\mathcal{C}} &&&& {\mathcal{C}}
	\arrow["g", from=1-1, to=1-4]
	\arrow[hook, from=1-1, to=2-1]
	\arrow["q"{description, pos=0.7}, from=1-1, to=3-4]
	\arrow["f", from=1-4, to=2-5]
	\arrow["p"{description, pos=0.7}, from=1-4, to=3-4]
	\arrow["\Phi", from=1-7, to=1-9]
	\arrow["{q\circ \mathrm{pr}}"', from=1-7, to=3-8]
	\arrow["q", from=1-9, to=3-8]
	\arrow["\Phi"{description}, from=2-1, to=2-5, crossing over]
	\arrow["{q\circ \mathrm{pr}}"', from=2-1, to=3-4]
	\arrow["q", from=2-5, to=3-4]
	\arrow[hook, from=2-6, to=1-7]
	\arrow["i"{description, pos=0.6}, from=2-6, to=1-9, crossing over,shorten >=1.5ex]
	\arrow["q"', from=2-6, to=3-8]
	\arrow["fg"{description}, from=1-1, to=2-5, crossing over, shorten >=1.5ex]
\end{tikzcd}\]The left-hand diagram is constructed as follows: Think of this as
an amalgamation of two $3$-simplices $\Delta^{3}\to\Cat_{\infty}$.
The front $3$-simplex is obtained by filling the horn $\Lambda^{3}_{1}\subset\Delta^{3}$,
using $\sigma$, the natural equivalence $fg\simeq\Phi\vert\cat B\times\{0\}$,
and the $2$-simplex corresponding to the equality $q\circ\mathrm{pr}\vert\cat B\times\{0\}=q$.
We then fill the $3$-simplex in the back by filling the horn $\Lambda^{3}_{2}\subset\Delta^{3}$.
Likewise, the right-hand $3$-simplex is obtained by filling the inner
horn $\Lambda^{3}_{1}\subset\Delta^{3}$. These diagrams lifts the
maps $g,i,$ and $\Phi$ to those in $\Cat_{\infty/\cat C}$. The
lifts give rise to a diagram in $\Cat_{\infty/\cat C}$ satisfying
conditions (I) through (IV). The proof is now complete.
\end{proof}

\section{\label{sec:Muro}Enriched Muro's theorem}

In this section, we show that every monoidal $\mathbf{V}$-model category
can be made into one with a cofibrant unit by changing the class of
cofibrations. The result is a minor generalization of \cite[Theorem 1]{Mur15},
which establishes the corresponding theorem in the unenriched setting.
\begin{prop}
\label{prop:Mur15}Let $\mathbf{V}$ be a combinatorial monoidal model
category with a cofibrant unit, and let $\mathbf{M}$ be a combinatorial
monoidal $\mathbf{V}$-model category. There is another combinatorial
monoidal $\mathbf{V}$-model structure on the underlying $\mathbf{V}$-monoidal
category of $\mathbf{M}$, denoted by $\widetilde{\mathbf{M}}$, satisfying
the following conditions:
\begin{enumerate}
\item The weak equivalences of $\widetilde{\mathbf{M}}$ are the weak equivalences
of $\mathbf{M}$.
\item Every cofibration of $\mathbf{M}$ is a cofibration of $\widetilde{\mathbf{M}}$.
\item The unit object is cofibrant in $\widetilde{\mathbf{M}}$.
\end{enumerate}
Moreover, it satisfies the following universal property:

\begin{enumerate}[resume]

\item Let $\mathbf{N}$ be another monoidal $\mathbf{V}$-model category
with a cofibrant unit. Then every monoidal left Quillen $\mathbf{V}$-functor
$\mathbf{M}\to\mathbf{N}$ extends uniquely to a left Quillen $\mathbf{V}$-functor
$\widetilde{\mathbf{M}}\to\mathbf{N}$.

\end{enumerate}

An analogous claim holds for tractable semi-model categories.
\end{prop}

\begin{proof}
We will follow \cite{Mur15}. We will use the following notation:
\begin{itemize}
\item Given a morphism $f:A\to A'$ in $\mathbf{V}$ or $\mathbf{M}$ and
a morphism $g:B\to B'$ in $\mathbf{V}$ or $\mathbf{M}$, we write
$f\square g$ for their \textbf{pushout product}
\[
A\otimes B'\amalg_{A\otimes B}A'\otimes B\to A'\otimes B'
\]
whenever the notation makes sense (i.e., either (i) both $f$ and
$g$ are maps of $\mathbf{V}$, (ii) $f$ is a morphism of $\mathbf{V}$
and $g$ is a morphism of $\mathbf{M}$, or (iii) both $f$ and $g$
are morphisms of $\mathbf{M}$.) 
\item If $S$ and $T$ are collections of morphisms in $\mathbf{V}$ or
$\mathbf{M}$, we will write $S\square T=\{s\square t\mid s\in S,t\in T\}$
whenever the right-hand side makes sense. 
\item If $i:A\to B$ and $p:X\to Y$ are morphisms of $\mathbf{M}$, we
will write $\inp{i,p}$ for their \textbf{pullback power}
\[
\mathbf{M}\pr{B,X}\to\mathbf{M}\pr{A,X}\times_{\mathbf{M}\pr{A,Y}}\mathbf{M}\pr{B,Y},
\]
where we wrote $\mathbf{M}\pr{-,-}$ for the enriched hom-objects
of $\mathbf{M}$. 
\item The symbol $\mathbf{1}$ will always mean the unit of $\mathbf{M}$;
when we need to refer to the unit of $\mathbf{V}$ (which happens
only once in the proof), we denote it by $\mathbf{1}_{\mathbf{V}}$.
\end{itemize}

Choose a weak equivalence $q:\widetilde{\mathbf{1}}\to\mathbf{1}$
that witnesses Muro's unit axiom in $\mathbf{M}$. Choose a factorization
\[
\widetilde{\mathbf{1}}\amalg\mathbf{1}\xrightarrow{\iota}C\xrightarrow{\pi}\mathbf{1}
\]
of the map $\pr{q,\id}$, where $\iota$ is a cofibration in $\mathbf{M}$
and $\pi$ is a weak equivalence in $\mathbf{M}$. Let $\opn{inc}_{1}:\widetilde{\mathbf{1}}\to\widetilde{\mathbf{1}}\amalg\mathbf{1}$
denote the inclusion of the first summand, and set $\kappa=\iota\circ\opn{inc}_{1}:\widetilde{\mathbf{1}}\to C$.
Also, choose generating sets $I_{\mathbf{M}}$ and $J_{\mathbf{M}}$
of cofibrations and trivial cofibrations of $\mathbf{M}$, and choose
generating sets $I_{\mathbf{V}}$ and $J_{\mathbf{V}}$ of cofibrations
and trivial cofibrations of $\mathbf{V}$. By enlarging $I_{\mathbf{V}}$
if necessary, we may assume that the map $\emptyset\to\mathbf{1}_{\mathbf{V}}$
belongs to $I_{\mathbf{V}}$. Let $\phi:\emptyset\to\mathbf{1}$ denote
the unique map from the initial object. We set
\begin{align*}
I_{\widetilde{\mathbf{M}}} & =I_{\mathbf{M}}\cup\pr{I_{\mathbf{V}}\square\{\phi\}},\\
J_{\widetilde{\mathbf{M}}} & =J_{\mathbf{M}}\cup\pr{I_{\mathbf{V}}\square\{\kappa\}}\cup\pr{J_{\mathbf{V}}\square\{\phi\}}.
\end{align*}
We claim that there is a cofibrantly generated model structure on
$\mathbf{M}$ with generating sets of cofibrations and trivial cofibrations
given by $I_{\widetilde{\mathbf{M}}}$ and $J_{\widetilde{\mathbf{M}}}$,
and with weak equivalences given by those of $\mathbf{M}$. Write
$I_{\widetilde{\mathbf{M}}}\-\cof$ and $I_{\widetilde{\mathbf{M}}}\-\inj$
for the classes of $I_{\widetilde{\mathbf{M}}}$-cofibrations and
$I_{\widetilde{\mathbf{M}}}$-injectives; we also make use of similar
notations, such as $J_{\mathbf{V}}\-\cof$ and $J_{\widetilde{\mathbf{M}}}\-\inj$,
and let $W$ denote the class of weak equivalences of $\mathbf{M}$.
According to \cite[Theorem 11.3.1]{Hirschhorn}, we must verify the
following (In the case of semi-model categories, we use Proposition
\ref{prop:cofibrantgeneration}):

\begin{enumerate}[label=(\alph*)]

\item $J_{\widetilde{\mathbf{M}}}\-\cof\subset\pr{I_{\widetilde{\mathbf{M}}}\-\cof}\cap W$.
(In the case of semi-model categories, we replace $J_{\widetilde{\mathbf{M}}}\-\cof$
by its subclass of maps with cofibrant domain.)

\item $I_{\widetilde{\mathbf{M}}}\-\inj\subset\pr{J_{\widetilde{\mathbf{M}}}\-\inj}\cap W$.

\item $\pr{J_{\widetilde{\mathbf{M}}}\-\inj}\cap W\subset I_{\widetilde{\mathbf{M}}}\-\inj$.

\end{enumerate}

Before we prove (a) through (c), we observe that the associativity
(up to natural isomorphism) of pushout-product implies the following
inclusions: 
\begin{itemize}
\item [(d)]$\pr{I_{\mathbf{V}}\-\cof}\square\pr{I_{\widetilde{\mathbf{M}}}\-\cof}\subset\pr{I_{\widetilde{\mathbf{M}}}\-\cof}$.
\item [(e)]$\pr{\pr{J_{\mathbf{V}}\-\cof}\square\pr{I_{\widetilde{\mathbf{M}}}\-\cof}}\subset J_{\widetilde{\mathbf{M}}}\-\cof$.
\item [(f)]$\pr{\pr{I_{\mathbf{V}}\-\cof}\square\pr{J_{\widetilde{\mathbf{M}}}\-\cof}}\subset J_{\widetilde{\mathbf{M}}}\-\cof$.
\end{itemize}
Note that (d) implies:
\begin{itemize}
\item [(d$'$)]A map $p$ of $\mathbf{M}$ is an $I_{\widetilde{\mathbf{M}}}$-injective
if and only if, for each $I_{\widetilde{\mathbf{M}}}$-cofibration
$i$, the map $\inp{i,p}$ is a trivial fibration of $\mathbf{V}$.
\end{itemize}
(There is a similar characterization of $J_{\widetilde{\mathbf{M}}}$-injectives,
but we will not need this.)

We also observe that:
\begin{itemize}
\item [($\ast$)]The functor $-\otimes\widetilde{\mathbf{1}}:\mathbf{M}\to\mathbf{M}$
preserves and detects weak equivalences. 
\item [($\ast\ast$)]The morphism $\kappa\otimes\widetilde{\mathbf{1}}$
is a $J_{\mathbf{M}}$-cofibration.
\end{itemize}
Indeed, assertion ($\ast$) is clear, because the functor $-\otimes\widetilde{\mathbf{1}}\from\mathbf{M}\to\mathbf{M}$
is naturally weakly equivalent to the identity functor. Assertion
($\ast\ast$) is a consequence of the fact that $\kappa\otimes\widetilde{\mathbf{1}}$
is an $I_{\mathbf{M}}$-cofibration (being the composite of a pushout
of $\emptyset\to\widetilde{\mathbf{1}}$ and the map $\iota\otimes\widetilde{\mathbf{1}}$)
and is a weak equivalence by ($\ast$).

Now we prove (a) through (c). We start from (a). To show that $J_{\widetilde{\mathbf{M}}}\-\cof\subset I_{\widetilde{\mathbf{M}}}\-\cof$,
it will suffice to show that $J_{\widetilde{\mathbf{M}}}\subset I_{\widetilde{\mathbf{M}}}\-\cof$.
For this, in light of (d), it suffices to show that $\kappa$ is an
$I_{\widetilde{\mathbf{M}}}$-cofibration. This is clear, because
it is the composite of the pushout of $\phi$ and the $I_{\mathbf{M}}$-cofibration
$\iota$. To show that $J_{\widetilde{\mathbf{M}}}\-\cof\subset W$,
we observe that $\pr{J_{\widetilde{\mathbf{M}}}\-\cof}\otimes\widetilde{\mathbf{1}}\subset J_{\mathbf{M}}\-\cof$
by ($\ast\ast$). It follows from ($\ast$) that $J_{\widetilde{\mathbf{M}}}\-\cof\subset W$,
as desired.

Next, we prove (b). Let $p$ be an $I_{\widetilde{\mathbf{M}}}$-injective.
We must show that $p\in\pr{J_{\widetilde{\mathbf{M}}}\-\inj}\cap W$.
Since $I_{\widetilde{\mathbf{M}}}\-\inj\subset I_{\mathbf{M}}\-\inj=\pr{J_{\mathbf{M}}\-\inj}\cap W$,
it suffices to show that $\inp{\kappa,p}$ is a trivial fibration
and $\inp{\phi,p}$ is a fibration. By (d$'$), this will follow once
we show that $\kappa$ and $\phi$ are $I_{\widetilde{\mathbf{M}}}$-cofibrations.
For $\phi$, this is clear. We have also seen in the previous paragraph
that $\kappa$ is an $I_{\widetilde{\mathbf{M}}}$-cofibration. Thus
we have proved (b).

Finally, we prove (c). Let $p\in\pr{J_{\widetilde{\mathbf{M}}}\-\inj}\cap W$.
We wish to show that $p\in I_{\widetilde{\mathbf{M}}}\-\inj$. Since
$\pr{J_{\widetilde{\mathbf{M}}}\-\inj}\cap W\subset\pr{J_{\mathbf{M}}\-\inj}\cap W=I_{\mathbf{M}}\-\inj$,
it suffices to show that $\inp{\phi,p}$ is a trivial fibration. Since
$\phi$ is a retract of the map $\psi:\emptyset\to C$, it suffices
to show that $\inp{\psi,p}$ is a trivial fibration. For this, factor
the map $\psi$ as
\[
\emptyset\xrightarrow{\widetilde{\phi}}\widetilde{\mathbf{1}}\xrightarrow{\kappa}C.
\]
The map $\widetilde{\phi}$ is an $I_{\mathbf{M}}$-cofibration, so
$\inp{\widetilde{\phi},p}$ is a trivial fibration. The map $\inp{\kappa,p}$
is a trivial fibration by the definition of $J_{\widetilde{\mathbf{M}}}$-injectives.
Since $\inp{\psi,p}$ is the composition of $\inp{\kappa,p}$ and
a pullback of $\inp{\widetilde{\phi},p}$, this proves that $\inp{\psi,p}$
is a trivial fibration, as desired.

We have thus shown that the sets $I_{\widetilde{\mathbf{M}}}$, $J_{\widetilde{\mathbf{M}}}$
and the class $W$ of weak equivalences of $\mathbf{M}$ gives rise
to a model structure on $\mathbf{M}$, which we denote by $\widetilde{\mathbf{M}}$.
By construction, it satisfies conditions (1) through (3). We complete
the proof by showing that $\widetilde{\mathbf{M}}$ is a monoidal
model category; once we prove this, condition (4) is immediate from
the construction.

To show that $\widetilde{\mathbf{M}}$ is a monoidal model category,
it suffices to show that $I_{\widetilde{\mathbf{M}}}\square I_{\widetilde{\mathbf{M}}}\subset I_{\widetilde{\mathbf{M}}}\-\cof$,
$I_{\widetilde{\mathbf{M}}}\square J_{\widetilde{\mathbf{M}}}\subset J_{\widetilde{\mathbf{M}}}\-\cof$,
and $J_{\widetilde{\mathbf{M}}}\square I_{\widetilde{\mathbf{M}}}\subset J_{\widetilde{\mathbf{M}}}\-\cof$.
Using the fact that $\mathbf{M}$ is a monoidal $\mathbf{V}$-model
category, and also using inclusions (a), (d), and (e), we are reduced
to showing the following:
\begin{itemize}
\item [($\ast\ast\ast$)]For every $I_{\mathbf{M}}$-cofibration $i$, the
maps $i\square\kappa$ and $\kappa\square i$ are $J_{\mathbf{M}}$-cofibrations
(and hence $J_{\widetilde{\mathbf{M}}}$-cofibrations). (In the semi-model
categorical case, we also assume $i$ has a cofibrant domain in $\mathbf{M}$.)
\end{itemize}

To prove ($\ast\ast\ast$), we observe that $i\square\kappa$ is a
composition of $i\square\iota$ and a pushout of $i\otimes\phi\cong i$.
Since the maps $i\square\iota$ and $i$ are $I_{\mathbf{M}}$-cofibrations,
we deduce that $i\square\kappa$ is an $I_{\mathbf{M}}$-cofibration.
Moreover, we know from ($\ast\ast$) that $\kappa\otimes\widetilde{\mathbf{1}}$
is an $J_{M}$-cofibration, so $\pr{i\square\kappa}\otimes\widetilde{\mathbf{1}}$
is also a $J_{\mathbf{M}}$-cofibration. It follows from ($\ast$)
that $i\square\kappa$ is a weak equivalence. Hence $i\square\kappa$
is a $J_{\mathbf{M}}$-cofibration. Likewise, $\kappa\square i$ is
a $J_{\mathbf{M}}$-cofibration. The proof is now complete. 
\end{proof}

\subsection*{Acknowledgment}

The author thanks David White for patiently answering many questions
about (semi-)model categories and monoidal categories. He also thanks
Chris Kapulkin and Daniel Carranza for helpful discussions on cubical
sets, without which this paper would not have been possible. He is
also grateful to Chris for general advice on writing. This work was
supported by JSPS KAKENHI Grant Number 24KJ1443.

\input{main.bbl}

\end{document}

%% file: macros.tex
\global\long\def\u#1{\underline{#1}}%

\global\long\def\cat#1{\mathcal{#1}}%

\global\long\def\tild#1{\widetilde{#1}}%

\global\long\def\mrm#1{\mathrm{#1}}%

\global\long\def\pr#1{\left(#1\right)}%

\global\long\def\abs#1{\left|#1\right|}%

\global\long\def\inp#1{\left\langle #1\right\rangle }%

\global\long\def\br#1{\left\{  #1\right\}  }%

\global\long\def\norm#1{\left\Vert #1\right\Vert }%

\global\long\def\hat#1{\widehat{#1}}%

\global\long\def\opn#1{\operatorname{#1}}%

\global\long\def\bigmid{\,\middle|\,}%

\global\long\def\Top{\mathsf{Top}}%

\global\long\def\Set{\mathsf{Set}}%

\global\long\def\SS{\mathsf{sSet}}%

\global\long\def\Kan{\mathsf{Kan}}%

\global\long\def\FB{\mathsf{FB}}%

\global\long\def\Fin{\mathsf{Fin}}%

\global\long\def\urCard{\mathsf{urCard}}%

\global\long\def\LPC{\mathsf{LocPresCat}}%

\global\long\def\RelCat{\mathsf{RelCat}}%

\global\long\def\Rel{\mathcal{R}\mathsf{el}}%

\global\long\def\PermRelCat{\mathsf{PermRelCat}}%

\global\long\def\RelCatlarge{\mathsf{Rel}\hat{\mathsf{Cat}}}%

\global\long\def\SMRelCat{\mathsf{SMRelCat}}%

\global\long\def\SMRelCatlarge{\mathsf{SMRel}\widehat{\mathsf{Cat}}}%

\global\long\def\TMMC{\mathsf{TractMMC}}%

\global\long\def\CMMC{\mathsf{CombMMC}}%

\global\long\def\CSMMC{\mathsf{CombSMMC}}%

\global\long\def\TSMMC{\mathsf{TractSMMC}}%

\global\long\def\CombSymAlg#1{\mathsf{CombSym}#1\text{-}\mathsf{Alg}}%

\global\long\def\combsymalg#1{\mathsf{combsym}#1\text{-}\mathsf{alg}}%

\global\long\def\CombCentAlg#1{\mathsf{CombCent}#1\text{-}\mathsf{Alg}}%

\global\long\def\BiCat{\mathcal{B}\mathsf{i}\mathcal{C}\mathsf{at}}%

\global\long\def\Cat{\mathcal{C}\mathsf{at}}%

\global\long\def\SM{\mathcal{SM}}%

\global\long\def\Mon{\mathcal{M}\mathsf{on}}%

\global\long\def\LFib{\mathcal{L}\mathsf{Fib}}%

\global\long\def\Cart{\mathcal{C}\mathsf{art}}%

\global\long\def\Pr{\mathcal{P}\mathsf{r}}%

\global\long\def\Del{\mathbf{\Delta}}%

\global\long\def\Sig{\mathbf{\Sigma}}%

\global\long\def\id{\operatorname{id}}%

\global\long\def\Aut{\operatorname{Aut}}%

\global\long\def\End{\operatorname{End}}%

\global\long\def\Hom{\operatorname{Hom}}%

\global\long\def\Env{\operatorname{Env}}%

\global\long\def\Map{\operatorname{Map}}%

\global\long\def\Und{\operatorname{Und}}%

\global\long\def\Ar{\operatorname{Ar}}%

\global\long\def\cEnd{\mathcal{E}\mathrm{nd}}%

\global\long\def\ho{\operatorname{ho}}%

\global\long\def\ob{\operatorname{ob}}%

\global\long\def\-{\text{-}}%

\global\long\def\rep{\mathrm{rep}}%

\global\long\def\op{\mathrm{op}}%

\global\long\def\bi{\mathrm{bi}}%

\global\long\def\sk{\mathrm{sk}}%

\global\long\def\cc{\mathrm{cc}}%

\global\long\def\gr{\mathrm{gr}}%

\global\long\def\act{\mathrm{act}}%

\global\long\def\inj{\mathrm{inj}}%

\global\long\def\cof{\mathrm{cof}}%

\global\long\def\fib{\mathrm{fib}}%

\global\long\def\BK{\mathrm{BK}}%

\global\long\def\Reedy{\mathrm{Reedy}}%

\global\long\def\loc{\mathrm{loc}}%

\global\long\def\Seg{\mathrm{Seg}}%

\global\long\def\semi{\mathrm{semi}}%

\global\long\def\weq{\mathrm{weq}}%

\global\long\def\opd{\mathrm{opd}}%

\global\long\def\idem{\mathrm{idem}}%

\global\long\def\rel{\mathrm{rel}}%

\global\long\def\Quilleq{\mathrm{Quill.eq}}%

\global\long\def\loceq{\mathrm{loc.eq}}%

\global\long\def\To{\Rightarrow}%

\global\long\def\rr{\rightrightarrows}%

\global\long\def\rl{\rightleftarrows}%

\global\long\def\mono{\rightarrowtail}%

\global\long\def\epi{\twoheadrightarrow}%

\global\long\def\comma{\downarrow}%

\global\long\def\ot{\leftarrow}%

\global\long\def\from{\colon}%

\global\long\def\corr{\leftrightsquigarrow}%

\global\long\def\lim{\operatorname{lim}}%

\global\long\def\colim{\operatorname{colim}}%

\global\long\def\holim{\operatorname{holim}}%

\global\long\def\hocolim{\operatorname{hocolim}}%

\global\long\def\Ran{\operatorname{Ran}}%

\global\long\def\Lan{\operatorname{Lan}}%

\global\long\def\Ind{\operatorname{Ind}}%

\global\long\def\Fun{\operatorname{Fun}}%

\global\long\def\Fact{\operatorname{Fact}}%

\global\long\def\Perm{\operatorname{Perm}}%

\global\long\def\Alg{\operatorname{Alg}}%

\global\long\def\Coll{\operatorname{Coll}}%

\global\long\def\MInd{\mathbf{Ind}}%

\global\long\def\SSeq{\Sigma\mathrm{Seq}}%

\global\long\def\MSSeq{\mathbf{\Sigma Seq}}%

\global\long\def\B{\mathrm{B}}%

\global\long\def\H{\mathrm{H}}%

\global\long\def\Day{\mathrm{Day}}%

\global\long\def\xmono#1#2{\stackrel[#2]{#1}{\rightarrowtail}}%

\global\long\def\xepi#1#2{\stackrel[#2]{#1}{\twoheadrightarrow}}%

\global\long\def\adj{\stackrel[\longleftarrow]{\longrightarrow}{\bot}}%

\global\long\def\w{\wedge}%

\global\long\def\t{\otimes}%

\global\long\def\ev{\operatorname{ev}}%

\global\long\def\bp{\boxplus}%

\global\long\def\rcone{\triangleright}%

\global\long\def\lcone{\triangleleft}%

\global\long\def\teq{\trianglelefteq}%

\global\long\def\ll{\vartriangleleft}%

\global\long\def\S{\mathsection}%

\global\long\def\p{\prime}%

\global\long\def\pp{\prime\prime}%

\global\long\def\sto{\rightsquigarrow}%

%% file: main.bbl
\providecommand{\bysame}{\leavevmode\hbox to3em{\hrulefill}\thinspace}
\providecommand{\MR}{\relax\ifhmode\unskip\space\fi MR }
\providecommand{\MRhref}[2]{%
  \href{http://www.ams.org/mathscinet-getitem?mr=#1}{#2}
}
\providecommand{\href}[2]{#2}